\documentclass[letterpaper,11pt,oneside,reqno]{amsart}
\usepackage[T1]{fontenc}

\usepackage{amsfonts,amsmath, amssymb,amsthm,amscd}
\usepackage{bm}
\usepackage{mathrsfs}
\usepackage{yfonts}
\usepackage{verbatim}
\usepackage{hyperref}
\usepackage{graphicx}
\usepackage[utf8]{inputenc}
\usepackage{latexsym}
\usepackage{upgreek}
\usepackage{relsize}
\usepackage{xcolor}
\usepackage{mathdots}
\usepackage{mathtools}
\usepackage{dsfont}
\usepackage[left=3cm, right=3cm, top=3cm, bottom=3cm]{geometry}
\usepackage{tikz}
\usepackage{lmodern}
\usepackage{stmaryrd}
\usepackage{lipsum}
\usepackage{subcaption}
\usepackage{amsmath}
\usepackage{tikz-cd}
\usepackage{subcaption}
\usetikzlibrary{patterns}

\numberwithin{equation}{section}
\renewcommand{\emph}[1]{\textsf{\textit{#1}}}
\usepackage[width=.9\textwidth]{caption}

\makeatletter
\renewcommand\section{\@startsection{section}{1}%
	\z@{1.5\linespacing\@plus.5\linespacing}{1\linespacing}%
	{\large\normalfont\scshape\centering}}
\makeatother

\def \Gav {\Ga^{-1}}
\def \be{\begin{equation}}
\def \ee{\end{equation}}

\newcommand{\EE}{\ensuremath{\mathbb{E}}}

\newcommand{\PP}{\ensuremath{\mathbb{P}}}
\newcommand{\R}{\ensuremath{\mathbb{R}}}
\newcommand{\C}{\ensuremath{\mathbb{C}}}

\newcommand{\Res}{\operatorname{Res}}

\newcommand{\la}{\lambda}
\newcommand{\ga}{\gamma}
\newcommand{\ta}{\theta}

\renewcommand{\leq}{\leqslant}
\renewcommand{\geq}{\geqslant}

\newcommand{\Geo}{\mathrm{Geo}}
\newcommand{\Gammainv}{\mathrm{Gamma}^{-1}}

\DeclareMathOperator{\Schur}{Schur}
\newcommand{\Whit}{\mathrm{Whit}}

\DeclareMathOperator{\Ga}{Gamma}
\DeclareMathOperator{\Bt}{Beta}
\DeclareMathOperator{\KPZ}{KPZ}
\DeclareMathOperator{\BM}{BM}
\DeclareMathOperator{\Sign}{Sign}
\DeclareMathOperator{\stat}{stat}

\DeclarePairedDelimiter\abs{\lvert}{\rvert} 

\def \d {\mathsf{d}}
\def \a {\alpha}

\def \la{\lambda}
\def \bl {\boldsymbol\lambda}

\newcommand{\id}{\mathrm{id}}
\def \RR {\mathbb{R}}
\def \ZZ{\mathbb{Z}}

\def \NN{\mathbb{N}}

\DeclareMathOperator{\LG}{LG}
\DeclareMathOperator{\Law}{Law}
\def \bL {\mathbf{L}}

\def \PP {\mathbb{P}}
\def \P{\mathsf{P}}
\def \U{\mathsf{U}}

\newcommand{\dd}{\,\mathrm d}

\newtheorem{theorem}{Theorem}[section]

\newtheorem{lemma}[theorem]{Lemma}
\newtheorem{proposition}[theorem]{Proposition}
\newtheorem{corollary}[theorem]{Corollary}

\theoremstyle{definition}
\newtheorem{remark}[theorem]{Remark}

\theoremstyle{definition}

\theoremstyle{definition}
\newtheorem{definition}[theorem]{Definition}

\theoremstyle{definition}

\title[Exponential functionals of  random walks]{Exponential functionals of two random walks and stationary measures of discrete growth processes}

\begin{document}

\author[Z.~Ouyang]{Zikun Ouyang}
\address{Z.~Ouyang, CEREMADE, CNRS, Université Paris-Dauphine, Université PSL, 75016 PARIS, FRANCE and Laboratoire de Physique de l'Ecole Normale Supérieure, Ecole Normale Supérieure, PSL University, CNRS, Sorbonne Université, Université Paris-Cité, 24 rue Lhomond, 75005 PARIS}
\email{ouyang@ceremade.dauphine.fr}

\begin{abstract} 
We study random walks with geometric or log-gamma increments, reweighted by certain exponential functionals. We establish precise asymptotics for all moments of exponential functionals, revealing a phase transition as the drifts of the walks vary. To describe the large time limit of reweighted random walks, we introduce a discrete $T$ transformation on random walk paths  and establish Girsanov relations. As a corollary, we show that the stationary measures of geometric last passage percolation and the log-gamma polymer on a strip converge to their half-space counterparts as the strip width tends to infinity. In the log-gamma case, this provides an alternative proof of a recent result in \cite{BrycWesolowskiLG}. 
\end{abstract}

\maketitle
\section{Introduction and main results}
\subsection{Preface}
The characterization of the stationary measures for models in the Kardar--Parisi--Zhang (KPZ) universality class in different geometries has a long history.  
%
Recent years have witnessed breakthroughs for integrable models in finite domain and half-space \cite{Corwin}. Stationary measures  of the open KPZ equation on an interval are first studied in  \cite{CorwinKnizel} and can be   characterized by their multi-point Laplace transform in the fan region, based on earlier works on the stationary measures of open ASEP  \cite{DerridaEvansHakimPasquier,UchiyamaSasamotoWadati,BrycWesolowski,BrycWesolowskiASEP}.  
 
\bigskip 
The result of \cite{CorwinKnizel} was later reformulated in terms of exponential functionals of two Brownian motions \cite{BrycKuznetsovWangWesolowski,BarraquandLeDoussal}. 
In particular, using the idea from physics of integrating out the zero mode,  the formulation in \cite{BarraquandLeDoussal} predicts candidate stationary measures for the open KPZ equation for all boundary parameters.  The prediction was made rigorous at the discrete level through integrable polymers: 
remarkably, for all choices of boundary parameters, the stationary measures of the geometric last passage percolation (GLPP) and the log-gamma polymer on a strip can be described in terms of a pair of random walks reweighted by certain exponential functionals \cite{BarraquandCorwinYang}. 
For the open KPZ equation, a direct proof of stationarity for these measures  was later obtained in \cite{DunlapGuRosati}   based on stochastic analysis.   

\bigskip 
Exponential functionals of Brownian motion are ubiquitous in
probability theory, see  \cite{ComtetMonthusYor,Yor,MatsumotoYorI,MatsumotoYorII,YorPassion} and references therein. 
In this paper, we focus on their
connections to the stationary measures of integrable models in
the KPZ universality class.

\bigskip 
The first observation in this direction is in \cite{BarraquandLeDoussal}, where  the authors notice that the stationary measure of the open KPZ equation is closely related to a family of measures studied in \cite{HariyaYor} which arises in a  Matsumoto--Yor  relation between exponential functionals of Brownian motions with different drifts   \cite{MatsumotoYor_Dufresne}.  
In particular, \cite{HariyaYor} establishes the existence of large time limits of these measures. It turns out that a phase transition emerges, and that limiting measures are given by the $T$ transformed Brownian motions introduced in \cite{DonatiMartinMatsumotoYor}.  
In the  KPZ context, 
this limit corresponds to sending the size of the interval to infinity.

\bigskip 
Based on results in \cite{HariyaYor}, \cite{BarraquandLeDoussal} conjectured stationary measures for the KPZ equation in half-space. The prediction was confirmed by considering the intermediate disorder limit for discrete polymer models. The stationary measures of the log-gamma polymer in half-space are constructed in \cite{BarraquandCorwin} using symmetry properties of the half-space Whittaker process \cite{OConnellSeppalainenZygouras,BarraquandBorodinCorwin}. 
Remarkably, the stationary measures admit a two-row polymer representation. This structure  is  closely related to the celebrated geometric Pitman's transform  and  has appeared in the study of Brownian queues and semi-discrete polymers \cite{OConnellYor_BrownianBurke,OConnellYor}.   

\bigskip 
Motivated by the extensive study of exponential functionals of Brownian motion and their connection to the stationary measures of the KPZ equation, the goal of this paper is to better understand their discrete counterparts at both zero and positive temperature, based on fruitful developments in the study of stationary measures of discrete growth processes in the KPZ universality class. 
In particular,  we study the limiting processes associated with exponential functionals of a pair of geometric or log-gamma random walk. To describe the limit, we introduce a  discrete  $T$ transform, which is a slight variant of the discrete (geometric) Pitman's transform and satisfies natural algebraic properties  analogous to those of its continuous counterpart \cite{DonatiMartinMatsumotoYor}.  As an application, we prove that as the strip width tends to infinity, the stationary measures of GLPP and the log-gamma polymer on a strip converge to their appropriate half-space counterparts depending on the boundary parameters.  For the log-gamma polymer, this convergence was recently established in \cite{BrycWesolowskiLG} using different methods. We discuss the connections  between the two in Section \ref{sec:HY for LG}.

\bigskip 
The rest of the introduction is structured as follows. In Section \ref{subsubsec:The Hariya Yor theorem}, we review the Hariya--Yor theorems in the context of  
the open KPZ equation.  In Section \ref{subsuc:GLPP} and Section \ref{subsec:Log gamma}, we formulate the stationary measures of GLPP and  log-gamma polymer  
in terms of exponentially reweighted or $T$ transformed random walks  and present our main results, which can be seen as discrete versions of the Hariya--Yor theorem at zero and positive temperature respectively. In Section \ref{subsec:Relations to the literature}, we explain the connection of our results to the literature, with a focus on a two-layer point of view and the different convergences between various aforementioned stationary measures. In Section \ref{subsec:outline and proof ideas}, we sketch the proof idea and  outline  the rest of the paper.

\subsection{The Hariya--Yor theorem}\label{subsubsec:The Hariya Yor theorem}

For $L>0$, the open KPZ equation on the interval $[0,L]$ (respectively, the half-space KPZ equation on $\RR_{\geq 0}$) with Neumann boundary conditions formally writes 
\be\label{eq:open KPZ eq}
\begin{cases} 
 \partial_t h=\frac12\partial_{xx}h +\frac12(\partial_xh)^2+\xi,
  \qquad
 \text{ for } x\in (0,L)  \quad (\text{resp. }x\in\RR_{>0}),\\
 \partial_xh(t,0)=u,\\
 \partial_xh(t,L)=-v \quad
   (\text{resp. } \varnothing),
\end{cases}
\ee
where $\xi=\xi(x,t)$ is the space time white noise. The precise meaning of its solution via the  Hopf-Cole transform is established in \cite{CorwinShen}  (we adapt the parameter conventions in \cite{CorwinKnizel} here). 
The law of a random continuous  function $h_0:[0,L]\ (\text{resp. } \RR_{\geq 0})\to \RR$ with $h_0(0)=0$ is said to be a stationary if,  for \eqref{eq:open KPZ eq} with initial data  $h(0,\cdot)\overset{d}{=}h_0$, it follows that $h(t,\cdot)-h(t,0)\overset{d}{=}h_0$ for all $t\geq0$.  

Let $\mathcal{C}=C(\RR_{\geq 0},\RR)$ be the space of real-valued continuous functions over $\RR_{\geq 0}$, endowed with the compact uniform topology. For a  path
$X \in \mathcal{C}$ and $x,z\geq 0$,  the  exponential functional associated to $X$ and the $T$ transformed path \cite{DonatiMartinMatsumotoYor} are defined respectively by 
\be\label{eq:one layer exponential functional and T transform}
 A_x(X)=\int_0^x e^{2X(s)}\dd s, \quad   
  T^z(X)(x)=X(x)-\log\left(1+zA_x(X)\right).
\ee
They enter the stationary measures of \eqref{eq:open KPZ eq} as follows. Throughout 
$\PP_{\BM}^b$ denotes the law of a Brownian motion starting from zero, with drift $b$ and diffusion coefficient $\frac{1}{2}$, and we write $\PP_{\BM,L}^b$ for its restriction to $[0,L]$. 

\begin{proposition}\cite{CorwinKnizel,BarraquandLeDoussal,DunlapGuRosati}
\label{prop:Stationary measure open KPZ}
For all $u,v\in\RR$ and $L>0$, let $X_L^{u,v} $ be the process on $C( [0,L],\RR)$ whose distribution   $\mathsf P_L^{u,v}$  is defined through its Radon-Nikodym derivative with respect to $\PP_{\BM,L}^v$ by 
\begin{equation}\label{eq:reweighted brownian motion}
 \frac{\dd\mathsf P_L^{u,v}}{\dd\PP_{\BM,L}^{v}}(X)
 =\frac{1}{\mathsf Z_L^{u,v}}A_L(X)^{-(u+v)}, 
\end{equation}
where  $ \mathsf Z_L^{u,v}  $ is the normalization constant and is finite   \cite{HariyaYor}. Then  the process 
\begin{equation}\label{eq:stationary measure of open kpz}
 h_L^{u,v}=W_L-X_L^{u,v},  
\end{equation}
where $W_L\sim\PP_{\BM,L}^0$ is  independent of $X_L^{u,v} $, 
 is stationary for the open KPZ equation \eqref{eq:open KPZ eq}.  
\end{proposition}
For $\ta>0$, we  denote the probability distribution on $\RR_{\geq0}$ with density 
  $\frac{1}{\Gamma(\ta)}x^{-\ta-1}e^{-x^{-1}}\d x$ by   $\Gav(\ta)$.
Throughout  by writing $Y\sim \Gav(0)$, we mean   that 
\be\label{eq:convention inverse gamma}
 \frac{1}{Y}=0  \quad  a.s.
\ee

\begin{proposition}\cite[Theorem 1.4]{BarraquandCorwin}  
\label{prop:Stationary measure half space KPZ}
For $u\in\RR$ and $v\leq\min\{0,u\}$, define a process $X_\infty^{u,v}$ by 
\begin{equation}
 X_\infty^{u,v}=  T^{\frac{1}{Y}}(B), 
\end{equation}
where $B\sim \PP_{\BM}^{-v}$ and $Y\sim\Gav(u- v)$ are independent. Then the process 
\begin{equation}\label{eq:stationary measure of half space kpz}
 h_\infty^{u,v}=W-X_\infty^{u,v},
\end{equation}
where $W\sim\PP_{\BM}^0$ is independent of $X_\infty^{u,v} $, 
 is stationary for the   KPZ equation in half-space \eqref{eq:open KPZ eq}.
\end{proposition} 

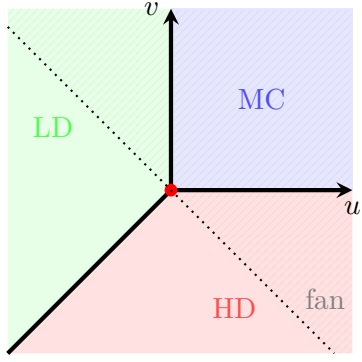
\begin{figure}
    \centering 
     \begin{tikzpicture}[scale=1.2]
     \fill[blue,opacity=0.1] (0,0) rectangle (2,2);
\fill[green,opacity=0.1]
    (-1.8,-1.8)--(-1.8,2)--(0,2)--(0,0)--cycle;
\fill[red,opacity=0.12]
    (-1.8,-1.8)--(0,0)--(2,0)--(2,-1.8)--cycle;
\draw[ultra thick, -stealth] (0,0) -- (2,0) node[anchor=north]{$u$};
\draw[ultra thick, -stealth] (0,0) -- (0,2) node[anchor=east]{$v$};
\draw[ultra thick] (0,0) --  (-1.8,-1.8) ;
\node[blue!70] at (1,1) {MC};
\node[green!70] at (-1.3,0.7) {LD};
\node[red!70] at (0.7,-1.3) {HD};
\fill[red] (0,0) circle (2pt);
\draw[dotted,thick] (-1.8,1.8)--(1.8,-1.8);
\node[gray]  at (1.7,-1.2){fan} ; 
\fill[pattern=north east lines,pattern color=black,opacity=0.2]
    (-1.8,1.8)--(-1.8,2)--(2,2)
    --(2,-1.8)--(1.8,-1.8)--cycle;
\end{tikzpicture}
    \caption{The phase diagram is separated by three lines: $\{u>v=0\}$, $\{u=v<0\}$ and $\{v>u=0\}$. The region $\{u+v>0\}$ is called the fan region.}
    \label{fig:phase diagram}
\end{figure}

A crucial observation in \cite{BarraquandLeDoussal} is that the process $X^{u,v}_L$ has been studied in a different context, where it appears in a Matsumoto--Yor relation between exponential functionals of Brownian motions with different drifts   \cite{MatsumotoYor_Dufresne}. In particular, 
\cite[Theorem 2.2]{HariyaYor} studies the full asymptotics of  moments of affine extensions of the exponential functional, which we refer to as the  partition function   
\be\label{eq:generalized partition function for KPZ}
\EE^{v}_{ \BM,L }
\left[\left(a+\xi A_L(B)\right)^{-(u+v)}\right] 
\ee
for $a\geq 0$ and $\xi>0$. 
We do not review the exact asymptotics here, but comment on the phase transitions they exhibit. The $(u,v)$ plane is divided into three phases (maximal current, low density and high density)  by three lines meeting at a triple point, see the phase diagram in Figure  \ref{fig:phase diagram}.  The asymptotics of the partition function vary  across these seven parts.  We give the exact asymptotics in the discrete setting, see Theorem \ref{thm:Geo main thm} and  Theorem \ref{thm:LG main thm}. 

\begin{remark}
When $a=0$ and $\xi=1$, the partition function \eqref{eq:generalized partition function for KPZ} is simply the moment of the integral of a geometric Brownian motion. In the maximal current phase $u>0,v>0$, the negative moments admit a remarkable integral representation \cite[Proposition 5]{DonatiMartinMatsumotoYor_Moments},
\be\label{eq:integral expression for moments of geometric Brownian motion}
\EE^{v}_{\BM,L}\left[\left(A_{L}(B) \right)^{-(u+v)} \right]= \frac{1}{\Gamma(u+v)}\int_{i\RR}\frac{\d z}{2\pi i} \frac{\Gamma(u\pm z)\Gamma(v\pm z)}{2\Gamma(\pm 2z)} e^{L(z^2-v^2)}, 
\ee 
with the parameter correspondence $\mu=2v$, $m=u+v$  and $t=L/2$,
see also \cite[Eq. (70)]{BarraquandLeDoussal}. Here and throughout, we use the shorthand notation
$
\Gamma(z_1\pm z_2):=\Gamma(z_1+z_2)\Gamma(z_1-z_2) 
$ for $z_1,z_2\in\C$.

The discrete analogues of \eqref{eq:integral expression for moments of geometric Brownian motion} for geometric and log-gamma random walks established in \cite{Barraquand} provide an important starting input for our analysis, see discussions in Section \ref{subsec:outline and proof ideas}. The formula \eqref{eq:integral expression for moments of geometric Brownian motion} is analytically continued to the fan region  $u+v>0$ in the works \cite{CorwinKnizel, BrycKuznetsovWangWesolowski,BrycKuznetsov} with consideration on the characterization of the stationary measures of the open KPZ equation \eqref{eq:stationary measure of open kpz}. 
We establish discrete analogues in Propositions \ref{prop:general formula for q_M^s for all phases} and \ref{prop:general formula for Q_M^s for all phases LG} and Corollary \ref{cor:analytic continuation for partition function LG}. 
Finally, the analytic continuation of   \eqref{eq:integral expression for moments of geometric Brownian motion} to all $u,v\in\RR$ is established in \cite[Eq. (4.16)]{Barraquand}. 
\end{remark}

The  partition function \eqref{eq:generalized partition function for KPZ} contains rich information of the process  $X_{L}^{u,v}$. In particular, 
\cite[Theorem 1.3]{HariyaYor} shows that the limit of the process $X_L^{u,v} $ coincides with the $T$ transformed Brownian motions, see also the reformulation in \cite[Section 5]{BarraquandLeDoussal}. Throughout we denote 
\be\label{eq:v star}
v_{\star}:=\min\{0,u,v\}.
\ee  

\begin{corollary}[\cite{HariyaYor}]
\label{cor:HY KPZ}
    As $L\to\infty$, there holds weak convergence in $\mathcal C$:
$$X_L^{u,v}\Longrightarrow X_{\infty}^{u,v_{\star}}.$$
\end{corollary}
In view of Propositions \ref{prop:Stationary measure open KPZ} and \ref{prop:Stationary measure half space KPZ}, Corollary \ref{cor:HY KPZ}   implies that the stationary measure of the open KPZ equation converges to the corresponding half-space stationary measure, according to the $(u,v)$ in the phase diagram (Figure \ref{fig:phase diagram}).  
The  limiting process $X_{\infty}^{u,v}$ has been studied in \cite{DonatiMartinMatsumotoYor},  where they show that the $T$ transformed process is absolutely continuous with respect to Brownian motions with an explicit Radon-Nikodym derivative that we define below.

\begin{definition}
    For $ v\leq \min\{u,0\}$ and $ a\geq0, \xi>0$, let us define a function 
    \be\label{eq:function N KPZ}
\mathscr N^{u,v}(  a,\xi) =\EE\left[\left(\frac{a}{Y} +\xi\right)^{-(u+v)}\right], 
\ee
where the expectation is taken over an inverse gamma random variable $Y\sim\Gammainv(u-v)$. In particular,  we have 
\be\label{eq:special case function N LG}
\mathscr N^{u,v}( 0,\xi)= \xi^{-(u+v)}, \qquad \mathscr N^{u,u}(  a,\xi) =\xi^{-2u},
\ee
in view of convention \eqref{eq:convention inverse gamma}.
\end{definition}
The function $\mathscr{N}$ can be expressed in terms of the Whittaker function \cite[Remark 3.2]{HariyaYor}. When $v\in(-u,0]$, it  also admits a useful expression as  an integral of the K-Bessel function, see Section \ref{subsubsec:whittaker preliminary}.

\begin{proposition}[{\cite[Theorem 1.5]{DonatiMartinMatsumotoYor}}]
\label{prop:Grisanov KPZ}
Let $u\in\RR, v\leq \min\{u,0\}$ and $x\geq0$. For every bounded measurable
functional $f$ on $C([0,x],\RR)$,
\begin{equation} \label{eq:KPZ Girsanov}
 \EE\Bigl[f\bigl(X_\infty^{u,v}|_{[0,x]}\bigr)\Bigr]
 =\EE^{v}_{ \BM} \Bigl[f\bigl(B|_{[0,x]}\bigr)
             \mathscr N_x^{u,v}(B)\Bigr], 
\end{equation}  
where for a path $X\in\mathcal{C}$ and $x\geq 0$, we denote its restriction to $[0,x]$  by $X |_{[0,x]}$ and  
 $$
\mathscr N_x^{u,v}(X)
 =\mathscr N^{u,v}\bigl(A_x(X),e^{2X(x)}\bigr). 
 $$
\end{proposition} 
 We establish discrete analogues of  Proposition \ref{prop:Grisanov KPZ}, see Propositions \ref{prop:Girsanov Geo} and   \ref{prop:Girsanov LG}.

\subsection{Stationary measures of geometric last passage percolation}
\label{subsuc:GLPP}
\subsubsection{The models} 
\label{subsubsec:GLPP models definition}
For $q\in(0,1]$, let $\Geo(q)$ denote the law of the random variable $Y$ such that 
$\PP(Y=j)=(1-q)q^j, j\geq 0$.  Throughout  by writing  $Y\sim\Geo(1)$ we mean that 
\be\label{eq:convention Geo}
Y=+\infty  \quad  a.s.
\ee

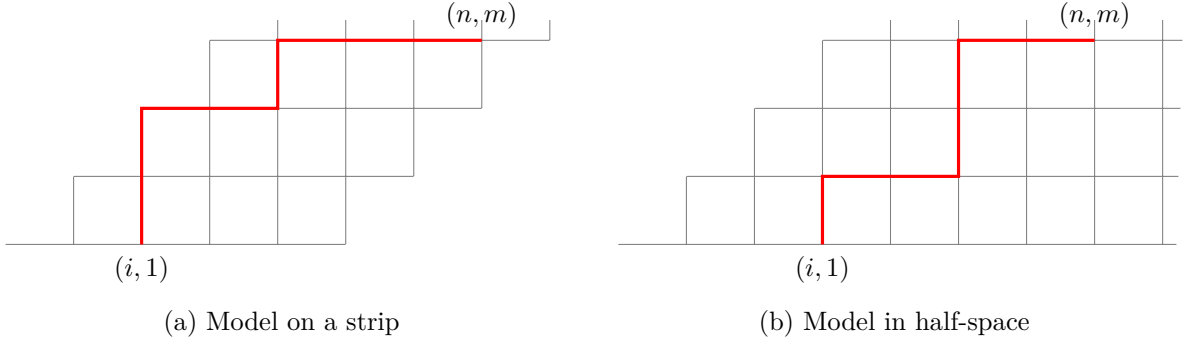
\begin{figure}
\centering
\begin{subfigure}[b]{0.48\textwidth}
    \centering
    \begin{tikzpicture}[scale=0.9]
        \begin{scope}
            \clip (-0.1,-0.1) -- (4.9,-0.1) -- (8.3,3.3) -- (3.3,3.3) -- cycle;
            \draw[gray] (0,0) grid (8.3,3.3);
        \end{scope}
        \draw[red,very thick]
            (2,0) node[below,black,font=\small]{$(i,1)$} -- (2,2) -- (4,2) -- (4,3) -- (7,3)
            node[above,text=black,font=\small] {$(n,m)$};
    \end{tikzpicture} 
    \caption{Model on a strip}
\label{fig:polymer picture on a strip}
\end{subfigure}
\hfill
\begin{subfigure}[b]{0.48\textwidth}
    \centering
    \begin{tikzpicture}[scale=0.9]
        \begin{scope}
            \clip (-0.1,-0.1) -- (8.2,-0.1) -- (8.3,3.3) -- (3.3,3.3) -- cycle;
            \draw[gray] (0,0) grid (8.3,3.3);
        \end{scope}
        \draw[red,very thick]
            (3,0) node[below,black,font=\small]{$(i,1)$}
            -- (3,1) -- (5,1) -- (5,3)  -- (7,3)
            node[above,black,font=\small] {$(n,m)$};
    \end{tikzpicture} 
    \caption{Model in half-space}
\label{fig:polymer picture in half-space}
\end{subfigure}
\caption{An admissible path is highlighted in red for each model.}
\end{figure}

\begin{definition}[GLPP on a strip \cite{BarraquandCorwinYang}]\label{def:Strip GLPP}
Let $a,r,s\in\RR$ be parameters such that
\be\label{eq:parameter domain geo}
a,ar,as\in(0,1).
\ee  
Consider independent random variables $(w_{i,j})_{1\leq j\leq i\leq j+N}$ such that for $  j<i<j+N$, $w_{i,j}\sim\Geo(a^2)$,  $w_{j,j}\sim\Geo(ar)$, and $w_{j+N,j}\sim\Geo(as)$. For a given initial condition $G_{0}:
\llbracket 0, N \rrbracket
\rightarrow\RR$ with $G_0(0)=0$, define for $1\leq m\leq n \leq m+N$,
$$
G(n,m)=\max_{1\leq i\leq N} \left\{  G_0(i)+\max_{\pi:(i,1)\rightarrow(n,m)}\sum_{(i,j)\in\pi}w_{i,j}  \right\},
$$
where the inner maximum runs over up-right paths $\pi$ from $(i,1)$ to $(n,m)$ in the strip, as shown in Figure \ref{fig:polymer picture on a strip}.
\end{definition}

\begin{definition}[Half-space GLPP \cite{BarraquandCorwin}] 
\label{def:Halfspace GLPP}
   Let $a,r$ be parameters such that $a, ar \in (0,1)$. Consider independent  random variables $(w_{i,j})_{i\geq j\geq 1  }$ such that for $i> j\geq 1$, $w_{i,j}\sim\Geo(a^2)$, and  $w_{j,j}\sim\Geo(ar)$. For a given initial condition $G_{0}:\NN\rightarrow\RR$ with $G_0(0)=0$, define for $  n\geq m \geq 1  $,
$$
G(n,m)=\max_{i\geq 1 } \left\{  G_0(i)+\max_{\pi:(i,1)\rightarrow(n,m)}\sum_{(i,j)\in\pi}w_{i,j}  \right\},
$$
where the inner maximum runs over up-right paths $\pi$ from $(i,1)$ to $(n,m)$ in half-space, as shown in Figure \ref{fig:polymer picture in half-space}.
\end{definition}
In both cases, we may then consider the Markov process $(G_m)_{m\geq0}$ on $\RR^{N}$ (resp. $\RR^{\NN_{\geq 0}}$),  where   for $  k\in \llbracket 1, N \rrbracket$ (resp. $k\geq 1$), 
 $G_m(k):=G(m+k,m)-G(m,m).$  
We say that a probability measure $\PP$ on $\RR^N$ (resp. $\RR^{\NN_{\geq 0}}$) is stationary for the GLPP  on the strip (resp. in  half-space)  if $G_0\sim \PP$ implies $G_m\sim\PP$ for all $m\geq 0$. 

\subsubsection{The stationary measures} 
\label{subsubsec:stationary measures GLPP}
Let $\mathcal P $ denote  the path space  
$$\mathcal P =\{L:\ZZ_{\geq0}\to\R_{\geq0}|\ L(0)=0, L(i)-L(i-1)\geq0, \quad \forall i\geq 1\},$$
endowed with the product topology. 
For $M\geq 0$,  define $\mathcal P_{M}=\{ (L(k))_{0\leq k\leq M}: L\in \mathcal P\}$.

\begin{definition} 
\label{def:functional and T transform Geo}
    Let  $(L_1,L_2)\in\mathcal{P}^2$. Define the  exponential functional associated to $(L_1,L_2)$   
    by $\mathcal{A}_{0}(L_1,L_2) =0$ and for $n\geq1$,
    \be\label{eq:exponential functional Geo}
  \qquad \mathcal{A}_{n}(L_1,L_2)
  =\max_{1\leq j\leq n} \Big(L_1(j)-L_2(j-1)\Big). 
\ee
    We also define $ \mathcal{A}_{\infty}(L_1,L_2)
 =\lim_{n\to\infty}\uparrow\mathcal{A}_{n}(L_1,L_2)
 \in\RR_{\geq0}\cup\{+\infty\}. $
 
 For $z\in\RR_{\geq0} $, define the  $\mathcal{T}$ transformed paths by 
  $  \mathcal{T}^{z}(L_1,L_2)(0)=(0,0)$ and  
\be
\label{eq:T transform Geo}
\mathcal T^z(L_1,L_2)(n)
  =\Big(
 L_1(n)+(\mathcal A_n(L_2,L_1)-z)_+,
 L_2(n)-(\mathcal A_n(L_2,L_1)-z)_+
 \Big)  
\ee 
for $n\geq1$.
We also define  
\be\label{eq:T transform at infinity Geo}
\mathcal{T}^{\infty}=\id_{\mathcal{P}^2}. 
\ee

\end{definition}

\begin{remark}
When $z=0$, we recover the celebrated discrete Pitman's transform \cite{Pitman} 
(see also \cite{BianeBougerolOConnell,Corwin_PolymerInvariance})
\be\label{eq:Pitman transform}\mathcal{T}^0 (L_1,L_2)=\big(L_2 \odot L_1, L_1\otimes L_2\big),\ee
where the transformations $\otimes$ and $\odot$ are defined by 
$$
L_1\odot L_2(k)=\max_{1\leq j \leq k} \Big( L_1(j)+L_2(k)-L_2(j-1) \Big), \ 
 L_1\otimes L_2(k)=\min_{1\leq j \leq k} \Big( L_1(j-1)+L_2(k)-L_2(j) \Big). 
$$
Pitman's transform is related to the Robinson--Schensted--Knuth (RSK) correspondence \cite{DauvergneMihaiBalint} and has played an important role in integrable probability,  see in particular \cite{OConnell_ConditionedRandomWalks,OConnellYor,Oconnell_pathtransform}. 
In view of \eqref{eq:T transform at infinity Geo}, the $\mathcal{T}$ transform can be seen as an interpolation between the identity map and Pitman transform on $\mathcal{P}^2$.  
\end{remark}
The functional $\mathcal{A}$ and transform $\mathcal{T}$ enjoy several deterministic algebraic properties (Proposition  \ref{prop:T properties Geo}) as their (one-layer) continuous  counterparts \eqref{eq:one layer exponential functional and T transform} do \cite[Proposition 2.1]{DonatiMartinMatsumotoYor}. 
In order to use them to describe the stationary measures of GLPP, we need  some notations.
For $q\in (0,1)$, define the geometric random walk  
$\PP_{\Geo}^{q}$ to be the probability distribution on $\mathcal{P}\ni L$ such that $L(i+1)-L(i)\overset{i.i.d.}{\sim}\Geo(q)$ for $i\geq 0$. For $M\geq 0$ we denote $\PP_{\Geo,M}^{q}$ to be the marginal distribution of $(L(k))_{0\leq k\leq M}$. For  $q_1,q_2\in(0,1)$, we 
denote 
$\PP_{\Geo}^{q_1,q_2}$ to be the product measure $\PP_{\Geo}^{q_1}\times\PP_{\Geo}^{q_2}$ and we define $\PP_{\Geo,M}^{q_1,q_2}$ likewise. 
 
\begin{proposition}[{\cite[Theorem 1.3]{BarraquandCorwinYang}}]
\label{prop:Stationary measure strip Geo}
For $M\geq1$ and parameters $a,r,s$ satisfying \eqref{eq:parameter domain geo}, define a probability measure on 
$(L_1,L_2)\in\mathcal P_{M}^2$ 
by
\be\label{eq:Stationary measure strip Geo}
\PP_{\Geo,M}^{a,r,s}(L_1,L_2)
=\frac{(rs)^{-L_2(M)+\mathcal A_M(L_2,L_1)}}{\mathcal K_M^{a,r,s}}
\PP_{\Geo,M}^{as,ar}(L_1,L_2),
\ee
where $\mathcal K_M^{a,r,s}$ is the normalizing constant and is finite.
Then the marginal law of $L_1$ is the (unique) ergodic stationary measure for GLPP on the strip in Definition \ref{def:Strip GLPP}.
\end{proposition}

\begin{proposition}[{\cite[Proposition 3.2]{BarraquandCorwin}}] 
\label{prop:Stationary measure half space Geo}
For parameters $a,r,s$ satisfying \eqref{eq:parameter domain geo} with   $s\geq\max\{1,r\}$, consider the $\mathcal{T}$ transformed process
defined by 
\be\label{eq:HY process Geo}
\mathbb{HY}_{\Geo}^{a,r,s}=\mathcal T^Y(L_2,L_1),
\ee
where $L_1\sim\PP_{\Geo}^{as}$, $L_2\sim\PP_{\Geo}^{a/s}$ and  $Y\sim\Geo(r/s)$ are independent.
Then the marginal law of $[\mathbb{HY}_{\Geo}^{a,r,s}]_1$ is
stationary for the half-space GLPP in Definition \ref{def:Halfspace GLPP}.
\end{proposition} 

\begin{remark} 
When $s=r$, the marginal process $[\mathbb{HY}_{\Geo}^{a,r,r}]_1$ is  simply a geometric random walk $\PP_{\Geo}^{a/r}$ in view of conventions \eqref{eq:convention Geo} and \eqref{eq:T transform at infinity Geo}. Historically, these types of product measures are the first family of stationary measures found for LPP and polymer models in half-space \cite{BeteaFerrariOccelli,BarraquandKrajenbrinkLeDoussal_productmeasure}. The full family \eqref{eq:HY process Geo} is first found and conjectured    to consist of all extremal stationary measures 
in \cite{BarraquandCorwin}, which is recently proven in \cite{DasSorensenYang}.
\end{remark}

\subsubsection{The results}
For a pair of paths $(L_1,L_2)\in\mathcal{P}^2$, 
consider the affine extensions of the exponential functional for 
$ \a\in\RR\cup\{-\infty\}$, $\psi\in\RR$ and $M\geq 0$: 
\be\label{eq:affine extended exponential functional Geo}
\mathcal A_M^{\alpha,\psi}(L_1,L_2)
 :=\alpha\vee(\mathcal A_M(L_1,L_2)+\psi),
\ee
with the convention that $-\infty\vee b=b$ for all $b\in\RR$. 
For $a,r,s$ satisfying \eqref{eq:parameter domain geo}, 
$M\in\ZZ_{\geq 1}$ and $ \a\in\ZZ\cup\{-\infty\}$, $\psi\in\ZZ$, define the associated  partition function
\be\label{eq:generalized partition function geo}
 \mathcal{K}_{M}^{a,r,s}(\alpha,\psi)
  :=
 \EE^{as,ar}_{\Geo,M}
 \left[(rs)^{-L_2(M)+\mathcal A_M^{\alpha,\psi}(L_2,L_1)}\right],
\ee
When $ \a=\psi=0$, we recover
$\mathcal{A}_{M}^{0,0}=\mathcal{A}_{M}  $
and 
$\mathcal{K}_{M}^{a,r,s}(0,0)
=\mathcal{K}_{M}^{a,r,s}$ in \eqref{eq:Stationary measure strip Geo}.
 
Our first main result gives the full asymptotics of the 
partition function, which is a zero temperature, discrete analogue of 
\cite[Theorem 2.2]{HariyaYor}.  
\begin{theorem}\label{thm:Geo main thm preliminary}
Assume the parameters $a,r,s$ satisfy \eqref{eq:parameter domain geo}. Throughout the paper we denote 
\be\label{eq:s star}
s_\star=\max\{1,r,s\}.
\ee
Then for all $\alpha \in\ZZ\cup\{-\infty\}$ and $\psi\in\ZZ $,
there exist explicit constants 
$  c^{a,r,s}(\alpha,\psi)>0$
such that, as $M\to\infty$,
\be
\mathcal K_M^{a,r,s}(\alpha,\psi)
\sim  c^{a,r,s}(\alpha,\psi)
M^{\chi(r,s)}(\rho(a,r,s))^M,
\ee
where
\be\label{eq:rho and chi Geo}
\rho(a,r,s)
=\frac{(1-ar)(1-as)}{(1-as_\star)(1-a/s_\star)},
\qquad
\chi(r,s)=
\begin{cases}
-3/2,&r,s<1,\\
-1/2,& \max\{r,s\}=1>\min\{r,s\},\\
1,&r=s>1,\\
0,&\text{otherwise}.
\end{cases}
\ee
\end{theorem}
Theorem \ref{thm:Geo main thm preliminary} is a simplified version of Theorem \ref{thm:Geo main thm}, where we specify the constants
$ c^{a,r,s}(\alpha,\psi)$ through the function that we define below. We remark that the function $\rho$ is continuous in $(r,s)$  while the exponent $\chi$ varies when crossing the separation lines in the phase diagram (Figure \ref{fig:phase diagram}), with the parameter correspondence 
$(u,v):=(-\log r,-\log s)$. 

\begin{definition}
    For $r>0$, $s\geq \max\{1,r\}$ and 
    $ \a\in\ZZ\cup\{-\infty\}, \psi\in\ZZ$, let us define a function 
    \be\label{eq:function N Geo}
\mathcal N^{r,s}(\a,\psi)
:=\EE\left[
(rs)^{\max\{\a-Y,\psi\}}
\right],
\ee
where the expectation is taken over a geometric random variable $Y\sim\Geo(\frac{r}{s})$. In particular,  we have 
\be\label{eq:special case function N Geo}
\mathcal N^{r,s}(-\infty,\psi)=(rs)^{\psi}, \qquad \mathcal N^{r,r}(\a,\psi)=r^{2\psi},
\ee
in view of convention \eqref{eq:convention Geo}.
\end{definition}
The function $\mathcal{N}$ can be expressed  as a sum of Schur polynomials, see Section \ref{subsubsec:schur preliminary}. 
For a pair of paths 
$L=(L_1,L_2)\in\mathcal P ^2\cap\ZZ^2$ and $k\geq0$, we denote 
$L|_{\llbracket 0, k \rrbracket }$ to be its restriction to the first $k$ coordinates $(L_1(i),L_2(i))_{0\leq i\leq k}$ and 
\be\label{eq:function N for two paths Geo}
\mathcal N_k^{r,s}(L)
=\mathcal N^{r,s}\Bigl(\mathcal A_k(L_2,L_1),L_2(k)-L_1(k)\Bigr).
\ee
Our second result is a Girsanov relation for the $\mathcal T $-transformed geometric random walks. 
It is a  zero temperature, discrete analogue of Proposition \ref{prop:Grisanov KPZ}.  
 
\begin{proposition} \label{prop:Girsanov Geo}
Assume \eqref{eq:parameter domain geo} and   $s\geq \max\{1,r\}$. Let $A,S,Y$ be
independent with
\be
 A\sim\PP_{\Geo}^{as},
 \qquad
 S\sim\PP_{\Geo}^{a/s},
 \qquad
 Y\sim\Geo(r/s).
\ee
Then for all $N\geq 1$ and any bounded measurable function $F:\mathcal P_N^2\to\RR$, we have 
\be  \label{eq:Geo Girsanov}
 \EE\left[
 F\left(
 \mathcal T^Y(S,A)|_{\llbracket 0, N \rrbracket }
 \right)
 \right] 
 =\EE\left[
 F\left((A,S)\big|_{\llbracket 0, N \rrbracket }\right)
 \mathcal{N}_{N}^{r,s}(A,S)
 \right].
\ee
\end{proposition}
\begin{remark}\label{rmk:queue interpretation}
For the model in equilibrium, that is when the parameter of the transform has distribution $Y_*\sim\Geo(1/s^2)$ in our notation, 
the joint process $(A,S)$ as well as the transformed process $(D,R):=\mathcal{T}^{Y_*}(S,A)$ have been considered in the literature \cite{DraiefMairesseOConnell}. In that case, their motivation was to study a joint version of Burke's theorem for the FIFO queue, based on ideas from \cite{OConnellYor,KonigOConnellRoch}. Using the queueing terminology, the processes $(A,S,D)$ can be interpreted respectively as the cumulative Arrival time,  Service time and Departure time of the queue.  
\end{remark}
Proposition \ref{prop:Girsanov Geo} is proven in Section \ref{subsec:proof of the theorems Geo}.
In view of Proposition \ref{prop:Stationary measure half space Geo}, it characterizes the marginal distribution of the stationary measure of the GLPP in half-space. As a corollary, we prove that the joint distribution of $(L_1,L_2)$ under $\PP_{\Geo,M}^{ a,r,s}$ converges to the corresponding  Hariya--Yor process, which can be seen as a discrete (two-layer) analogue of Corollary \ref{cor:HY KPZ}.

\begin{corollary}
\label{cor:HY theorem Geo}
Assume the parameters $a,r,s$ satisfy  \eqref{eq:parameter domain geo}. Then as $M\to\infty$, there holds weak convergence in $\mathcal{P}^2$
$$
\PP_{\Geo,M}^{a,r,s}
\Longrightarrow\Law\bigl(\mathbb{HY}_{\Geo}^{a,r,s_\star}\bigr). 
$$ 
\end{corollary}
Corollary \ref{cor:HY theorem Geo} is proven in Section \ref{subsec:proof of the theorems Geo}.   In view of Propositions \ref{prop:Stationary measure strip Geo} and \ref{prop:Stationary measure half space Geo},  it implies that the stationary measure of the GLPP on a strip converges to the corresponding half-space stationary measure, according to the 
$(r,s)$ in the phase diagram (Figure \ref{fig:phase diagram}). 

Now let us turn to the log-gamma polymer.

\subsection{Stationary measures of log-gamma polymers} 
\label{subsec:Log gamma}
\subsubsection{The models} 
\begin{definition}[Log-gamma polymer on the strip \cite{BarraquandCorwinYang}]\label{def:Strip LG polymer}
Let $\a,u,v$ be parameters such that
\be \label{eq:parameter domain LG}
 \alpha>0,\qquad \alpha+u>0,\qquad \alpha+v>0.
\ee
Consider independent  random variables $(\omega_{i,j})_{1\leq j\leq i\leq j+N}$ such that for $  j<i<j+N$, $\omega_{i,j}\sim\Gammainv(2\a)$,  $\omega_{j,j}\sim\Gammainv(\a+u)$, and $\omega_{j+N,j}\sim\Gammainv(\a+v)$. For a given initial condition 
$H_{0}:\llbracket 0, N \rrbracket \rightarrow\RR$ with $H_0(0)=0$, define for $1\leq m\leq n \leq m+N$,
$$
H(n,m)=\log\left( \sum_{1\leq i\leq N}  e^{  H_0(i)}\sum_{\pi:(i,1)\rightarrow(n,m)}\prod_{(i,j)\in\pi}\omega_{i,j}  \right),
$$ 
where the inner sum runs over up-right paths $\pi$ from $(i,1)$ to $(n,m)$ in the strip.
\end{definition}

\begin{definition}[Half-space log-gamma polymer \cite{BarraquandCorwin}] \label{def:Halfspace LG polymer}
  Let $\a,u$ be parameters such that $\a, \a+u>0$. Consider independent  random variables 
  $(\omega_{i,j})_{i\geq j\geq 1 }$
such that for $i>j $, $\omega_{i,j}\sim\Gammainv(2\a)$ and $\omega_{i,i}\sim\Gammainv(\a+u)$. For a given initial condition $H_{0}: \NN\rightarrow\RR$ with $H_0(0)=0$, define for  $  n\geq m \geq 1  $, 
$$
H(n,m)=\log\left( \sum_{i\geq 1}  e^{  H_0(i)}\sum_{\pi:(i,1)\rightarrow(n,m)}\prod_{(i,j)\in\pi}\omega_{i,j}  \right),
$$ 
where the inner maximum runs over up-right paths $\pi$ from $(i,1)$ to $(n,m)$ in half-space.
\end{definition}
In both cases, we may  then define the Markov process $(H_m)_{m\geq0}$ and its stationarity   in the same way as for the GLPP in Section \ref{subsubsec:GLPP models definition}.

\subsubsection{The stationary measures}
Let $\mathscr{P}$ denote the paths space  
$$
 \mathscr{P}=\{L:\ZZ_{\geq0}\to\RR\mid L(0)=0\}, 
$$
endowed with the product topology. 
For $M\geq 0$, define  $\mathscr P_{M}=\{ (L(k))_{0\leq k\leq M}: L\in \mathscr P\}$.
\begin{definition} 
    Let  $(L_1,L_2)\in\mathscr{P}^2$. Define the  exponential functional associated to $(L_1,L_2)$ 
    by $\mathscr{A}_{0}(L_1,L_2) =0$ and for   $n\geq1$,
    \be\label{eq:exponential functional LG}
  \qquad \mathscr{A}_{n}(L_1,L_2)
  =\sum_{j=1}^n e^{L_1(j)-L_2(j-1)}. 
\ee
    We also define $ \mathscr{A}_{\infty}(L_1,L_2)
 =\lim_{n\to\infty}\mathscr{A}_{n}(L_1,L_2)
 \in\RR_{\geq0}\cup\{+\infty\}. $
 
 For $z\geq 0$,   define the  $\mathscr{T}$ transformed paths 
 $\mathscr{T}^{z}(L_1,L_2)$ by $  \mathscr{T}^{z}(L_1,L_2)(0)=(0,0)$ and for $n\geq1$, 
\be
\label{eq:T transform LG}
\mathscr{T}^{z}(L_1,L_2)(n)=\Bigl(L_1(n)+\log\bigl(1+z\mathscr{A}_{n}(L_2,L_1)\bigr) ,L_2(n)-\log\bigl(1+z\mathscr{A}_{n}(L_2,L_1)\bigr) \Bigr)
\ee
In particular $\mathscr{T}^0=\id_{\mathscr{P}^2}$.
\end{definition}

\begin{remark}
We recover the celebrated discrete geometric Pitman's transform \cite{Kirillov} (see also \cite{NoumiYamada,Corwin_PolymerInvariance}) in the following way: 
\be\label{eq:renormalization to get geometric Pitman transform}
\lim_{z\to\infty}
\Big(\mathscr{T}^z (L_1,L_2)+(-\log z,\log z)\Big)
=\big(L_2 \odot L_1, L_1\otimes L_2\big),
\ee
where  the transformations $\otimes$ and $\odot$ are defined (with a slight abuse of notation) by 
\be\label{eq:geometric Pitman transform} L_1\odot L_2(k)=\log\sum_{j=1 }^{k} e^{   L_1(j)+L_2(k)-L_2(j-1) }, \quad 
 L_1\otimes L_2(k)=-\log\sum_{j=1 }^{k} e^{- \big(L_1(j-1)+L_2(k)-L_2(j)\big) }.
 \ee
Geometric Pitman's transform is related to the geometric RSK  correspondence  and has played an important role in integrable probability,  see in particular \cite{OConnell_ConditionedRandomWalks,OConnellYor,Oconnell_pathtransform,CorwinOConnellSeppalainenZygouras}. 
In view of  \eqref{eq:renormalization to get geometric Pitman transform}, the (appropriately renormalized) $\mathscr{T}$ transform can be seen as an interpolation between the identity map and Pitman transform on $\mathscr{P}^2$.  
\end{remark}

\begin{remark}[Tropicalization]
Let $(L_1,L_2)\in\mathcal{P}^2$. Then for all $M\geq 1 $ and $y\geq 0$, we have 
$$
\frac{1}{\beta}\log\mathscr{A}_M(\beta L_1,\beta L_2)\xrightarrow{\beta\to\infty}
\mathcal{A}_M(L_1,L_2),
\quad 
\frac{1}{\beta}\mathscr{T}^{e^{-\beta y}}(\beta L_1,\beta L_2)\xrightarrow{\beta\to\infty}
\mathcal{T}^{y}(L_1,L_2).
$$
\end{remark}

The functional $\mathscr{A}$ and transform $\mathscr{T}$ enjoy several deterministic algebraic properties (Proposition  \ref{prop:T properties LG}). Probabilistically, they enter the stationary measures of the log-gamma polymers as follows. We need some notations. 
For $\theta>0$, define the log-gamma random walk $\PP_{\LG}^{\theta}$ to be the probability distribution on $\mathscr{P}\ni L$ such that   
$e^{L(i+1)-L(i)}\overset{i.i.d.}{\sim}\Gammainv(\ta)$ for $i\geq 0$.
One can then define $\PP_{\LG,M}^{\theta}, \PP_{\LG}^{\theta_1,\theta_2},\PP_{\LG,M}^{\theta_1,\theta_2}$ analogously to their  geometric random walk counterparts in Section \ref{subsubsec:stationary measures GLPP}.

\begin{proposition}[{\cite[Theorem 1.6]{BarraquandCorwinYang}}]  
\label{prop:Stationary measure strip LG}
For $M\geq 1$ and $\a,u,v$ satisfying \eqref{eq:parameter domain LG}, 
define a probability distribution on $(L_1,L_2)\in\mathscr{P}_M^2$ by 
\be\label{eq:Stationary measure LG}
  \PP_{\LG,M}^{\alpha,u,v} (L_1,L_2)
 =\frac{1}{\mathscr{K}_{M}^{\alpha,u,v}}
 e^{(u+v)\big(L_2(M)-\log\mathscr{A}_{M}(L_2,L_1)\big)}\PP_{\LG,M}^{\alpha+v,\alpha+u}(L_1,L_2).
\ee  
    Then under $\PP_{\LG,M}^{\alpha,u,v}$, the marginal distribution of $L_1$ is the (unique) ergodic stationary measure for the log gamma polymer on a strip in Definition \ref{def:Strip LG polymer}.
\end{proposition}

\begin{proposition}[{\cite[Theorem 1.8]{BarraquandCorwin}}] 
\label{prop:Stationary measure half space LG}
For parameters $\a,u,v$ satisfying \eqref{eq:parameter domain LG} with $v\leq \min\{0,u\}$,  consider the $\mathscr{T}$ transformed process:
\be\label{eq:HY process LG}
\mathbb{HY}^{\a,u,v}_{\LG}= \mathscr T^{1/Y}(L_2,L_1),
\ee
where $L_1\sim\PP^{\a+v}_{\LG}$, $L_2\sim\PP^{\a-v}_{\LG}$ and $Y\sim\Gammainv(u-v)$ are independent.
Then  the marginal law of $[\mathbb{HY}^{\a,u,v}_{\LG}]_1$  is stationary for the  log-gamma polymer in half-space in Definition \ref{def:Halfspace LG polymer}. 
\end{proposition}

\subsubsection{The results}
For a pair of paths $(L_1,L_2)\in\mathscr{P}^2$, let us define affine extensions of the exponential functional for $ a\geq0$  and $\xi>0$ by 
\be\label{eq:affine extended exponential functional LG}
 \mathscr{A}_{M}^{ a,\xi}(L_1,L_2)
  :=  a+\xi\mathscr{A}_{M}(L_1,L_2).
\ee
For $\a,u,v$ satisfying \eqref{eq:parameter domain LG} and 
$M\in\ZZ_{\geq1}$, define the associated partition function 
\be\label{eq:generalized partition function LG}
 \mathscr{K}_{M}^{\alpha,u,v}(  a,\xi)
 :=\EE_{\LG,M}^{\alpha+v,\alpha+u}
 \left[e^{(u+v)L_2(M)}
 \Big(\mathscr{A}_{M}^{  a,\xi}(L_2,L_1)\Big)^{-(u+v)}\right]. 
\ee
When $ a=0, \xi=1$, we recover 
$\mathscr{A}_{M}^{0,1}=\mathscr{A}_{M}  $ and 
$\mathscr{K}_{M}^{\alpha,u,v}(0,1)
=\mathscr{K}_{M}^{\alpha,u,v}$ in \eqref{eq:Stationary measure LG}.

Our first main result is the full asymptotics of $ \mathscr{K}_{M}^{\alpha,u,v}(  a,\xi)$.

\begin{theorem}\label{thm:thm:LG main thm preliminary}
Assume the parameters $\a, u,v $ satisfy  \eqref{eq:parameter domain LG}. Recall the definition of $v_\star$   in \eqref{eq:v star}. 
Then for all $ a\geq0 $ and $ \xi>0$,  there exists explicit constants
$ C^{\a,u,v}(a,\xi)>0$ such that  as $M\to\infty$, 
    \be
 \mathscr{K}_{M}^{\a,u,v}(a,\xi)
 \sim  C^{\a,u,v}(a,\xi)M^{\chi(u,v)}
 \big(\varrho (\a,u,v)\big)^M, 
\ee
where 
\be\label{eq:rho and chi LG}
\varrho (\a,u,v)
=\frac{\Gamma(\alpha+v_\star)\Gamma(\alpha-v_\star)}
 {\Gamma(\alpha+u)\Gamma(\alpha+v)},
 \qquad
\chi(u,v)=
\begin{cases}
    -3/2, &    u,v>0, \\
    -1/2, &    \max\{ u, v\}>0=\min\{ u, v\}, \\
    1,&   u=v<0\\
    0,& \text{otherwise}. 
\end{cases}
.
\ee
\end{theorem}
Theorem \ref{thm:thm:LG main thm preliminary} is a simplified version of Theorem \ref{thm:LG main thm}, where we specify the constants
$ C^{\a,u,v}(a,\xi)>0$ using the same function $\mathscr{N}^{u,v}(a,\xi)$ defined in \eqref{eq:function N KPZ}. 

Our second result is a Girsanov relation for the $\mathscr{T}$ transformed log-gamma random walks.  For a  pair of paths $L=(L_1,L_2)\in\mathscr{P}^2$ and
$k\in\ZZ_{\geq0}$, we denote  $L|_{\llbracket 0, k \rrbracket }$ its restriction on the first $k$ coordinates $(L_1(i),L_2(i))_{0\leq i\leq k}$ and 
\be 
\label{eq:function N for two paths LG} 
 \mathscr{N}_k^{u,v}(L)
  =\mathscr{N}^{u,v}\left(
 \mathscr{A}_{k}(L_2,L_1),
 e^{L_2(k)-L_1(k)}\right).
\ee

\begin{proposition} \label{prop:Girsanov LG}
Assume \eqref{eq:parameter domain LG} and   $v\leq \min\{0,u\}$. Let $A,S,Y$ be
independent  with
\be
 A\sim\PP_{\LG}^{ \a+v },
 \qquad
 S\sim\PP_{\LG}^{\a-v},
 \qquad
 Y\sim\Gav(u-v).
\ee
Then for all $N\geq 1 $ and all bounded measurable function $F:
\llbracket 0, N \rrbracket \to\RR$, we have 
\be  \label{eq:LG Girsanov}
 \EE\left[
 F\left( \mathscr T ^{1/Y}(S,A)|_{\llbracket 0, N \rrbracket } \right)
 \right] 
 =\EE\left[
 F\left((A,S)\big|_{\llbracket 0, N \rrbracket }\right)
 \mathscr{N}_{N}^{u,v}(A,S)
 \right].
\ee
\end{proposition}

Proposition \ref{prop:Girsanov LG} is proven in Section \ref{subsec:proof of the theorems LG}.
In view of Proposition \ref{prop:Stationary measure half space LG}, it characterizes the marginal distribution of the stationary measure of the log gamma polymer in half-space. As a corollary, we prove that the joint distribution of $(L_1,L_2)$ under $\PP_{\LG,M}^{ \a,u,v}$ converges to the corresponding Hariya--Yor process.  

\begin{corollary}
\label{cor:HY theorem LG}
 Assume the parameters satisfy  \eqref{eq:parameter domain LG}. Then  as $M \to \infty$,  there holds weak convergence in $\mathscr P^2$:
$$\PP_{\LG,M}^{\alpha,u,v}\Longrightarrow
\Law\Big(\mathbb{HY}_{\LG}^{u,v_{\star}}\Big).$$ 
\end{corollary}
Corollary \ref{cor:HY theorem LG} is proven in Section \ref{subsec:proof of the theorems LG}. In view of Propositions \ref{prop:Stationary measure strip LG} and \ref{prop:Stationary measure half space LG},  
taking the first layer marginal of both sides recovers a  result established recently  in \cite[Theorem 1.2]{BrycWesolowskiLG},
which says that the stationary measure of the log-gamma polymer on a strip converges to the corresponding half-space stationary measure, according to the $(u,v)$ in the phase diagram (Figure \ref{fig:phase diagram}). We discuss the relation between this convergence and that in    Corollary \ref{cor:HY KPZ} in Section \ref{subsubsec:On different convergences}.

\subsection{Relation to previous work}
\label{subsec:Relations to the literature}
\subsubsection{A two-layer formulation of the Hariya--Yor theorem}
\label{subsubsec:A two-layer formulation of the Hariya--Yor theorem}
A keen reader might wonder why the description of the  stationary measures of the KPZ equation on an interval (Proposition \ref{prop:Stationary measure open KPZ}) or half-space (Proposition \ref{prop:Stationary measure half space KPZ}) concerns the exponential functional and $T$ transform for only one Brownian motion. We insist in this section that they admit perfectly a two-layer description and this point of view is an important starting point of this paper. 

For a pair of continuous paths $(X_1,X_2)\in\mathcal{C}^2$ and $x,z\in\RR_{\geq 0}$,  
consider the adapted exponential functionals and  define the $\mathbb{T}$ transform by  
    \be\label{eq:exponential functional KPZ}
\mathbb A_x(X_1,X_2)
=\int_0^x e^{X_1(s)-X_2(s)}\dd s, 
\ee 
and 
\be
\mathbb{T}^{z}(X_1,X_2)(x)=\Bigl(X_1(x)+\log\bigl(1+z\mathbb{A}_{x}(X_2,X_1)\bigr), X_2(x)-\log\bigl(1+z\mathbb{A}_{x}(X_2,X_1)\bigr) \Bigr).
\ee
\begin{remark}
    When $z$ is taken as in \eqref{eq:renormalization to get geometric Pitman transform}, the resulting operation is the celebrated continuous Pitman's transform \cite{MatsumotoYor_geoPitman,BianeBougerolOConnell_Geo}. Both the  transform and the functional \eqref{eq:exponential functional KPZ} have played an important role in integrable probability, see in particular \cite{OConnellYor_BrownianBurke,OConnell_Toda}.
\end{remark}

\begin{remark}
     For $X\in\mathcal{C}$ and $x,z\in\RR_{\geq 0}$, the functional $\mathbb{A}$ and the transform $\mathbb{T}$ are related to their one-layer counterparts \eqref{eq:one layer exponential functional and T transform} by 
$$
\mathbb{A}_x(X,-X)=A_x(X), \qquad  \mathbb{T}^z  (-X, X)=\big( -T^z(X),T^z(X) \big).
$$
\end{remark}

Then, the marginal process of $B_1$ under the probability measure 
$$
    \dd\PP_{\KPZ,L}^{u,v} (B_1,B_2)
:=\frac{e^{(u+v)B_2(L)}\mathbb A_L(B_2,B_1)^{-(u+v)}}
{\mathbb K_{\BM,L}^{u,v}}\dd\widehat\PP_{\BM,L}^{-v,-u}(B_1,B_2), 
$$
where $ {\widehat\PP_{\BM}^b}$ denotes  the Brownian motion with drift $b$  and diffusion coefficient $ 1$, is stationary for the open KPZ equation \eqref{eq:open KPZ eq}.  Indeed, by the Cameron--Martin theorem, we have 
\be\label{eq:change of variables KPZ}
\begin{aligned}
\dd\PP_{\KPZ,L}^{u,v} (B_1,B_2)
&=\frac{ \mathbb A_L(B_2,B_1)^{-(u+v)}}
{\mathsf Z_L^{u,v}}\dd\widehat\PP_{\BM,L}^{-v,v}(B_1,B_2)\\
&=\frac{ A_L(X)^{-(u+v)}}
{\mathsf Z_L^{u,v}}\dd\PP_{\BM,L}^{v,0}(X,W)
=\dd\mathsf P_L^{u,v}(X)\d\PP_{\BM,L}^{0}(W),
\end{aligned}
\ee
where the second equality follows from the fact that, under the reference measure $\widehat\PP_{\BM}^{-v,v}$, 
the change of variables
\be\label{eq:paths decomposition}
X=\frac{B_2-B_1}{2},
\qquad
W=\frac{B_1+B_2}{2}
\ee
makes $X$ and $W$ again independent Brownian motions by the Kac-Bernstein theorem. 
Reversing the decomposition \eqref{eq:paths decomposition},  the marginal process $B_1=W-X$ recovers \eqref{eq:stationary measure of open kpz}, see also the same computations in \cite[Section 5]{Corwin}.

The independence  and the fact that $W$ is still a Brownian motion make it essentially suffice to study  the non-trivial part $\dd\mathsf P_L^{u,v}(X)$, as  in Section \ref{subsubsec:The Hariya Yor theorem}. 
Reciprocally, we insist that the decomposition \eqref{eq:paths decomposition} does not keep independence between $X$ and $W$ in general, as the Kac-Bernstein theorem characterizes the Gaussian distribution. In particular in our discrete setting, we study the joint distribution of both non-trivial paths using adapted two-layer exponential functionals and $T$ transform,  in the spirit of \cite{BarraquandCorwin,BarraquandCorwinYang}.

Likewise one checks that the marginal process
$[\mathbb{HY}_{\KPZ}^{u,v}]_1$ of the two-layer Hariya--Yor process 
$$
\mathbb{HY}_{\KPZ}^{u,v}
=\mathbb T^{1/Y}(B_2,B_1),
$$
where $(B_1,B_2)\sim\widehat\PP_{\BM}^{-v,v}$ is independent from
$Y\sim\Gammainv(u-v)$, is stationary for the  
half-space KPZ equation \eqref{eq:open KPZ eq}, and we have the convergence 
$ \PP_{\KPZ,L}^{u,v}\xRightarrow{} \mathbb{HY}_{\KPZ}^{u,v_{\star}}$ as $L\to\infty$.

\subsubsection{On different convergences}
\label{subsubsec:On different convergences}
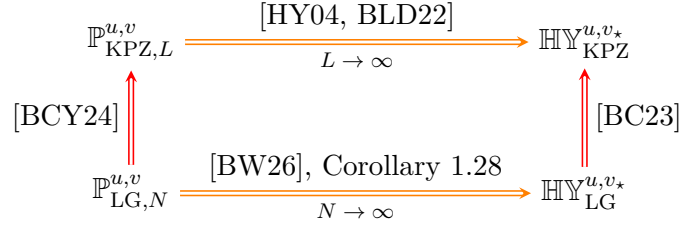
\begin{figure}
    \centering
    \begin{tikzpicture}[>=stealth]
        \node (openKPZ)  at (0,2) {$\PP_{\KPZ,L}^{u,v}$};
         \node (KPZhalfspace) at (6,2 ) {$\mathbb{HY}_{\KPZ}^{u,v_{\star}}$};
          \node (openLG) at (0,0) {$\PP_{\LG,N}^{u,v}$};
           \node (LGhalfspace) at (6,0 ) {$\mathbb{HY}_{\LG}^{u,v_{\star}}$};
           \draw[->,double,double distance=1.1pt,semithick,draw=orange]  (openKPZ.east) -- node[above] {\cite{HariyaYor,BarraquandLeDoussal} }
  node[below] {\scriptsize $L\to\infty$} (KPZhalfspace);
  \draw[->,double,double distance=1.1pt,semithick,draw=red] (openLG.north) -- node[left] {\cite{BarraquandCorwinYang}}
  (openKPZ.south);
  \draw[->,double,double distance=1.1pt,semithick,draw=red] (LGhalfspace.north) -- node[right] {\cite{BarraquandCorwin}}
  (KPZhalfspace.south);
\draw[->,double,double distance=1.1pt,semithick, draw=orange](openLG.east) -- node[above] {\cite{BrycWesolowskiLG}, Corollary \ref{cor:HY theorem LG}}
  node[below] {\scriptsize $N\to\infty$} (LGhalfspace.west);
 \end{tikzpicture}
    \caption{The large time and intermediate-disorder limits of the stationary measures of the open/half-space log-gamma polymer/KPZ equation.}
    \label{fig:convergences}
\end{figure}

In this section, we discuss the convergences between various stationary measures for discrete or continuous models defined in different geometries, which we encounter in the previous sections. These measures are related through different scaling limits, see Figure \ref{fig:convergences}. 

One direction is the convergence from the log-gamma polymer to the KPZ equation, see the red arrows in Figure \ref{fig:convergences}.   
The intermediate disorder limit is first introduced in \cite{AlbertsKhaninQuastel} to prove  that polymer partition functions converge to the solution of the stochastic heat equation in full space.  Half-space analogues are  studied in  \cite{Wu}, see also extensions in \cite{Parekh,BarraquandCorwin}. In the same spirit, the free energy of the log-gamma polymer on a strip is expected to converge to the open KPZ equation as a space-time process \cite[Conjecture 4.2]{BarraquandCorwinYang}.  However, at the level of stationary measures, the intermediate disorder convergence to their continuum counterparts   has been established for   the log-gamma polymer both on a strip  and in half-space, see \cite[Proposition 4.5]{BarraquandCorwinYang} and \cite{BarraquandCorwin} respectively. 

The other direction concerns the convergence of stationary measures from a finite domain to half-space for the same model, see the orange arrows in Figure \ref{fig:convergences}. For the open KPZ equation, we have explained in Section \ref{subsubsec:The Hariya Yor theorem} that such convergence has been studied by \cite{HariyaYor}, though in a completely different context, as observed in \cite{BarraquandLeDoussal}.  As a consequence,  it is appealing to study the corresponding convergence in the discrete setting to complete the story, see Figure   \ref{fig:convergences}. This amounts one of the   initial motivations for an earlier version of this paper. In particular, 
we believe that this convergence  can be addressed in a more general setting,  where the disorder of the models can be inhomogeneous, and the stationary measures can be indexed by arbitrary down-right paths in the lattice, see \cite[Remark 1.3 and 1.4]{Barraquand}.
During the preparation of the earlier version, the preprint  \cite{BrycWesolowskiLG} came out. In this work  we choose to stay in a simpler setting in order to highlight the connection with exponential  functionals of random walks  in the spirit of \cite{HariyaYor}, and we address the convergence  in a two-layer setting  as discussed in Section \ref{subsubsec:A two-layer formulation of the Hariya--Yor theorem}.

\subsection{Outline and proof ideas}
\label{subsec:outline and proof ideas}
The rest of the paper is devoted to  proving the results outlined in the introduction. 
We treat   the GLPP and log-gamma polymer in Section \ref{sec:HY for GLPP} and Section \ref{sec:HY for LG} respectively. 
Although the analysis for the log-gamma polymer is often technically more involved, the proof strategies for both models remain  completely in    parallel. In the region $\min\{u,v\}<0$, which we refer to as the boundary dominated region, we adapt the approach in  \cite{HariyaYor}.  We show that the desired  asymptotics of the partition function can be obtained using simple probabilistic arguments, including notably a Dufresne's relation and time reversal argument. In the fan region $u+v>0$,  using the integral formulas developed in \cite{Barraquand} and performing an analytic continuation, we show that  under  certain parameter  assumptions, the two-layer Schur/Whittaker process admits a homogeneous Markov process as large-time limit and use it to analyze asymptotics. Of course, on the overlap between the boundary dominated region and the fan region, both approaches yield same conclusions. 
  
\subsection*{Acknowledgments}
We are very grateful to Guillaume Barraquand for suggesting this problem as well as for many insightful discussions related to the topics. 
We thank the organisers of the 4th Simons Math Summer Workshop: Algebraic Methods in Probability for providing a stimulating environment where     part of this work was done and many inspiring discussions took place. In particular, we are grateful to Neil O'Connell for illuminating discussions on queueing theory, and to Wlodek Bryc for valuable discussions around his work \cite{BrycWesolowskiLG}.  
We acknowledge support from Guillaume Barraquand's ANR grant ANR-23-ERCB-0007.

\subsection*{Use of AI}
ChatGPT was used for language polishing purposes: we input complicated individual sentences and seek suggestions for improvement. The mathematical part of this paper, from ideas to writing, from statements to proofs, remains classically human-generated.



\section{Hariya--Yor theorem for geometric last passage percolation}
\label{sec:HY for GLPP}
\subsection{Two-layer  $\mathcal T$ transform and asymptotics of the  partition function} 
The goal of this section is to present the full asymptotics of the  partition function $\mathcal{K}_{M}^{a,r,s}(\alpha,\psi)$ defined in 
\eqref{eq:generalized partition function geo}.
Let us start with proving some useful 
deterministic properties satisfied by the $\mathcal{T} $ transform in Definition \ref{def:functional and T transform Geo}.  
\begin{proposition}[Algebraic properties of $\mathcal{T}$]
\label{prop:T properties Geo}
For $L=(L_1,L_2)\in\mathcal P^2$ and $y\in\RR_{\geq0}$:
\begin{enumerate}
\item $ \mathcal{T}^yL\in\mathcal P^2$. 
\item For all $n\geq0$,
\be\label{eq:capping Geo}
 \mathcal A_n\bigl(
 [\mathcal{T}^yL]_2,[\mathcal{T}^yL]_1
 \bigr)
 =\mathcal A_n(L_2,L_1)\wedge y.
\ee
\item  For all $n\geq0$,
\be 
 [\mathcal{T}^yL]_1(n)+[\mathcal{T}^yL]_2(n)=L_1(n)+L_2(n).
 \label{eq:conservation Geo}
\ee
\item For all $z\geq 0$, 
\be\label{eq:semi group property Geo}
 \mathcal{T}^z\mathcal{T}^yL
 =\mathcal{T}^{y\wedge z}L
 =\mathcal{T}^y\mathcal{T}^zL.
\ee
 In particular, recall the Pitman transform \eqref{eq:Pitman transform}, we have 
\be\label{eq:communicate with Pitman transform Geo}
 [\mathcal{T}^yL]_2\odot [\mathcal{T}^yL]_1=L_2\odot L_1,
\qquad
[\mathcal{T}^yL]_1\otimes [\mathcal{T}^yL]_2=L_1\otimes L_2.
\ee
\end{enumerate}
\end{proposition}

\begin{proof}
    Let us prove $\mathcal{T}^yL\in\mathcal P^2$.
    By definition, $L_1(0)=L_2(0)=\mathcal A_0(L_2,L_1)=0$, hence 
    $ [\mathcal{T}^yL]_1(0)= [\mathcal{T}^yL]_2(0)=0$. Let $n\geq 1$, since $\mathcal A_n(L_2,L_1)=\mathcal A_{n-1}(L_2,L_1)\vee (L_2(n)-L_1(n-1))$, $\mathcal A_{n-1}(L_2,L_1)\geq L_2(n-1)-L_1(n-1)$, we have  
    $$0\leq \mathcal A_n(L_2,L_1)-\mathcal A_{n-1}(L_2,L_1)\leq L_2(n)-L_2(n-1).$$
    Since $x\mapsto (x-y)_+$ is non decreasing and $1-$Lipschitz, 
    $$
    [\mathcal{T}^yL]_1(n)-[\mathcal{T}^yL]_1(n-1)=L_1(n)-L_1(n-1)+ (\mathcal A_n(L_2,L_1)-y)_+ 
    -(\mathcal A_{n-1}(L_2,L_1)-y)_+\geq0
    $$
    $$
    [\mathcal{T}^yL]_2(n)-[\mathcal{T}^yL]_2(n-1)=L_2(n)-L_2(n-1)- ((\mathcal A_n(L_2,L_1)-y)_+
    -(\mathcal A_{n-1}(L_2,L_1)-y))\geq0.
    $$

For  \eqref{eq:capping Geo}, we first prove that   
$ \mathcal A_n\bigl(
 [\mathcal{T}^yL]_2,[\mathcal{T}^yL]_1
 \bigr)
 \leq \mathcal A_n(L_2,L_1)\wedge y.$ 
Indeed, for all  $0\leq j\leq n$, 
 \begin{multline*}
     \bigl[\mathcal{T}^yL\bigr]_2(j)
-\bigl[\mathcal{T}^yL\bigr]_1(j-1)
=L_2(j)-L_1(j-1)
-\bigl(\mathcal A_j(L_2,L_1)-y\bigr)_+  
-\bigl(\mathcal A_{j-1}(L_2,L_1)-y\bigr)_+ \\
\leq
\mathcal A_j(L_2,L_1)
-\bigl(\mathcal A_j(L_2,L_1)-y\bigr)_+ 
=\mathcal A_j(L_2,L_1)\wedge y  \leq\mathcal A_n(L_2,L_1)\wedge y .
 \end{multline*} 
We then prove the other direction. If $\mathcal A_n(L_2,L_1)\leq y$, then $\mathcal T^yL(j)=L_j$ for all $0\leq j\leq n$.   If $\mathcal A_n(L_2,L_1)> y$, there exists a minimal   $j_0\leq n$ such that $\mathcal A_{j_0}(L_2,L_1)>y$, we have
$$
[\mathcal{T}^yL]_2(j_0)-[\mathcal{T}^yL]_1(j_0-1)=L_2(j_0)- \mathcal A_{j_0}(L_2,L_1)+y-L_1(j_0-1)\leq y.
$$
Hence we have $ \mathcal A_n\bigl(
 [\mathcal{T}^yL]_2,[\mathcal{T}^yL]_1
 \bigr)
 \geq \mathcal A_n(L_2,L_1)\wedge y $ in both cases.

The property \eqref{eq:conservation Geo} follows   directly from definition. Let us prove \eqref{eq:semi group property Geo}.  
\begin{align*}
[\mathcal{T}^z\mathcal{T}^yL]_1(n) &=
[ \mathcal{T}^yL]_1(n)+ 
\big( \mathcal A_n([\mathcal T^yL]_2,[\mathcal T^yL]_1) -z\big)_+ \\
&=L_1(n)+\big( \mathcal A_n(L_2,L_1) -y\big)_+
+\big( \mathcal A_n(L_2,L_1) \wedge y-z\big)_+\\
&= L_1(n)+\big( \mathcal A_n(L_2,L_1) - y\wedge z\big)_+
= [\mathcal{T}^{y\wedge z}L]_1(n),
\end{align*} 
where we use \eqref{eq:capping Geo} in the second equality. By \eqref{eq:conservation Geo}, we also have 
$[\mathcal{T}^z\mathcal{T}^yL]_1 =[\mathcal{T}^{y\wedge z}L]_2$. Finally,  \eqref{eq:communicate with Pitman transform Geo} follows by taking $z=0$ and relation \eqref{eq:T transform Geo}. 
\end{proof}

Recall the extended exponential functional $\mathcal{A}_{M}^{ \a,\psi}(L_1,L_2)$ \eqref{eq:affine extended exponential functional Geo}.
For $  k\geq 0 $  and  $L\in\mathcal{P}$, define $L ^{(k)}\in\mathcal{P}$ to be the translated path
$
  L ^{(k)}(n)=L (k+n)-L (k) 
$
for $n\geq 0. $
\begin{lemma} 
\label{lem:cutting}
For $N\geq 0$, $\psi\in\RR $, $\alpha\in\RR\cup\{-\infty\}$ and $L_1,L_2\in\mathcal P$.  We have 
\be
 L_2(N)-\mathcal A_N^{\alpha,\psi}(L_2,L_1)
 =\min\{L_2(N)-\alpha, (L_1\otimes L_2)(N)-\psi\} 
 \label{eq:exponential functional developped Geo}
\ee
Moreover,  for $0\leq k\leq N$, we have 
\begin{equation} \label{eq:cutting geo} 
 L_2(N)-\mathcal A_N^{\alpha,\psi}(L_2,L_1)
 = L_2^{(k)}(N-k)
  -\mathcal A_{N-k}^{\alpha_k,\psi_k}
    (  L_2^{(k)},  L_1^{(k)}),
\end{equation}
where 
$$
 \alpha_k=-\min\{ L_2(k)-\a, (L_1\otimes L_2)(k)-\psi\},
 \qquad \psi_k=\psi-L_1(k).
$$
\end{lemma}

\begin{proof}
 The   identity
\eqref{eq:exponential functional developped Geo} follows from the definition of the extended exponential functional \eqref{eq:affine extended exponential functional Geo} and the   Pitman transform \eqref{eq:Pitman transform}.  Hence 
 
\begin{multline*}
    \mathcal A_N^{\alpha,\psi}(L_2,L_1)
=\alpha
\vee\left(\psi+
\bigl( \mathcal A_k(L_2,L_1)\bigr) 
\vee\Bigl(
 L_2(k)-L_1(k)
+\mathcal A_{N-k}(L_2^{(k)},L_1^{(k)})
\Bigr)
\right) \\
=L_2(k)+
\mathcal A_{N-k}^{\alpha_k,\psi_k}
(L_2^{(k)},L_1^{(k)}).
\end{multline*}
where we split the maximum in the exponential functional at $k$ for the first equality  and  use  \eqref{eq:exponential functional developped Geo} for $L_1^{(k)},L^{(k)}_2$ in the second equality.  
\end{proof}

Recall the   partition function $\mathcal{K}_{M}^{a,r,s}(\alpha,\psi)$ defined for 
$\a\in\ZZ\cup\{-\infty\}$ and $\psi\in\ZZ$ in 
\eqref{eq:generalized partition function geo}.
\begin{proposition}
\label{prop:decomposition Geo}
Let $M\geq 1$. For  $0\leq k\leq M$ and  
$l=(\ell_1,\ell_2)\in\mathcal{P}_{k}^2\cap\ZZ^2$, we have 
\be 
\PP_{\Geo,M}^{a,r,s}
\Big((L_1,L_2)|_{[0,k]}=(\ell_1,\ell_2)\Big) 
 =\PP_{\Geo,k}^{as}(\ell_1)\PP_{\Geo,k}^{ar}(\ell_2)
 \frac{\mathcal K_{M-k}^{a,r,s}
   \big(-(\ell_1\otimes\ell_2)(k),-\ell_1(k)\big)}
  {\mathcal{K}_{M}^{a,r,s} }.
 \label{eq:decomposition formula geo}
\ee
\end{proposition}
\begin{proof}
Let 
$\mathcal F_k=\sigma\left(L_1(i),L_2(i):0\leq i\leq k\right)$.
By Markovianity of the random walk,  the paths $(L_1^{(k)},L_2^{(k)})$ are independent
of $\mathcal F_k$ and have distribution
$\PP_{\Geo,M-k}^{as,ar}$. Applying
\eqref{eq:cutting geo} with $ \a=\psi=0 $ and $N=M$, we have 
\begin{multline*}
\EE_{\Geo,M}^{as,ar}
\left[(rs)^{-L_2(M)+\mathcal A_M (L_2,L_1)} 
| \mathcal F_k\right]  
 =
\EE_{\Geo,M}^{as,ar}
\left[
(rs)^{-L_2^{(k)}(M-k)
  +\mathcal A_{M-k}^{\alpha_k,\psi_k}
    (  L_2^{(k)},  L_1^{(k)})}
|\mathcal F_k\right] \\
=\mathcal K_{M-k}^{a,r,s}
   \big(-(\ell_1\otimes\ell_2)(k),-\ell_1(k)\big).
\end{multline*} 
Together with \eqref{eq:Stationary measure strip Geo} proves 
\eqref{eq:decomposition formula geo}.  
\end{proof}

As $M\to\infty$, the marginal distributions of $\PP_{\Geo,M}^{a,r,s}$ are therefore given by asymptotics of the  partition function, which is the main goal of this section.  
Recall the definition of $\rho=\rho(a,r,s)$ in \eqref{eq:rho and chi Geo} and the function $\mathcal N^{r,s}(\a,\psi)$ in \eqref{eq:function N Geo}.

\begin{theorem}
\label{thm:Geo main thm}
Assume \eqref{eq:parameter domain geo} and  let  $\alpha\in\ZZ\cup\{-\infty\}$, $\psi\in\ZZ$.  As
$M\to\infty$:
\begin{enumerate}

\item if $0<r,s<1$, then
\be\label{eq:asymp of K geo MC all}
 \mathcal{K}_{M}^{a,r,s}(\alpha,\psi)
 \sim 
 \frac{(1-rs)(1-a)^3s^\psi}
 {2\sqrt\pi\,a^{3/2}(1-r)^2(1-s)^2}
 \mathcal N^{r,1}(\alpha,\psi)
 M^{-3/2} \rho ^M;
\ee

\item if $s=1$ and $0<r<1$, then
\be\label{eq:asymp of K geo r axis all}
 \mathcal K_M^{a,r,1}(\alpha,\psi)
 \sim 
 \frac{1-a}{(1-r)\sqrt{\pi a}}
 \mathcal N^{r,1}(\alpha,\psi)
 M^{-1/2} \rho ^M;
\ee

\item if $r=1$ and $0<s<1$, then
\be\label{eq:asymp of K geo s axis all}
 \mathcal K_M^{a,1,s}(\alpha,\psi)
 \sim 
  \frac{1-a}{(1-s)\sqrt{\pi a}}
 \mathcal N^{s,1}(-\infty,\psi)
 M^{-1/2} \rho ^M;
\ee

\item if $s>\max\{1,r\}$  and $rs\ne1$, then
\be\label{eq:asymp of K geo HD all}
 \mathcal{K}_{M}^{a,r,s}(\alpha,\psi)
 \sim 
 \frac{1-s^{-2}}{1-r/s}\mathcal N^{r,s}(\alpha,\psi)  \rho ^M;
\ee

\item if $r>\max\{1,s\}$  and $rs\ne1$, then
\be\label{eq:asymp of K geo LD all}
 \mathcal{K}_{M}^{a,r,s}(\alpha,\psi)
 \sim 
 \frac{1-r^{-2}} {1-s/r}\mathcal N^{s,r}(-\infty,\psi) \rho ^M;
\ee

\item if $r=s >1$,  then
\be\label{eq:asymp of K geo CE all}
 \mathcal K_M^{a,r,r}(\alpha,\psi)
 \sim \frac{a(r^2-1)^2}{r^2(r-a)(1-ar)}r^{2\psi}
 M \rho ^M.
\ee

\item if $rs=1$, in particular when $r=s=1$, then
$\mathcal{K}_{M}^{a,r,s}(\alpha,\psi)=1$ for all $M\geq1$.
\end{enumerate} 

\end{theorem}
Theorem \ref{thm:Geo main thm} is proven in Section \ref{subsec:proof of the theorems Geo}. 

\subsection{The boundary dominated region}
\subsubsection{The high density phase}
Originally due to   \cite{Dufresne},  Dufresne's   identity is generalized to the process level in \cite{MatsumotoYor_Dufresne}.  
We start with a discrete, zero temperature counterpart, which is studied in  \cite{DraiefMairesseOConnell} in  the context of queueing theory.   

\begin{lemma}
\label{lem:geo Dufresne} 
Let $0<p<q<1$.  Let $A,S,Y$ be independent with
$$
 A\sim\PP_{\Geo}^{q},\qquad S\sim\PP_{\Geo}^{p},
 \qquad Y\sim\Geo(p/q).
$$ 
Then
\be\label{eq:joint dufresne}
 \left(\mathcal{T} ^{Y} (S,A), Y\right)
 \overset{\mathrm d}=
 \left( (A,S), \mathcal A_{\infty}(S,A)) \right).
\ee
As a consequence,
\begin{equation}\label{eq:process Dufresne}
 \left( \mathcal A_n(S,A)\right)_{n\geq0} 
 \overset{\mathrm d}=
 \left(Y\wedge\mathcal A_n(A,S)) \right)_{n\geq0},
\end{equation}
and
\begin{equation} \label{eq:Geo Dufresne}
 \mathcal A_M(S,A)\uparrow \mathcal A_{\infty}(S,A)\sim\Geo(p/q).
\end{equation}
\end{lemma}

\begin{proof}
\cite[Theorem 4.4]{DraiefMairesseOConnell} gives the identity in law between processes:
\begin{equation}\label{eq:DMO}
   \mathcal{T}^Y(S,A)
 \overset{\mathrm d}=(A,S).
\end{equation}
Taking deterministic functional $ L\mapsto (\mathcal A_n( L_2,L_1))_{n\geq0}$ on both sides of \eqref{eq:DMO}, \eqref{eq:process Dufresne} follows by the property   \eqref{eq:capping Geo}.
Since $p<q$,  $\mathcal{A}_{n}(A,S)\to+\infty$ by strong law of large numbers,  hence $\PP_{\Geo}^{q}\times \PP_{\Geo}^{p} $-almost surely,
$$ \mathcal A_\infty
 \bigl([\mathcal{T}^Y(S,A)]_2,
 [\mathcal{T}^Y(S,A)]_1\bigr)=Y. $$
Taking deterministic functional $ L\mapsto (L, \mathcal A_{\infty}( L_2,L_1))$ on both sides of \eqref{eq:DMO}, \eqref{eq:joint dufresne} follows by the property of \eqref{eq:capping Geo}.
The second coordinate marginal of equality in distribution \eqref{eq:joint dufresne} gives 
\eqref{eq:Geo Dufresne}.  
\end{proof}

Let us prove the asymptotics of 
$\mathcal{K}_{M}^{a,r,s}(\alpha,\psi)$ 
in the high density phase.
\begin{proposition}
\label{prop:HD Geo} 
Assume \eqref{eq:parameter domain geo} and $s>\max\{1,r\}$. Then for all  
$\alpha\in\ZZ\cup\{-\infty\} $ and $\psi\in\ZZ$, 
\begin{equation} \label{eq:asymp of K HD Geo}
 \mathcal{K}_{M}^{a,r,s}(\alpha,\psi)
 \sim 
 \frac{1-s^{-2}}{1-r/s}\mathcal N^{r,s}(\alpha,\psi)  \left(\frac{1-ar}{1-a/s}\right)^M.
\end{equation} 
\end{proposition}

\begin{proof}
When $a/s<1$, changing the reference measure of $L_2$ from $\PP_{\Geo,M}^{ar}$ to $\PP_{\Geo,M}^{a/s}$, we have 
$$ \mathcal{K}_{M}^{a,r,s}(\alpha,\psi)
  =
 \left(\frac{1-ar}{1-a/s}\right)^M
 \EE_{\Geo,M}^{as,a/s}
 \left[(rs)^{\a\vee(\psi+\mathcal{A}_{M}(L_2,L_1))}\right]. 
  $$
Apply Lemma  \ref{lem:geo Dufresne} with
$A=L_1$, $S=L_2$, $p=a/s$, and $q=as$, we have 
$$\mathcal A_M(L_2,L_1)\uparrow \mathcal A_\infty (L_2,L_1)\sim\Geo(s^{-2}).$$ 
 Using monotone convergence when $rs>1$ and dominated convergence when $rs\leq 1$, we obtain \eqref{eq:asymp of K HD Geo} by noticing that, for $s>\max\{r,1\}$ and $G\sim\Geo(s^{-2})$, 
$$
  \EE 
 \left[(rs)^{\a\vee(\psi+ G  )}\right]=\frac{1-s^{-2}}{1-r/s}\mathcal N^{r,s}(\alpha,\psi) . 
 $$
\end{proof}

\subsubsection{The low density phase}
Inspired by the argument in the proof of \cite[Theorem 2.2 (iii)]{HariyaYor} which relies on a Dufresne's relation between Brownian motions with different drift, let us start with proving a simple time reversal lemma and use it to prove the asymptotics of the partition function in the low density phase.
\begin{lemma}[Time reversal]
\label{lem:time reversal Geo} 
Consider two independent random walks $A$ and $S$ in $\mathcal{P}$. 
Then  for all $M\geq 0$, we have 
\begin{equation}\label{eq:increments time reversal}
\bigl(\mathcal A_M(A,S),A(M)-S(M)\bigr) 
 \overset{\mathrm d}=\bigl(A(M)-S(M)+\mathcal A_M(S,A),A(M)-S(M)\bigr)
 .
\end{equation}
\end{lemma}

\begin{proof}
Let $M\geq 1$. For $0\leq i\leq M$, define
$$\overleftarrow A(i)=A(M)-A(M-i),\quad \overleftarrow S(i)=S(M)-S(M-i).$$ 
Then 
$$
 \mathcal A_M(\overleftarrow A,\overleftarrow S)
 =A(M)-S(M)+\mathcal A_M(S,A).
$$
Since $A$ and $S$ are random walks in $\mathcal{P}$, they have  i.i.d. increments and start from $0$, hence 
$$\left(  A(i), S(i)\right)_{0\leq i\leq M}
\overset{\mathrm d}{=}
\left( \overleftarrow A(i), \overleftarrow S(i)\right)_{0\leq i\leq M}.
$$
In particular, 
$$\left(\mathcal A_M(A,S),A(M)-S(M)\right)
\overset{\mathrm d}{=}
\left(\mathcal  A_M( \overleftarrow A,\overleftarrow S),\overleftarrow A(M)-\overleftarrow S(M)\right),
$$
which is \eqref{eq:increments time reversal}.
\end{proof}
Let us prove the asymptotics of 
$\mathcal{K}_{M}^{a,r,s}(\alpha,\psi)$ 
in the low density phase.
\begin{proposition}
\label{prop:LD Geo} 
Assume  \eqref{eq:parameter domain geo} and $r>\max\{1,s\}$. Then for all 
$\alpha\in\ZZ\cup\{-\infty\}$ and $  \psi\in\ZZ$,
\begin{equation}
 \mathcal{K}_{M}^{a,r,s}(\alpha,\psi)
 \sim \frac{1-r^{-2}} {1-s/r}\mathcal N^{s,r}(-\infty,\psi) \left(\frac{1-as}{1-a/r}\right)^M
 \label{eq:asymp of K LD}
\end{equation} 
\end{proposition}

\begin{proof}
Denote $d=(\alpha-\psi)_+$. Since $r>1$ and in particular $a/r<1$,   we have 
\begin{align*}
    \mathcal{K}_{M}^{a,r,s}(\alpha,\psi)
  &=
 (rs)^{\psi}
 \left(\frac{1-as}{1-a/r}\right)^M
 \EE_{\Geo,M}^{a/r,ar}
 \left[
 (rs)^{L_1(M)-L_2(M)+
 (d\vee\mathcal{A}_{M}(L_2,L_1))}
 \right]\\
 &=(rs)^{\psi}
 \left(\frac{1-as}{1-a/r}\right)^M
 \EE_{\Geo,M}^{a/r,ar}
 \left[
 (rs)^{(L_1(M)-L_2(M)+d) \vee\mathcal{A}_{M}(L_1,L_2)\}}
 \right],
\end{align*}
where we change the reference measure of $L_1$ from $\PP_{\Geo,M}^{as}$ to $\PP_{\Geo,M}^{a/r}$ in the first equality and apply  Lemma \ref{lem:time reversal Geo}   with
$A=L_2\sim\PP_{\Geo}^{ar} $, $S=L_1\sim\PP_{\Geo}^{a/r}$ in the second equality. 

Since $r>1$, $L_1(M)-L_2(M)\rightarrow-\infty$ almost surely. Hence by the Dufresne's identity \eqref{eq:Geo Dufresne}, we have 
$$(L_1(M)-L_2(M)+d) \vee\mathcal{A}_{M}(L_1,L_2) \xrightarrow[M\to\infty]{a.s.}\mathcal{A}_{\infty}(L_1,L_2)\sim\Geo(r^{-2}).$$
Observe that for all $M\geq 1$ we have
$$ 0\leq (L_1(M)-L_2(M)+d) \vee\mathcal{A}_{M}(L_1,L_2) 
\leq A_{\infty}(L_1,L_2)+d.
$$  
Using the left inequality for $rs\leq 1$ and   the right inequality for  $rs>1$,  
\eqref{eq:asymp of K LD} follows from the dominated convergence theorem and the fact that for 
 $G\sim\Geo(r^{-2})$, 
$$
(rs)^{\psi}\EE\left[(rs)^G\right]
=\mathcal{N}^{r,s}(-\infty,\psi)\frac{1-r^{-2}}{1-s/r}.
$$ 
\end{proof}

\subsubsection{The coexistence line}
In order to see the desired $M\rho^M$ behavior of the partition function on the coexistence line, we prove that the successive differences of the sequence $(\rho ^{-M}
\mathcal{K}_{M}^{a,r,r})_{M\geq1}$ converge to some constant. Let us first prove a preliminary lemma.

\begin{lemma} 
\label{lem:Geo K recursion} 
Assume \eqref{eq:parameter domain geo} and  $a/s<1$, denote  
$\varrho=\frac{1-ar}{1-a/s}$.  Then for all 
$\alpha\in\ZZ\cup\{-\infty\}, \psi\in\ZZ $  and
$M\geq 1$, we have 
\begin{equation}
 \mathcal{K}_{M+1}^{a,r,s}(\alpha,\psi)
 -
 \varrho\mathcal{K}_{M}^{a,r,s}(\alpha,\psi)
 =
 \left(1-
 \varrho\right)
 \left(\frac{a}{s}\right)^{-\psi}
 \left(\frac{1-as}{1-a^2}\right)^M
 \mathcal{K}_{M}^{a,r,a}(\alpha,\psi).
 \label{eq:Geo K recursion}
\end{equation}
\end{lemma}

\begin{proof}
Let 
$\mathcal F_M= \sigma\Big( L_1(i),L_2(i): 0\leq i\leq M \Big)$.   
Apply  \eqref{eq:cutting geo} with  $N=M+1$ and $k=M$, we have 
$$
 L_2(M+1)-
 \mathcal{A}_{M+1}^{\alpha,\psi}(L_2,L_1)
 = \min\{L_2(M+1)-L_2(M)-\alpha_M,-\psi_M\} ,
$$
where 
$$
 \alpha_M=-\min\{ L_2(M)-\a, (L_1\otimes L_2)(M)-\psi\},
 \qquad \psi_M=\psi-L_1(M),
$$
are measurable with respect to $\mathcal{F}_M$. 
By Markovianity of the random walk,  $L_2(M+1)-L_2(M)\sim\Geo(ar)$ and  is independent
of $\mathcal F_M$. Hence  
\begin{align*}
\EE\left[
(rs)^{-L_2(M+1)+ \mathcal{A}_{M+1}^{\alpha,\psi}(L_2,L_1)}
\big|
\mathcal F_M\right]=(rs)^{\alpha_M}
\left(\varrho+(1-\varrho)
\left(\frac as\right)^{\a_M-\psi_M}
\right),
\end{align*}
Taking expectation with respect to $\EE_{\Geo,M}^{as,ar}$, we obtain
\begin{align*}
    \mathcal{K}_{M+1}^{a,r,s}(\alpha,\psi)
&=\EE_{\Geo,M}^{as,ar} \left[(rs)^{\alpha_M}
 \left(
 \varrho+(1-\varrho)
 \left(\frac as\right)^{\a_M-\psi_M}
 \right)\right].\\
 &=\varrho\mathcal{K}_{M}^{a,r,s}(\alpha,\psi) 
 +(1-\varrho)
 \left(\frac as\right)^{-\psi}
 \EE_{\Geo,M}^{as,ar} \left[
 \left(\frac as\right)^{L_1(M)}
 (ar)^{-L_2(M)+
 \mathcal{A}_{M}^{\alpha,\psi}(L_2,L_1)}
 \right].\\
 &=\varrho\mathcal{K}_{M}^{a,r,s}(\alpha,\psi) +
 (1-\varrho)
 \left(\frac{a}{s}\right)^{-\psi}
 \left(\frac{1-as}{1-a^2}\right)^M
 \mathcal{K}_{M}^{a,r,a}(\alpha,\psi).
\end{align*}
where we change the reference measure of $L_1$ from $\PP_{\Geo,M}^{as}$ to $\PP_{\Geo,M}^{a^2}$ in the last equality.  
\end{proof}
 
We are ready to prove the asymptotics of 
$\mathcal{K}_{M}^{a,r,s}(\alpha,\psi)$ 
on the coexistence line.
\begin{proposition}
\label{prop:Geo coexistence asymptotics} 
Assume \eqref{eq:parameter domain geo} and $r>1$. Then for all 
$\alpha\in\ZZ\cup\{-\infty\}$ and 
$ \psi\in\ZZ$,
\begin{equation}\label{eq:critical line asymptotics}
 \mathcal{K}_{M}^{a,r,r}(\alpha,\psi)
 \sim r^{2\psi}
 \frac{a(r^2-1)^2}{r^2(r-a)(1-ar)}
 M \left(\frac{1-ar}{1-a/r}\right)^M.
\end{equation} 
\end{proposition}

\begin{proof} 
Apply Lemma \ref{lem:Geo K recursion} with $s=r$. Recall that $\rho=\frac{1-ar}{1-a/r}$ on the coexistence line. 
Divide both sides of    \eqref{eq:Geo K recursion} by
$\rho^{M+1}$, we obtain \be\label{eq:Geo critical to LD} \rho ^{-(M+1)}
 \mathcal{K}_{M+1}^{a,r,r}(\alpha,\psi)
 - \rho ^{-M}
 \mathcal{K}_{M}^{a,r,r}(\alpha,\psi)
=\frac{a(r^2-1)}{r(1-ar)}
\left(\frac{a}{r}\right)^{-\psi}
\left(\frac{1-a^2}{1-a/r}\right)^{-M}
\mathcal{K}_{M}^{a,r,a}(\alpha,\psi).
\ee
Since $(r,a)$ is in the strict low-density phase, we can apply Proposition \ref{prop:LD Geo}  with
$(r,s)=(r,a)$. The right hand side of
\eqref{eq:Geo critical to LD} converges to the constant 
$$
 r^{2\psi}
 \frac{a(r^2-1)^2}{r^2(r-a)(1-ar)}.
$$ 
Asymptotics \eqref{eq:critical line asymptotics} follow by the  Stolz-Cesaro theorem.
\end{proof}

\subsection{The fan region}
\subsubsection{Preliminaries about Schur functions}
\label{subsubsec:schur preliminary}
We follow the notation in \cite[Section 2.1]{Barraquand}.
 For $n\geq1$, let
$$
 \Sign_n
 =\{\lambda=(\lambda_1,\ldots,\lambda_n)\in\ZZ^n:
 \lambda_1\geq\cdots\geq\lambda_n\},
 \qquad
 \Sign_n^+=\{\lambda\in\Sign_n:\lambda_n\geq0\}.
$$
For $\lambda,\mu\in\Sign_n$, we say $\la$ interlaces with $\mu$, denoted as $\lambda\prec\mu$ if 
$$
 \mu_1\geq\lambda_1\geq\mu_2\geq\lambda_2
 \geq\cdots\geq\mu_n\geq\lambda_n,
$$
For $\la \prec\mu$, the skew Schur polynomials are defined in a single variable $x\in\RR$   by  
$$
 s_{\mu/\lambda}(x)
 =\mathbf 1_{\{\lambda\prec\mu\}}x^{|\mu|-|\lambda|}.
$$
and in multiple variables $x_1,\ldots,x_m\in\RR$ by the branching rule
\be\label{eq:Schur branching} 
 s_{\mu/\lambda}(x_1,\ldots,x_m)
 =\sum_{\nu\in\Sign_n}
 s_{\mu/\nu}(x_1,\ldots,x_{m-1})s_{\nu/\lambda}(x_m).
\ee
For $\lambda\in\Sign_n^+$, write
$s_\lambda(x)=s_{\lambda/0^n}(x)$.  For 
$x=(x_1,\ldots,x_n)$ and $y=(y_1,\ldots,y_m)$, let
$$
\Pi(x;y)=\prod_{i=1}^n\prod_{j=1}^m\frac{1}{1-x_i y_j}.
$$
For specializations $x,y$ with at most $n$ variables and $\mu\in\Sign_n^+$, 
we have
\begin{itemize}
    \item  Cauchy identity
    \be\label{eq:Schur Cauchy}
 \sum_{\lambda\in\Sign_n^+}s_\lambda(x)s_\lambda(y)
 =\Pi(x;y).
\ee
\item Skew-Cauchy identity
\be\label{eq:Schur skew Cauchy}
 \sum_{\la\in\Sign_n^+}s_{ \lambda/\mu}(x)s_{\la}(y)
 =s_{\mu}(y)\Pi(x;y).
 \ee
\end{itemize} 
The function $\mathcal N^{r,s}(\a,\psi)$ defined in \eqref{eq:function N Geo} admits the following representation for $\alpha\in\ZZ\cup\{-\infty\},  \psi\in\ZZ$ and $ s\in[1,\frac{1}{r}) $:
\be\label{eq:function N geo as sum of Schur}
\mathcal N^{r,s}(\a,\psi)=\frac{(rs)^\psi}{\Pi(r;(s,1/s))} 
\sum_{l\in\ZZ} r^l \mathsf{1}_{l\geq (\a-\psi)_+}s_{(l,0)}(s,s^{-1}).
\ee

\subsubsection{The two-layer Schur process} 
 
In the fan region $rs<1$, the stationary measure  $\PP_{\Geo,N}^{a,r,s}$ \eqref{eq:Stationary measure strip Geo} has a Markovian description, which can be studied through   integral formulas developed in \cite{Barraquand}. 
\begin{definition}[Two-layer Schur process]
Let $N\in\NN$. Assume the parameters    $a,r,s$ satisfy \eqref{eq:parameter domain geo} and  $ rs<1$.   Define the   probability measure $\PP_{\Schur,N}^{a,r,s}$ on $\bl=(\la^0,\cdots,\la^N)\in(\Sign^+_2)^{N+1}$   by 
    \be
\PP_{\Schur,N}^{a,r,s}(\bl)
     = \frac{1}{\mathcal{Z}^{a,r,s}_{\Geo,N}}
    r^{\la^0_1-\la_2^0} \mathbf{1}_{\la_2^0=0}  \prod_{i=1}^{N} (1-a )^2 s_{\la^i/\la^{i-1}}(a ) 
     s^{\la_1^N-\la_2^N},
    \ee
    where   $\mathcal{Z}^{a,r,s}_{\Geo,N}$ is the normalizing constant.
\end{definition}
The two-layer Schur process is a Markov process. 
\begin{lemma}\cite[Proposition 2.10]{Barraquand}
\label{lem:Markovian description r,s<1} 
    Under $\PP_{\Schur,N}^{a,r,s}$, the process $(\la^x)_{0\leq x\leq N}$ is an inhomogeneous Markov chain on $\Sign^+_2$ with initial condition and transition probability given by 
    \be\label{eq:initial distribution and transition kernel r,s<1}
    \P^{r,s}_{N}(\la)=r^{\la_1-\la_2} \mathrm{1}_{\la_2=0}   h_{0,N}^{r,s}(\la_1-\la_2), \quad
    \U_{x,N}^{r,s}(\la,\mu)=(1-a)^2 s_{\mu/\la}(a) 
    \frac{h^{r,s}_{x+1,N}(\mu_1-\mu_2)}{h^{r,s}_{x,N}(\la_1-\la_2)},
    \ee
where the functions $(h_{x,N})_{0\leq x\leq N}: \ZZ_+\rightarrow \RR_{+}$ are defined by 
\be\label{eq:def of function h and q(l)}
h_{x,N}^{r,s}(l)=\frac{q^{s}_{N-x}(l)}{\sum_{l\geq 0}r^l q^{s}_N(l)}, \quad  \text{ with } \quad 
q_{N}^{s}(l)=(1-a)^{2N} \sum_{\la\in\Sign_2}s_{\la/(l,0)}((a)^N) 
    s^{\la_1-\la_2}.
\ee
\end{lemma}

When  $rs<1$, the two-layer Schur process  is related to the stationary measure $\PP_{\Geo,N}^{a,r,s}$ \eqref{eq:Stationary measure strip Geo} by \cite[Proposition 2.17]{BarraquandCorwinYang}: 
\be\label{eq:relation between reweighted rw and two layer Schur process}
     (L_1(x),L_2(x))_{0\leq x\leq N}\overset{d}{=}(\la^x_1-\la^0_1,\la^x_2-\la^0_2)_{0\leq x\leq N}, 
\ee
 where $\bl\sim\PP_{\Schur,N}^{a,r,s}$ and $\bL\sim \PP_{\Geo,N}^{a,r,s}$. 
Let us prove a relationship between the function  $q^{s}_{N}(l)$ in \eqref{eq:def of function h and q(l)}  and our  partition function  $\mathcal{K}_{M}^{a,r,s}(\alpha,\psi)$ \eqref{eq:generalized partition function geo}.  
\begin{lemma}  
\label{lem:K and q geo}
Assume \eqref{eq:parameter domain geo} and  $rs<1$. Then for all $N\geq 1$ and $\alpha\in\ZZ\cup\{-\infty\}, \psi\in\ZZ$, 
\be\label{eq:relation K q}
\mathcal{K}_{M}^{a,r,s}(\alpha,\psi)
  =(1-rs)(rs)^{\psi} 
 \left(\frac{(1-ar)(1-as)}{(1-a)^2}\right)^M
 \sum_{l\in\ZZ}r^l \mathsf{1}_{l\geq (\psi-\alpha)_+}q_M^s(l).
\ee 
\end{lemma}
\begin{proof}
Recall the definition of $q_M^s(l)$ in
\eqref{eq:def of function h and q(l)}.
Develop $s_{\lambda/(l,0)}((a)^M)$ using the branching rule
\eqref{eq:Schur branching} and consider the change of variables
$\lambda^i=(l+L_1(i),L_2(i))$ for $0\leq i\leq M$. We have
\begin{equation*}
q_M^s(l)
=\left(\frac{(1-a)^2}{(1-as)(1-ar)}\right)^M s^l \EE_{\Geo,M}^{as,ar}
\left[(rs)^{-L_2(M)}
\mathsf{1}_{ l\geq \mathcal A_M(L_2,L_1)  }\right].
\end{equation*}
Denote $d=(\psi-\alpha)_+$, by Tonelli's theorem we have,
\begin{align*}
\sum_{l=d}^{\infty} r^l q_M^s(l)
&=\left(\frac{(1-a)^2}{(1-as)(1-ar)}\right)^M
\EE_{\Geo,M}^{as,ar}\left[
(rs)^{-L_2(M)}
\sum_{l\geq d\vee\mathcal A_M(L_2,L_1)}(rs)^l\right]\\
&=\frac1{1-rs}
\left(\frac{(1-a)^2}{(1-as)(1-ar)}\right)^M
\EE_{\Geo,M}^{as,ar}\left[
(rs)^{-L_2(M)+(d\vee\mathcal A_M(L_2,L_1))}\right]\\
&=\frac{(rs)^{-\psi}}{1-rs}
\left(\frac{(1-a)^2}{(1-as)(1-ar)}\right)^M
\mathcal K_M^{a,r,s}(\alpha,\psi).
\end{align*} 
\end{proof}
 
In view of Lemma \ref{lem:K and q geo},  the asymptotics of the  partition function therefore depend on that of the function $q_M^s(l)$. Our starting point is its  integral representation  developed in \cite{Barraquand}.
For $R>0$, we denote $\gamma_R$ to be  a circle of radius $R$ around the origin. 

\begin{proposition}\cite[Proposition 2.12]{Barraquand} \label{prop:barraquands integral formula for q_M^s} 
    Let $a\in(0,1)$, $M\geq 1$ and $l\geq0$. When $s\in(0,1)$, 
 \be\label{eq:integral representation of q_M(l)}
q_M^s(l)=\oint_{\gamma_1} \frac{\d z}{2\pi i }  
F^s_{M,l}(z),
\ee
where 
\be\label{eq:F^s_M(z)}
F^s_{M,l}(z)=\frac{1}{2z}\Big( z^{l+1}-\frac{1}{z^{l+1}}\Big)
\Big(\frac{1}{z}-z\Big)
\frac{1}{1-sz}\frac{1}{1-s/z}
 \Big( \frac{(1-a)^2}{(1-az)(1-a/z)}\Big)^M.
\ee
When $s\geq1$, $q_M^s(l)$ is given by the analytic continuation of  \eqref{eq:integral representation of q_M(l)}.  
\end{proposition}

The integral formula \eqref{eq:integral representation of q_M(l)} allows to study  asymptotics for the function $q^s_M(l)$ for $s<1$ and hence those of  the  partition function in the maximal current phase, see \cite[Proposition 2.13] {Barraquand}. Based on this approach,  we first specify the analytic continuation of the function $q^s_M(l)$ in Section \ref{subsubsec:analytic continuation of q}, then study its  large $M$ limit in Section \ref{subsubsec:asymptotics of q} and ultimately arrive at the asymptotics of 
$\mathcal{K}_{M}^{a,r,s}(\alpha,\psi)
$  for all $(r,s)$ in the fan region  in Section \ref{subsubsec:asymptotics of K geo}.  
 
\subsubsection{Analytic continuation of $q^s_M(l)$}
\label{subsubsec:analytic continuation of q}

\begin{lemma}\label{lem:analytic continuation of q_M(l)} 
Let $a\in(0,1)$ and $s\in\RR$ such that $a< s<\frac{1}{a}$. Then for any  $R>0$ satisfying $\max(s,\frac{1}{s})<R<\frac{1}{a}$, we have
\begin{equation}\label{eq:analytic continuation of q_M(l) in s}
\begin{aligned}
q^s_M(l)& = \oint_{\gamma_R} \frac{\d z}{2\pi i }
F^{s}_{M,l}(z)+
\frac{1}{2s}(s^{l+1}-\frac{1}{s^{l+1}})
\Big( \frac{(1-a)^2}{(1-a/s)(1-as)} \Big)^M.
\end{aligned}
\end{equation}
\end{lemma}

\begin{proof} 
We prove that for any fixed  $R$ satisfying $\max(s,\frac{1}{s})<R<\frac{1}{a}$,  
\be\label{eq:analytic continuation of q_M(l)}
q^t_M(l)=\oint_{\gamma_R} \frac{\d z}{2\pi i }
F^t_M(z)-\Res_{z=\frac{1}{t}} F^t_M(z)
\ee
for all $t\in(\frac{1}{R},R)$,  \eqref{eq:analytic continuation of q_M(l) in s} follows by taking $t=s$ and computing the residue  
\be
\Res_{z=\frac{1}{s}} F^{s}_{M,l}(z)
= -\frac{1}{2s}(s^{l+1}-\frac{1}{s^{l+1}})
\Big( \frac{(1-a)^2}{(1-a/s)(1-as)} \Big)^M .
\ee

 Recall the integral representation of $q^t_M(l)$ \eqref{eq:integral representation of q_M(l)}.  
    For $t\in(\frac{1}{R},1)$, \eqref{eq:analytic continuation of q_M(l)} holds by the residue theorem, see Figure \ref{fig:contour shrinking Geo}. 
    When $t\geq 1$, $q^t_M(l)$ is given by the analytic continuation of \eqref{eq:integral representation of q_M(l)} by Proposition \ref{prop:barraquands integral formula for q_M^s}. 
    By uniqueness of the analytic continuation, it remains to justify the two terms on the right hand side of \eqref{eq:analytic continuation of q_M(l)} are analytic for   $t\in(\frac{1}{R},R)$. 
    
    The residue term
    \be t\mapsto
\Res_{z=\frac{1}{t}} F^t_M(z)
= \frac{1}{2t}(\frac{1}{t^{l+1}}-t^{l+1})
\Big( \frac{(1-a)^2}{(1-a/t)(1-at)} \Big)^M
    \ee
    is a rational function of  $t$ without pole in $(a,1/a)$.  Hence it is real analytic on  $(a,1/a)$, and in particular on
$ (\frac{1}{R},R)$.

To argue the real analyticity of the function  $t\mapsto\oint_{\gamma_R} \frac{\d z}{2\pi i }F^t_M(z)$ 
for  $t\in (\frac{1}{R},R)$, it is convenient to regard it as a complex function and  prove it is holomorphic in the open annulus
$$
\mathcal A_R
=
\left\{
t\in\mathbb C:
\frac{1}{R}<|t|<R
\right\}\subset\C.
$$
Observe that   for each fixed $z\in\gamma_R$ , the function $t\mapsto F^{t}_M(z)$ is holomorphic in $t\in\mathcal A_R$: indeed,  the only $t$-dependent denominators are
$1-tz$ and $1-t/z$, and for all $z\in\gamma_R$,
$$
|1-t z|
\geq |t|R-1>0,
\qquad
\left|1-\frac{t}{z}\right|
\geq 1-\frac{|t|}{R}>0.
$$ 
 Hence, by holomorphic parameter dependence of contour integrals, see e.g. \cite[Theorem 1, Section 16]{SidorovFedoriukShabunin}, 
$
t\longmapsto
\oint_{\gamma_R}\frac{\d z}{2\pi i}\,F_M^t(z)
$
is holomorphic on $\mathcal A_R$, and  in particular real analytic on the real interval $(\frac{1}{R},R)$.  
\end{proof}

\begin{figure}
\centering
\begin{subfigure}[b]{0.48\textwidth}
    \centering
    \begin{tikzpicture}[>=stealth,scale=1]
        \draw (0,-0.09)--(0,0.09)
         node[below left,font=\scriptsize]   {$0$}; 
            \draw[->] (-2.55,0)--(2.65,0) node[right] {$\Re z$};
            \draw[thick] (0,0) circle (1.15);
            \draw[thick,dashed] (0,0) circle (2.05);
            \draw[->,thick] (-0.82,0.82)--(-1.42,1.42);
            \draw[thick,fill=white] (0.65,0) circle (1.6pt) 
             node[below,font=\scriptsize]   {$t$};
            \draw[thick,fill=black] (1.55,0) circle (1.6pt) 
             node[below,font=\scriptsize]  {$1/t$};
             \node[fill=white,inner sep=1pt] at (0.72,0.92) {$\gamma_1$};
            \node[fill=white,inner sep=1pt] at (1.48,1.48) {$\gamma_R$};
    \end{tikzpicture} 
    \caption{When $t\in(\frac{1}{R},1)$, enlarging
    $\gamma_1$ to $\gamma_R$ crosses the pole $z=1/t$.}
\label{fig:contour shrinking Geo}
\end{subfigure}
\hfill
\begin{subfigure}[b]{0.48\textwidth}
    \centering
    \begin{tikzpicture}[>=stealth,scale=1]
        \draw (0,-0.09)--(0,0.09)
         node[below left,font=\scriptsize]   {$0$};
            \draw[->] (-2.55,0)--(2.65,0) node[right] {$\Re z$};
            \draw[thick] (0,0) circle (2.05);
            \draw[thick,dashed] (0,0) circle (1.15);
            \draw[->,thick] (-1.42,1.42)--(-0.82,0.82); 
            \draw[thick,fill=white] (0.65,0) circle (1.6pt) 
             node[below,font=\scriptsize]   {$1/s$};
            \draw[thick,fill=black] (1.55,0) circle (1.6pt) 
             node[below,font=\scriptsize]  {$s$};
            \node[fill=white,inner sep=1pt] at (0.72,0.92) {$\gamma_1$};
            \node[fill=white,inner sep=1pt] at (1.48,1.48) {$\gamma_R$}; 
    \end{tikzpicture} 
    \caption{When $s\in(1,\frac{1}{a})$, shrinking
    $\gamma_R$ back to $\gamma_1$ crosses the pole $z=s$.}
\label{fig:contour shrinking back Geo}
\end{subfigure}
\caption{Pole crossed when shrinking the  contour.} 
\end{figure}
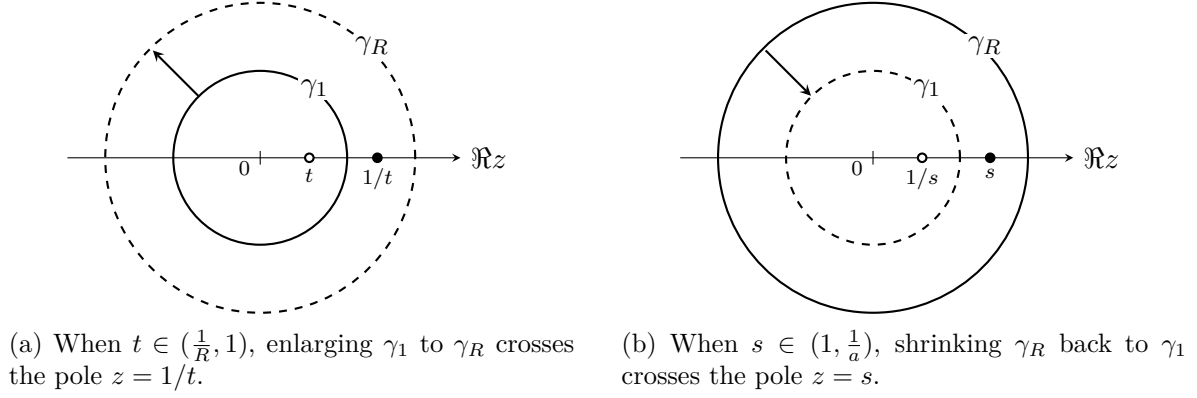

\begin{proposition}\label{prop:general formula for q_M^s for all phases} 
    Let $a\in(0,1)$ and  $s\in(0,\frac{1}{a})$. Then  for all $ M\geq 1 $ and $l \geq0 $ ,
    \begin{equation}\label{eq:q_M^s for s geq 1}
\begin{aligned}
q^s_M(l) 
&= \oint_{\gamma_1} \frac{\d z}{2\pi i }
F^{s}_{M,l}(z)
+\frac{1}{s}(s^{l+1}-\frac{1}{s^{l+1}})
\Big( \frac{(1-a)^2}{(1-a/s)(1-as)} \Big)^M\mathbf{1}_{s\geq 1}.
\end{aligned}
\end{equation} 
\end{proposition}

\begin{proof}

For $s\in(0,1)$, the residue term vanishes hence \eqref{eq:q_M^s for s geq 1} follows by Proposition  \ref{prop:barraquands integral formula for q_M^s}.

When $s\in(1,\frac{1}{a})$,  by Lemma \ref{lem:analytic continuation of q_M(l)}, for all $R\in(\max(s,\frac{1}{s}),\frac{1}{a})$, 
\be
q^s_M(l)=\oint_{\gamma_R} \frac{\d z}{2\pi i }F^{s}_{M,l}(z)-\Res_{z=1/s} F^{s}_{M,l}(z).
\ee
Shrinking the contour from $\ga_R$ back to $\ga_1$ crosses the pole at $z=s$, see Figure \ref{fig:contour shrinking back Geo} for illustration. Hence by the residue theorem, 
\be
q^s_M(l)=\oint_{\gamma_1} \frac{\d z}{2\pi i }
F^{s}_{M,l}(z)-\Res_{z=\frac{1}{s}} F^{s}_{M,l}(z) + \Res_{z=s} F^{s}_{M,l}(z)
\ee
The residue at $z=s$ is  
\be
\Res_{z=s} F^{s}_{M,l}(z)
= \frac{1}{2s}(s^{l+1}-\frac{1}{s^{l+1}})
\Big( \frac{(1-a)^2}{(1-a/s)(1-as)} \Big)^M = -\Res_{z=\frac{1}{s}} F^{s}_{M,l}(z).
\ee

For $s=1$, by Proposition \ref{lem:analytic continuation of q_M(l)}, we know that for any $R\in(1,\frac{1}{a})$, 
   $$q^1_M(l) =\oint_{\gamma_R} \frac{\d z}{2\pi i }
F^1_{M,l}(z). $$
We claim that the contour can be deformed back to $\ga_1$ without changing the value of the integral.
Indeed, when $s=1$, \eqref{eq:F^s_M(z)}  rewrites as  
\be
 F^1_{M,l}(z)=\frac{(z+1)^2}{2z^{l+2}}\Big(\sum_{i=0}^{l}z^{2i}\Big) 
\Big( \frac{(1-a)^2}{(1-a/z)(1-az)} \Big)^M,
\ee
which is analytic on the annulus  
$\{z: a<\abs{z}<\frac{1}{a}\} $  and in particular in a neighborhood of the closed annulus $\{z: 1\leq \abs{z}\leq R\}$. The claim therefore follows from the residue theorem.  This proves \eqref{eq:q_M^s for s geq 1}, whose residue  term   vanishes at $s=1$.
\end{proof}

Let us deduce the asymptotics of $q^s_M(l)$ using representations in Proposition \ref{prop:general formula for q_M^s for all phases}.

\subsubsection{Asymptotics of $q^s_M(l)$}
\label{subsubsec:asymptotics of q}
Fix $a\in(0,1)$ and $l\in\ZZ_{\geq0}$,  define a function $ G^s_{M,l}(x)$ for  $M\geq 1, s>0$ and $x\in\RR$ by
\begin{multline}
     G^s_{M,l}(x):=   
    \frac{\mathbf{1}_{\abs{x}\leq  \pi \sqrt{M} }   }{  \pi \sqrt{M}} \sin\left(\frac{(l+1)x}{\sqrt{M}}\right) 
     \sin\left(\frac{x}{\sqrt{M}}\right) \\
     \left((1-s)^2+ 2s\big(1-\cos\frac{x}{\sqrt{M}}\big)  \right)^{-1}
     \left( 1+\frac{2a\big(1-\cos\frac{x}{\sqrt{M}}\big)}{(1-a)^2}  \right)^{-M}.
\end{multline} 
Let us first prove some  upper bounds for later use in dominated convergence. 
\begin{lemma}\label{lem: upper bound for function G} 
Let $a\in(0,1)$. There exist constants $c=c(a), C=C(a)$ such that for all $ M\geq 1 $, $l \in\ZZ_{\geq0} $ and   $x\in \RR$,  
\be\label{eq:upper bound for function G s not 1}
 \abs{M^{3/2}G^s_{M,l}(x)}  \leq 
 \frac{(l+1)}{  \pi  (1-s)^2 }  
 C x^2e^{-cx^2}, \qquad s\neq 1,
\ee 
and
    \be\label{eq:upper bound for function G s=1}
\abs{\sqrt{M}G^1_{M,l}(x)}\leq 
\frac{\pi(l+1)}{2 }Ce^{-cx^2}. 
\ee
\end{lemma}
\begin{proof}
       Using $\abs{\sin y}\leq \abs{y} $ and $1-\cos y\geq 2y^2/\pi^2$ for $y\in[-\pi,\pi]$, we have  
$$
 \abs{G^s_{M,l}(x)}  \leq 
    \frac{  \mathbf{1}_{\abs{x}\leq  \pi \sqrt{M} } }{  \pi \sqrt{M}(1-s)^2} \frac{(l+1)\abs{x}}{\sqrt{M}} \frac{\abs{x}}{\sqrt{M}}  
     \Big( 1+\frac{4ax^2}{(1-a)^2\pi^2M}\Big)^{-M},
$$
for $s\neq  1$, and
$$
\abs{G^1_{M,l}(x)}\leq 
\frac{\mathbf{1}_{\abs{x}\leq  \pi \sqrt{M} }  }{  \pi \sqrt{M}}  \frac{(l+1)\abs{x}}{\sqrt{M}}  
 \frac{ \abs{x}}{\sqrt{M}}
    \frac{M\pi^2}{2x^2}
\Big( 1+\frac{4ax^2}{(1-a)^2\pi^2M}\Big)^{-M}.
$$
There exist $c=c(a), C=C(a)$ such that for all $M\geq 1$ and all $x\in \RR$, 
\be\label{eq:Gaussian type bound}
\mathbf{1}_{\abs{x}\leq  \pi \sqrt{M} }\Big( 1+\frac{4ax^2}{(1-a)^2\pi^2M}\Big)^{-M}\leq  Ce^{-cx^2}. 
\ee
Bounds \eqref{eq:upper bound for function G s not 1} and \eqref{eq:upper bound for function G s=1} follow by implementing \eqref{eq:Gaussian type bound}. 
\end{proof}

\begin{proposition}
    Let $a\in(0,1)$. Then
\begin{itemize}
        \item For $s\in(0,1)$, there exists constant $C=C(a,s)>0$ such that for all $l\in\ZZ_{\geq0}$, 
   \be\label{eq:upper bound of q for s<1}
\abs{q_M^s(l)}\leq  \frac{C(l+1)}{M^{3/2}}, \qquad \forall M\geq 1,
   \ee
   and
\be\label{eq:asymptotics of q^s_M(l) for s<1}
q_M^s(l)
 \underset{M\to\infty}{\sim}
 \frac{ (l+1)}{2\sqrt{\pi}(1-s)^2}
\Big( \frac{(1-a)^2}{aM}\Big)^{3/2}.
\ee 

\item For $s=1$, there exists constant $C=C(a)>0$ such that for all $l\in\ZZ_{\geq0}$,
\be\label{eq:upper bound of q for s=1}
|q_M^1(l)| \leq \frac{C(l+1)}{M^{1/2}}, \qquad \forall M\geq 1,
\ee
and
\be\label{eq:asymptotics of q^s_M(l) for s=1}
q_M^1(l)
 \underset{M\to\infty}{\sim}
 \frac{(l+1)}{\sqrt{\pi}}\Big(\frac{(1-a)^2}{aM}\Big)^{1/2}.
\ee

\item For $s\in(1,1/a )$, there exists constant $C=C(a,s)>0$ such that for all $l\in\ZZ_{\geq0}$,
   \be\label{eq:upper bound of q for s>1}
 | q_M^s(l)  | 
\leq C(a,s) s^l
\Big( \frac{(1-a)^2}{(1-a/s)(1-as)} \Big)^M , \qquad \forall M\geq 1,
   \ee
and 
\be\label{eq:asymptotics of q^s_M(l) for s>1}
q_M^s(l)
 \underset{M\to\infty}{\sim}
 \frac{1}{s}
 \Big( s^{l+1}-\frac{1}{s^{l+1}} \Big)
\Big( \frac{(1-a)^2}{(1-a/s)(1-as)} \Big)^M.
\ee
    \end{itemize}
\end{proposition}

\begin{proof}
For all $s>0$, consider the change of variable $z=e^{ix/\sqrt{M}}$,  we have 
    \be 
\oint_{\gamma_1} \frac{\d z}{2\pi i }
F^{s}_{M,l}(z)= \int_{\RR}G^s_{M,l}(x) \d x.
    \ee
Take the constants $C,c$ from Lemma \ref{lem: upper bound for function G},  when $s\neq 1$, we have 
  \be\label{eq:inequality for unit cercle integral s not 1}
    \left| \oint_{\gamma_1} \frac{\d z}{2\pi i }
F^{s}_{M,l}(z)\right|
\leq \int_{\RR} \left| G^s_{M,l}(x)\right| \d x
\leq \int_{\RR}
\frac{(l+1)}{  \pi  (1-s)^2 M^{3/2}}  
 C x^2e^{-cx^2}  \d x 
 \leq \frac{C(a,s)(l+1)}{M^{3/2}} 
 \ee
for some constant $C(a,s)$. And when $s=1$, we  have 
 \be\label{eq:inequality for unit cercle integral s=1}
    \left| \oint_{\gamma_1} \frac{\d z}{2\pi i }
F^1_{M,l}(z)\right|
\leq \int_{\RR} \left| G^1_{M,l}(x)\right| \d x
\leq \int_{\RR}
\frac{\pi(l+1)}{2 M^{1/2}}Ce^{-cx^2}  \d x 
 \leq \frac{C(a)(l+1)}{M^{1/2}} 
 \ee
 for some constant $C(a)$. 
 
When $s\in (0,1)$, the only contribution in   \eqref{eq:q_M^s for s geq 1} is the integral term.  Hence \eqref{eq:upper bound of q for s<1} follows directly by inequality \eqref{eq:inequality for unit cercle integral s not 1}. 
Observe that when $s\neq1$, we have pointwise convergence   for all $x\in\RR$    
$$
M^{3/2}G^s_{M,l}(x)\xrightarrow[M\to\infty]{}
\frac{l+1}{\pi(1-s)^2} x^2e^{-\frac{ax^2}{(1-a)^2}}. 
$$
We can apply the dominated convergence theorem thanks to the bound \eqref{eq:upper bound for function G s not 1},  then    \eqref{eq:asymptotics of q^s_M(l) for s<1} follows by a   Gaussian integral.

When $s=1$, the only contribution in   \eqref{eq:q_M^s for s geq 1} is the integral term since the residue term is $0$.  Hence \eqref{eq:upper bound of q for s=1} follows directly by inequality \eqref{eq:inequality for unit cercle integral s=1}. 
Observe that when $s=1$, we have pointwise convergence   for all $x\in\RR$    
$$
 M^{1/2}G^s_{M,l}(x)\xrightarrow[M\to\infty]{} 
     \frac{l+1}{\pi}  e^{-\frac{ax^2}{(1-a)^2}}. 
$$
We can apply the dominated convergence theorem thanks to the bound \eqref{eq:upper bound for function G s=1},  then    \eqref{eq:asymptotics of q^s_M(l) for s=1} follows by a   Gaussian integral.

When $s\in(1,\frac{1}{a})$, the integral term is bounded by \eqref{eq:inequality for unit cercle integral s not 1},   hence  the contribution in  \eqref{eq:q_M^s for s geq 1} is dominated by the residue term which grows exponentially. 
\end{proof}

\subsubsection{Asymptotics of the  partition function }
\label{subsubsec:asymptotics of K geo}
Recall the relation between the function $q$ and  partition function $\mathcal{K}_{M}^{a,r,s}(\alpha,\psi)$ \eqref{eq:relation K q}. 
When $r<1$, the   sum in \eqref{eq:relation K q} is summable in $l$. 
When $r\geq 1$,  inspired by the arguments in the low density phase, we prepare the following lemma.  
 
\begin{lemma}[Time reversal] Assume \eqref{eq:parameter domain geo} and   $rs<1$. 
    Let $\bl \sim \mathbb P_{\mathrm{Schur},M}^{a,r,s}$. 
    Then 
    $\overleftarrow \bl \sim\mathbb P_{\mathrm{Schur},M}^{a,s,r}$, where for $0\leq i\leq N$, 
   \be\label{eq:time reversed two layer schur process}\overleftarrow \bl^i =\left( \la^M_1-\la_2^{M-i},   \la^M_1-\la_1^{M-i}\right).
   \ee
In particular,
\be\label{eq:boundary parameter symmetry}
 \mathcal Z_{\Schur,M}^{a,r,s} =\mathcal Z_{\Schur,M}^{a,s,r}.
\ee
\end{lemma}
\begin{proof}
    Recall the definition of Schur function, for $0\leq i\leq M$ we have 
    $$s_{\overleftarrow\lambda^i/\overleftarrow\lambda^{i-1}}(a)
 =s_{\lambda^{M-i+1}/\lambda^{M-i}}(a).$$
 Hence 
 \begin{align*}
\mathbb P_{\mathrm{Schur},M}^{a,r,s}( \bl )
&=\frac{1}{ \mathcal Z_M^{a,r,s} }
r^{\la^0_1-\la_2^0} \mathsf{1}_{\la_2^0=0}
\prod_{j=1}^{M}(1-a)^2
s_{\lambda^j/\lambda^{j-1}}(a)
 s^{\la^M_1-\la_2^M}\\
&=\frac{1}{ \mathcal Z_M^{a,r,s} }
r^{\overleftarrow\la^M_1-\overleftarrow\la_2^M}
\prod_{i=1}^{M}(1-a)^2
s_{\overleftarrow\lambda^i/\overleftarrow\lambda^{i-1}}(a)
s^{\overleftarrow\la^0_1-\overleftarrow\la_2^0}
 \mathsf{1}_{\overleftarrow\la_2^0=0}
=\frac{\mathcal Z_M^{a,s,r}}{ \mathcal Z_M^{a,r,s} }
\mathbb P_{\mathrm{Schur},M}^{a,s,r}(\overleftarrow \bl ).
\end{align*}
To see the symmetry of $\mathcal Z_{\Schur,M}^{a,r,s} $ with respect to the boundary parameters, observe that $\bl\mapsto\overleftarrow\bl$ is an involution on the space 
$$\{\bl\in(\Sign_2^+)^{M+1} | \lambda^0\prec\lambda^1\prec\cdots\prec\lambda^M, \la^0_2=0\}$$
and in particular a bijection. As the probability vanishes outside the space, we have 
\begin{multline*}
1=\sum_{\bl\in(\Sign_2^+)^{M+1}}\mathbb P_{\mathrm{Schur},M}^{a,r,s}( \bl )
=\sum_{\bl\in(\Sign_2^+)^{M+1}}\frac{\mathcal Z_M^{a,s,r}}{ \mathcal Z_M^{a,r,s} }
\mathbb P_{\mathrm{Schur},M}^{a,s,r}(\overleftarrow \bl )\\
=\frac{\mathcal Z_M^{a,s,r}}{ \mathcal Z_M^{a,r,s} }\sum_{\overleftarrow\bl\in(\Sign_2^+)^{M+1}}
\mathbb P_{\mathrm{Schur},M}^{a,s,r}(\overleftarrow \bl )=\frac{\mathcal Z_M^{a,s,r}}{ \mathcal Z_M^{a,r,s} }.
\end{multline*}

 \end{proof}

We are now ready to prove the asymptotics of  $\mathcal{K}_{M}^{a,r,s}(\alpha,\psi)$  throughout the fan region.  
\begin{proposition}[Fan-region asymptotics]
\label{prop:fan region asymptotics Geo}
Assume \eqref{eq:parameter domain geo} and $rs<1$.  Fix $\alpha,\psi\in\ZZ_{\geq0}$ and put
$d=(\psi-\alpha)_+$.  
If $r,s<1$, then \eqref{eq:asymp of K geo MC all} holds.  
If $s=1$ and $r<1$, then  \eqref{eq:asymp of K geo r axis all}  holds. 
If $r=1$ and $s<1$, then \eqref{eq:asymp of K geo s axis all} holds.  
If $s>1$, then \eqref{eq:asymp of K geo HD all} holds. 
If $r>1$, then\eqref{eq:asymp of K geo LD all} holds. 
\end{proposition}
\begin{proof} 
Recall the relation between $\mathcal{K}_{M}^{a,r,s}(\alpha,\psi)$ and $q_M^s(l)$ \eqref{eq:relation K q}. 
    When $rs<1 $ and $r<1$, thanks to the upper bound \eqref{eq:upper bound of q for s<1} \eqref{eq:upper bound of q for s=1} and \eqref{eq:upper bound of q for s>1}, by applying the dominated convergence theorem, the asymptocis of $\mathcal{K}_{M}^{a,r,s}(\alpha,\psi)$ from that of function  $q_M^s(l)$ \eqref{eq:asymptotics of q^s_M(l) for s<1}  \eqref{eq:asymptotics of q^s_M(l) for s=1} and \eqref{eq:asymptotics of q^s_M(l) for s>1}.

When $rs<1$ and $r\geq 1$,   using  the time reversal \eqref{eq:boundary parameter symmetry}, by Lemma \ref{lem:K and q geo} we have 
$$
\mathcal{K}_{M}^{a,r,s}(\alpha,\psi)
  =(1-rs)(rs)^{\psi} 
 \left(\frac{(1-ar)(1-as)}{(1-a)^2}\right)^M 
 \left(\sum_{l\geq 0} s^l q_M^r(l)- \sum_{l=0}^{d-1}r^l q_M^s(l)\right).
$$
Thanks to the upper bound \eqref{eq:upper bound of q for s<1}  and \eqref{eq:upper bound of q for s=1},  by applying the dominated convergence theorem, the asymptotics of the infinite sum follows from that of function  $q_M^r(l)$ \eqref{eq:asymptotics of q^s_M(l) for s<1}  and \eqref{eq:asymptotics of q^s_M(l) for s=1}. The finite substraction is    $ O(M^{-3/2}) $  due to bound \eqref{eq:upper bound of q for s<1}, hence negligeable compared to the  infinite sum, which is of order $M^{-1/2}$ for $r=1$ and exponentially large when $r>1$. 
\end{proof}

\subsubsection{Infinite two-layer Schur process} 
As a Markov process, the two-layer Schur process is expected to  have a well defined limit as $N\to\infty$  \cite[Remark 1.16]{Barraquand}.  We show that the limit does exist for certain $(r,s)$, and that the possible limits are described by a family of simple stochastic processes 
$\Big(\PP^{r,s}_{\Schur,\infty}\Big)_{s\geq 1, rs<1}$,
which we refer to as the infinite two-layer Schur processes.

\begin{definition}[Infinite two-layer Schur processes]
    Fix $a\in(0,1)$. For $r,s>0$ such that $s\in[1,1/a)$ and 
    $ rs<1$, we define $\PP^{r,s}_{\Schur,\infty}$ as  the distribution of the homogeneous Markov chain on $\Sign_2$ with initial distribution and transition kernel given by
    \be\label{eq:ic and transition kernel of infinite two layer Schur process}
    \P^{r,s}(\la)= r^{\la_1-\la_2} \mathrm{1}_{\la_2=0}
    h^{r,s}(\la_1-\la_2)  ,     \quad 
     \U^{r,s}(\la,\mu)=\frac{s_{\mu/\la}(a)}{\Pi(a;(s,1/s))} 
    \frac{h^{r,s}(\mu_1-\mu_2) }{h^{r,s}(\la_1-\la_2) },
    \ee
    where
    $$
    h^{r,s}(l)=\frac{s_{(l,0)}(s,s^{-1})}{\Pi(r;(s,1/s))}.
    $$
    The Cauchy identity \eqref{eq:Schur Cauchy} implies that $\P^{r,s}$ is a probability distribution and the skew Cauchy identity \eqref{eq:Schur skew Cauchy} implies that $\U^{s}$ is a stochastic kernel. 
\end{definition}

\begin{proposition}\label{prop:two layer Schur converges to infinite two layer Schur}
Let $a,r\in(0,1)$. We have finite-dimensional distributions convergence 
For $s\in(0,1]$, 
    \be
    \PP_{\Schur,N}^{r,s}\xRightarrow[N\to\infty]{}
    \PP_{\Schur,\infty}^{r,1}.
    \ee  
    For $s\in(1,\min(\frac{1}{a},\frac{1}{r}))$, 
    \be
    \PP_{\Schur,N}^{r,s}\xRightarrow[N\to\infty]{}
    \PP_{\Schur,\infty}^{r,s}.
    \ee  
\end{proposition}
\begin{proof}
Since $\PP_{\Schur,N}^{r,s}$ \eqref{eq:initial distribution and transition kernel r,s<1} and $\PP_{\Schur,\infty}^{r,s}$ \eqref{eq:ic and transition kernel of infinite two layer Schur process} are distributions of Markov chains, it suffices to prove that the initial distribution $\P^{r,s}_{N} $ and the  transition kernel $\U_{x,N}^{r,s} $ converges to the corresponding limit for each  $x\geq0$. 

When $s< 1$ or $s=1$, we know the asymptotics   of $q^{s}_{M}(l)$ by  \eqref{eq:asymptotics of q^s_M(l) for s<1} and   \eqref{eq:asymptotics of q^s_M(l) for s=1} . Recall the definition of function $h_{x,N}^{r,s}(l)$   \eqref{eq:def of function h and q(l)}. In both case we have  
\be
h_{x,N}^{r,s}(l)=\frac{q^{s}_{N-x}(l)}{\sum_{l\geq 0}r^l q^{s}_N(l)}
\xrightarrow[N\to\infty]{}
(1-r)^2(l+1)=h^{r,1}(l). 
\ee
As a consequence,  initial distribution and transition kernel \eqref{eq:initial distribution and transition kernel r,s<1}  of 
$\PP_{\Schur,N}^{r,s}$ converges as $N\to\infty$. For all $\la,\mu\in\Sign_2$ and all $x\geq 0$,
\be
\P^{r,s}_{N}(\la)
=r^{\la_1-\la_2} \mathrm{1}_{\la_2=0}   h_{0,N}^{r,s}(\la_1-\la_2)
\xrightarrow[N\to\infty]{}
r^{\la_1-\la_2} \mathrm{1}_{\la_2=0}h^{r,1}(\la_1-\la_2)=\P^{r,1}(\la),
\ee
\be
\U_{x,N}^{s}(\la,\mu)
=(1-a)^2 s_{\mu/\la}(a)  
    \frac{h^{r,s}_{x+1,N}(\mu_1-\mu_2)}{h^{r,s}_{x,N}(\la_1-\la_2)}
    \xrightarrow[N\to\infty]{}
(1-a)^2 s_{\mu/\la}(a)
\frac{h^{r,1}(\mu_1-\mu_2) }{h^{r,1}(\la_1-\la_2) }=\U^{r,1}(\la,\mu),
\ee
as desired.

Similarly, when $s>1$, we know the asymptotics  of $q^{s}_{M}(l)$ by   \eqref{eq:asymptotics of q^s_M(l) for s>1}.  We have  
\be
h_{x,N}^{r,s}(l)=\frac{q^{s}_{N-x}(l)}{\sum_{l\geq 0}r^l q^{s}_N(l)}
\xrightarrow[N\to\infty]{}
h^{r,s}(l)
\Big( \frac{(1-a)^2}{(1-a/s)(1-as)} \Big)^{-x} . 
\ee
As a consequence,  for all $\la,\mu\in\Sign_2$ and all $x\geq 0$,
\be
\P^{r,s}_{N}(\la) 
\xrightarrow[N\to\infty]{}
\P^{r,s}(\la),\qquad
\U_{x,N}^{s}(\la,\mu)
 \xrightarrow[N\to\infty]{}
\U^{r,s}(\la,\mu),
\ee
as desired.
\end{proof}

\subsubsection{Relation with the Hariya--Yor process}\label{subsubsec:two layer Schur and HY relation}
Let $a\in (0,1)$ and $s\in(1,\frac{1}{a})$.  Let $\mu$ be a probability measure on $\RR_+$. Consider independent random variables 
    $$w_{11}\sim\mu,\  w_{2,1}=0,\ w_{2,2}\sim\Geo(r/s)
    \text{ and }
    w_{i1}\sim\Geo(as),\ w_{i2}\sim\Geo(a/s),\ i\geq 3.$$
For $m=1,2$ and $n\geq m$, define  $G^{\mu}_{r,s}(n,m)$ to be the last passage time from $(1,1)$ to $(n,m)$
    $$
    G^{\mu}_{r,s}(n,m)= \max_{\pi:(1,1)\rightarrow(n,m)}\sum_{(i,j)\in\pi}w_{i,j}, 
    $$

\begin{lemma}\label{lem:two row LPP time and infinite two layer Schur process}
    When $rs<1$ and $\mu=\Geo(rs)$,  
\be\label{eq:joint process is infinite two layer  Schur distributed}
    \Big( G^{\mu}_{r,s}(k+2,2), \sum_{\substack{ j=1,2\\1\leq i\leq k+2 }} w_{i,j}-G^{\mu}_{r,s}(k+2,2)  \Big)_{k\geq 0}\sim \PP^{r,s}_{\Schur,\infty}. 
\ee
\end{lemma}

\begin{figure} 
\centering
\begin{subfigure}{0.48\textwidth}
\centering
\begin{tikzpicture}[scale=1.5,font=\small]
\foreach \y in {0,1}{
  \draw[gray!35] (2,\y)--(2.8,\y); 
} 
  \draw (-1,0)--(0,1)--(2,1)
        (-1,0)--(2,0);
  \foreach \x in {0,1,2}{
    \draw (\x,0)--(\x,1);
    \fill (\x,0) circle (1.5pt);
    \fill (\x,1) circle (1.5pt);
  }
  \fill (-1,0) circle (1.5pt);
  \node[below] at (-1,0) {$\mu$};
  \node[below] at (0,0) {$0$};
  \node[above] at (0,1) {$\Geo(r/s)$};
  \foreach \x in {1,2}{
    \node[above] at (\x,1) {$ \Geo(a/s)$};
    \node[below] at (\x,0) {$\Geo(as)$};
  }
\end{tikzpicture} 
\end{subfigure}
\hfill
\begin{subfigure}{0.48\textwidth}
\centering
\begin{tikzpicture}[scale=1.5,font=\small]
\foreach \y in {0,1}{
  \draw[gray!35] (2,\y)--(2.8,\y); 
} 
  \draw (0,0)--(2,0) (0,1)--(2,1);
  \foreach \x in {0,1,2}{
    \draw (\x,0)--(\x,1);
    \fill (\x,0) circle (1.5pt);
    \fill (\x,1) circle (1.5pt);
  }
  \node[below] at (0,0) {$\Geo(rs)$};
  \node[above] at (0,1) {$\Geo(r/s)$};
  \foreach \x in {1,2}{
    \node[above] at (\x,1) {$\Geo(a/s)$};
    \node[below] at (\x,0) {$\Geo(as)$};
  }
\end{tikzpicture} 
\end{subfigure}
\caption{When $\mu=\Geo(rs)$, $G^{\mu}_{r,s}$ corresponds to a two-row GLPP in full space.}
\label{fig:two row polymer}
\end{figure}
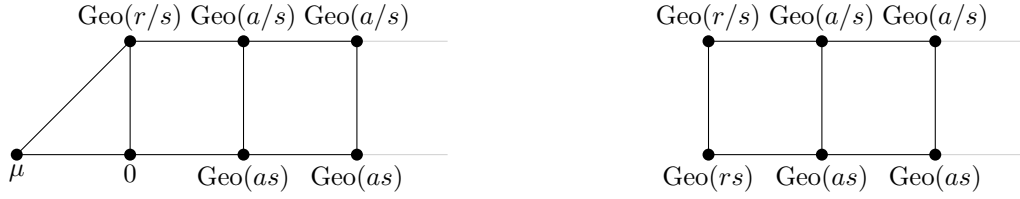

\begin{proof}
When $\mu\sim\Geo(rs)$ for $rs<1$, the passage times  
$\big(G^{\mu}_{r,s}(n,m)\big)_{n\geq 2, m=1,2}$ can be identified with those of a full space model with two rows, see Figure \ref{fig:two row polymer}. 
Then by Greene's theorem \cite{Greene} and the RSK correspondence, it is well known that the first $N$ marginal of the left hand side of \eqref{eq:joint process is infinite two layer  Schur distributed}   form a Schur process with specialization $(r,a,\ldots,a;\frac{1}{s},s)$, see e.g. \cite[Theorem 2.4]{DimitrovYang} for details. We conclude by recognizing that the first  $N$ marginal of $ \PP^{r,s}_{\Schur,\infty}$  \eqref{eq:ic and transition kernel of infinite two layer Schur process} form the Schur process with the same parameters. 
\end{proof}

\begin{lemma}\label{lem:two row LPP time and two layer HY process}
For any probability distribution $\mu$, we have equality in distribution 
\be
    \Big( G^{\mu}_{r,s}(k+2,2)-G^{\mu}_{r,s}(2,2), \sum_{\substack{ j=1,2\\1\leq i\leq k+2 }} w_{i,j}-G^{\mu}_{r,s}(k+2,2)  \Big)_{k\geq 0} =\Big(\mathbb{HY}^{a,r,s}_{\Geo}(k)\Big)_{k\geq 0}. 
\ee
\end{lemma}

\begin{proof} 
From the simple observation in \cite{BarraquandCorwin} that every up-right path goes through the point $(1,1)$, the left hand side does not depend on the random variable $w_{11}$. 
When $\mu=\delta_0$, $G^{\mu}_{r,s}=G^{\stat}_{r,s}$ defined in \cite{BarraquandCorwin}.  The process $\big( G^{\stat}_{r,s}(k+2,2)-G^{\stat}_{r,s}\big)_{k\geq 0}$ equals $[\mathbb{HY}_{\Geo}^{a,r,s}]_1$ by \cite[Proposition 3.2]{BarraquandCorwin}.
The equality between  second layer marginals follows from  \eqref{eq:conservation Geo}.  
\end{proof}

\begin{corollary}\label{cor:HY equals the increments of infinite two layer Schur process}
    When $rs<1$. Let $(\la^x)_{x\geq0} \sim \PP^{r,s}_{\Schur,\infty}$, we have 
    $$
\big(\la^x-\la^0\big)_{x\geq0}=\Big(\mathbb{HY}^{a,r,s}_{\Geo}(k)\Big)_{k\geq 0}
    $$
\end{corollary}
\begin{proof}
    Let $\mu=\Geo(rs)$. By Lemma \ref{lem:two row LPP time and infinite two layer Schur process}, we have 
    $$
    \big(\la^x \big)_{x\geq0}\overset{d}{=} \Big( G^{\mu}_{r,s}(k+2,k), \sum_{\substack{ j=1,2\\1\leq i\leq k+2 }} w_{i,j}-G^{\mu}_{r,s}(k+2,2)  \Big)_{k\geq 0}.  
    $$
    Hence by Lemma \ref{lem:two row LPP time and two layer HY process}, 
    $\big(\la^x_1-\la^0_1,\la^x_2 \big)_{x\geq 0}\overset{d}{=}\Big(\mathbb{HY}^{a,r,s}_{\Geo}(k)\Big)_{k\geq 0} $. 
    The statement follows from the observation that $\PP^{r,s}_{\Schur,\infty}$- almost surely, $\la_2^0=0$. 
\end{proof}

\subsection{Proof of the theorems}\label{subsec:proof of the theorems Geo}

\begin{proof}[Proof of Proposition \ref{prop:Girsanov Geo}]
Assume first that $s>1$.  
Let $Y_*\sim\Geo(s^{-2})$ be independent of $(A,S)$.  We have
\begin{align*}
    \EE\left[
 F\left(\mathcal{T}^Y(S,A)|_{\llbracket 0, N \rrbracket}
 \right)\right]
  &=\frac{1-r/s}{1-s^{-2}}
 \EE\left[F\left(\mathcal{T}^{Y_*}(S,A)|_{\llbracket 0, N \rrbracket}
 \right)(rs)^{Y_*}
 \right]\\
 &=\frac{1-r/s}{1-s^{-2}}
 \EE\left[ F\left((A,S)\big|_{\llbracket 0, N \rrbracket}\right)
 (rs)^{\mathcal{A}_{\infty}(S,A)}
 \right],
\end{align*}
 where we   change the distribution of the parameter of the  $\mathcal T$ transform from that of $Y$ to $Y_*$ in the first equality and we use  Dufresne's relation \eqref{eq:joint dufresne} in the second equality.  

Use \eqref{eq:cutting LG} and \eqref{eq:exponential functional developped Geo},  
we obtain
\begin{align}
 \mathcal{A}_{\infty}(S,A) 
  =\mathcal{A}_{N}(S,A)+\Big(\mathcal{A}_{\infty} (S^{(N)},A^{(N)})-\big(A(N)-S(N)+\mathcal{A}_{N}(S,A)\big)\Big)_+. 
\end{align}
By Markovianity of the random walk, $(A^{(N)},S^{(N)})$ is independent of $\sigma\left(A(i),S(i):0\leq i\leq N\right)$  and has the same law as $(A,S)$.  By the Dufresne's relation \eqref{eq:Geo Dufresne}
$$\mathcal{A}_{\infty}(S^{(N)},A^{(N)})
 \sim\Geo(s^{-2}). 
$$
Recall the definition of function $\mathcal{N}$
\eqref{eq:function N for two paths Geo}. We have
$$ 
\frac{1-r/s}{1-s^{-2}}
\EE\left[(rs)^{\mathcal{A}_{\infty}(S,A)}| \mathcal F_N\right] 
=\mathscr{N}_N^{r,s}(A,S), 
$$
which gives \eqref{eq:Geo Girsanov}.  

When $rs<1$, including in particular the case  $s=1$, we have  
\begin{multline*}
    \EE 
\left[
 F\left(\mathcal{T}^Y(S,A)|_{\llbracket 0, N \rrbracket}\right)
 \right]
\\
=\sum_{(\Sign_2^+)^{N+1}}
F\left((\lambda^k-\lambda^0)_{0\leq k\leq N}\right)
r^{\la^0_1-\la^0_2} \mathrm{1}_{\la^0_2=0}
h^{r,s}(\la^0_1-\la^0_2)\prod_{i=1}^{n}
\frac{s_{\la^i/\la^{i-1}}(a)}{\Pi(a;(s,1/s))} 
    \frac{h^{r,s}(\la^i_1-\la^i_2) }{h^{r,s}(\la^{i-1}_1-\la^{i-1}_2) }
 \\
=\sum_{A,S\in\mathcal P_{ N}}
F((A,S))
\PP_{\Geo,N}^{as}(A) \PP_{\Geo,N}^{a/s}(S) s^{S(N)-A(N)}  
\sum_{l\geq0}
r^l \mathsf{1}_{l\geq \mathcal A_N(S,A)}
h^{r,s}(l+A(N)-S(N)) \\ 
=\EE \left[
 F(A,S)\mathcal{N}_{N}^{r,s}(A,S)
 \right].
\end{multline*}
The the first equality follows from Corollary
\ref{cor:HY equals the increments of infinite two layer Schur process} and the definition of infinite two-layer Schur process    $\PP^{r,s}_{\Schur,\infty}$ \eqref{eq:ic and transition kernel of infinite two layer Schur process}. 
In the second equality, we expand Schur function and use the change of variables $\la^i_1=A(i)+l,\la^i_2=S(i)$. We evaluate the last sum using  \eqref{eq:function N geo as sum of Schur}.  
\end{proof}

\begin{proof}[Proof of Theorem \ref{thm:Geo main thm}]
The high density case \eqref{eq:asymp of K geo HD all} and the low density case \eqref{eq:asymp of K geo LD all} are proven in Proposition \ref{prop:HD Geo} and \ref{prop:LD Geo}.

The coexistence line case \eqref{eq:asymp of K geo CE all} is proven in Proposition \ref{eq:asymp of K geo CE all}.

The asympotics in the fan region is proven in Proposition \ref{prop:fan region asymptotics Geo}, which contains in particular the maximal current case \eqref{eq:asymp of K geo MC all} and the two axis cases \eqref{eq:asymp of K geo r axis all} and \eqref{eq:asymp of K geo s axis all}. 
 
Finally, when $rs=1$, which includes the case $r=s=1$, we have 
  $\mathcal K_M=1$   by definition \eqref{eq:generalized partition function geo}.  
\end{proof}

\begin{proof}[Proof of Corollary \ref{cor:HY theorem Geo}]
Let us first prove the convergence for   $r<1, rs<1$, in which case $s_{\star}=\max\{1,s\}$.   
For all $k\geq 0$ and all $M\geq k$, let 
$$\big(L_1^x(M),L_2^x(M)\big)_{0\leq x\leq M} \sim\PP_{\Geo,M}^{a,r,s}, \quad   (\la^x_1(M),\la^x_2(M) \big)_{0\leq x\leq M}\sim \PP_{\Schur,M}^{r,s} \quad  (\la^x)_{x\geq0}\sim \PP^{r,s}_{\Schur,\infty}.$$ 
Then we have 
\begin{multline*}
    \big(L_1^x(M),L_2^x(M)\big)_{0\leq x\leq k}\overset{d}{=}\big(\la^x_1(M)-\la^0_1(M),\la^x_2(M)-\la^0_2(M)\big)_{0\leq x\leq k}\\
    \xRightarrow[M\to\infty]{}
    \big(\la^x_1 -\la^0_1 ,\la^x_2 -\la^0_2 \big)_{0\leq x\leq k}
    \overset{d}{=}
    \Big(\mathbb{HY}^{r,s}_{\Geo}(x)\Big)_{0\leq x\leq k}. 
\end{multline*} 
The first equality in distribution  follows from the relation \eqref{eq:relation between reweighted rw and two layer Schur process}, the second equality in distribution follows from Corollary \ref{cor:HY equals the increments of infinite two layer Schur process} and the convergence follows from  Proposition \ref{prop:two layer Schur converges to infinite two layer Schur}.

Let us   prove the convergence for   the cases $r\geq \max\{s,1\} $ and $s>\max\{r,1\}$. We have  
\begin{multline*}
     \PP_{\Geo,M}^{a,r,s}
\Big((L_1,L_2)|_{[0,k]}=(\ell_1,\ell_2)\Big) 
 =\PP_{\Geo,k}^{as}(\ell_1)\PP_{\Geo,k}^{ar}(\ell_2)
 \frac{\mathcal K_{M-k}^{a,r,s}
   \big(-(\ell_1\otimes\ell_2)(k),-\ell_1(k)\big)}
  {\mathcal{K}_{M}^{a,r,s} }\\
  \xrightarrow[]{M\to\infty}\  \PP_{\Geo,k}^{as}(\ell_1)\PP_{\Geo,k}^{ar}(\ell_2)\rho^{-k}
  \begin{cases}
      (rs)^{\ell_1(k)},   &r\geq \max\{s,1\},\\
      \mathcal{N}^{r,s}(-(\ell_1\otimes\ell_2)(k),-\ell_1(k)), &s>\max\{r,1\},
  \end{cases}\\ 
  =\begin{cases}
      \mathbb{HY}_{\Geo}^{a,r,r} \quad  &r\geq \max\{s,1\},\\
      \mathbb{HY}_{\Geo}^{a,r,s} &s>\max\{r,1\}. 
  \end{cases}
\end{multline*}
The first equality is Proposition \ref{prop:decomposition Geo}. 
The second convergence follows from asymptotics  \eqref{eq:asymp of K geo s axis all}\eqref{eq:asymp of K geo LD all}\eqref{eq:asymp of K geo CE all} and \eqref{eq:asymp of K geo HD all} in Theorem \ref{thm:Geo main thm} respectively, together with the fact that 
$\mathcal{N}^{r,s}(0,0)=\mathcal{N}^{r,s}(-\infty,0)=1$ and 
the evaluation in \eqref{eq:special case function N Geo}. The marginal distribution of the Hariya Yor process is given by Proposition \ref{prop:Girsanov Geo}. And  for the first case in the third equality, we change the reference measure of $L_1$ from $\PP_{\Geo,k}^{as}$ to $\PP_{\Geo,k}^{a/r}$, and use the fact that   $ \mathbb{HY}_{\Geo}^{a,r,r}=\PP_{\Geo,k}^{a/r}\times \PP_{\Geo,k}^{ar}$. For the second case in the third equality, we change the reference measure of $L_2$ from $\PP_{\Geo,k}^{ar}$ to $\PP_{\Geo,k}^{a/s}$, and use the fact that 
$$
\mathcal{N}^{r,s}(-(\ell_1\otimes\ell_2)(k),-\ell_1(k))=(rs)^{l_2(k)}\mathcal N^{r,s}\Bigl(\mathcal A_k(l_2,l_1),l_2(k)-l_1(k)\Bigr).
$$ 
\end{proof}

\section{Hariya--Yor theorem for log-gamma polymers}
\label{sec:HY for LG}
\subsection{Two-layer $\mathscr{T}$ transform and  asymptotics of the   partition function}  
Let us start with proving some algebraic properties satisfied by the $\mathscr{T}$ transform in Definition \ref{def:functional and T transform Geo}.  
\begin{proposition} 
\label{prop:T properties LG}
For $L=(L_1,L_2)\in\mathscr{P}^2$ and $z \in\RR_{\geq0}$:
\begin{enumerate}
\item $\mathscr{T}^zL\in\mathscr{P}^2$.

\item For every $n\geq1$,
\be
 \frac1{\mathscr{A}_{n}([\mathscr{T}^zL]_2,
 [\mathscr{T}^zL]_1)}
 =\frac1{\mathscr{A}_{n}(L_2,L_1)}+z.
 \label{eq:LG capping}
\ee

\item For every $n\geq0$,
\be \label{eq:LG conservation}
 [\mathscr{T}^zL]_1(n)
 +[\mathscr{T}^zL]_2(n)=L_1(n)+L_2(n).
\ee

\item For all $z'\geq 0$,
\be
 \mathscr{T}^z\mathscr{T}^{z'}L
 =\mathscr{T}^{z+z'}L
 =\mathscr{T}^{z'}\mathscr{T}^zL.
 \label{eq:LG semigroup}
\ee
 In particular, recall the geometric  Pitman transform \eqref{eq:geometric Pitman transform}, we have 
\be\label{eq:communicate with geo Pitman transform LG}
 [\mathscr{T}^zL]_2\odot [\mathscr{T}^zL]_1=L_2\odot L_1,
\qquad
[\mathscr{T}^zL]_1\otimes [\mathscr{T}^zL]_2=L_1\otimes L_2.
\ee
\end{enumerate}
\end{proposition}
\begin{proof}
    The properties $[\mathscr{T}^zL]_1(0)=[\mathscr{T}^zL]_2(0)=0$ and \eqref{eq:LG conservation} follows from definition. 

Let us prove \eqref{eq:LG capping}. For $0\leq i\leq n$, denote $\mathscr{A}_i=\mathscr{A}_{i}(L_2,L_1)$ for short. Then  
    \begin{align*}
        \mathscr{A}_{n}([\mathscr{T}^zL]_2,
 [\mathscr{T}^zL]_1) &=\sum_{i=1}^{n}e^
 {L_2(i)-\log(1+z\mathscr{A}_i)-L_1(i-1)-\log(1+z\mathscr{A}_{i-1})}\\
&=\sum_{i=1}^{n}\frac{\mathscr{A}_{i}-\mathscr{A}_{i-1}}{(1+z\mathscr{A}_i)(1+z\mathscr{A}_{i-1})}=\sum_{i=1}^{n}\frac{\mathscr{A}_{i}}{1+z\mathscr{A}_i}-\frac{\mathscr{A}_{i-1}}{1+z\mathscr{A}_{i-1}}  \\
 &=\frac{\mathscr{A}_{n}}{1+z\mathscr{A}_n}. 
    \end{align*} 
    
To prove \eqref{eq:LG semigroup}, for the first coordinate we have for   $n\geq0$,
\begin{align*}
 [\mathscr{T}^{z'}\mathscr{T}^zL]_1
    (n)&= [\mathscr{T}^zL]_1(n)+\log\left(1+z'\mathscr{A}_{n}([\mathscr{T}^zL]_2,
 [\mathscr{T}^zL]_1)\right)\\
 &= L_1(n)+\log(1+z\mathscr{A}_{n}(L_2,L_1))+
 \log\left(1+z'\frac{\mathscr{A}_{n}(L_2,L_1)}{1+z\mathscr{A}_{n}(L_2,L_1)}\right)\\
 &= L_1(n)+\log(1+(z+z')\mathscr{A}_{n}(L_2,L_1))
 =[\mathscr{T}^{z+z'}L]_1,
\end{align*}
where we use \eqref{eq:LG capping} in the second equality. 
By \eqref{eq:LG conservation}, we also have $[\mathscr{T}^{z'}\mathscr{T}^zL]_2=[\mathscr{T}^{z+z'}L]_2$.  
Finally, \eqref{eq:communicate with geo Pitman transform LG} follows by
 letting $z'\to\infty$ in \eqref{eq:LG semigroup} and using
\eqref{eq:renormalization to get geometric Pitman transform}.  
\end{proof}
 
Recall the extended exponential functional $\mathscr{A}_{M}^{ a,\xi}(L_1,L_2)$ \eqref{eq:affine extended exponential functional LG}.
For $k\geq0$ and a path $L\in\mathscr{P}$, define $L^{(k)}\in\mathscr{P}$ to be the translated path $
  L ^{(k)}(n)=L (k+n)-L (k) 
$
for $n\geq 0. $

\begin{lemma}
\label{lem:cutting LG}
Let $N\geq 0$, $a\geq0$, and $\xi>0$.  For every
$L_1,L_2\in\mathscr{P}$,
\be \label{eq:exponential functional developped LG}
 L_2(N)-\log\mathscr{A}_{N}^{a,\xi}(L_2,L_1)
 =-\log\left(
 a e^{-L_2(N)}
 +\xi e^{-(L_1\otimes L_2)(N)}\right). 
\ee
Moreover, for $0\leq k\leq N$, we have
\be
 \label{eq:cutting LG}
 L_2(N)-\log\mathscr{A}_{N}^{a,\xi}(L_2,L_1)
 =L_2^{(k)}(N-k)
 -\log\mathscr{A}_{N-k}^{a_k,\xi_k}
 (L_2^{(k)},L_1^{(k)}),
\ee
where
$$
 a_k
 =a e^{-L_2(k)}
 +\xi e^{-(L_1\otimes L_2)(k)},
 \qquad
 \xi_k=\xi e^{-L_1(k)}.
$$
\end{lemma}
\begin{proof}
  The   identity
\eqref{eq:exponential functional developped LG} follows from the definition of the extended exponential functional \eqref{eq:affine extended exponential functional LG} and the geometric Pitman transform \eqref{eq:geometric Pitman transform}. Hence  
$$ 
 \mathscr A_N^{a,\xi}(L_2,L_1)
 = a+\xi\mathscr A_k(L_2,L_1)
+\xi e^{L_2(k)-L_1(k)}
\mathscr A_{N-k}(L_2^{(k)},L_1^{(k)}) 
=e^{L_2(k) }\mathscr A_{N-k}^{a_k,\xi_k}(L_2^{(k)},L_1^{(k)}),
$$
where we split the sum in the exponential functional at $k$ in the first equality and use \eqref{eq:exponential functional developped LG} for $L_1^{(k)},L^{(k)}_2$ in the second equality.  
\end{proof}
 
Recall the   partition function $\mathscr{K}_{M}^{\alpha,u,v}(  a,\xi)$ in  \eqref{eq:generalized partition function LG}.

\begin{proposition}
\label{prop:decomposition LG}
Let $M\geq1$. For $0\leq k\leq M$ and
$ (\ell_1,\ell_2)\in\mathscr{P}^2$, we have
\be
\PP_{\LG,M}^{\alpha,u,v}\big((L_1,L_2)|_{[0,k]}\in\d(\ell_1,\ell_2)\big)=\dd\PP_{\LG,k}^{\alpha+v,\alpha+u}(\ell_1,\ell_2)
 \frac{\mathscr{K}_{M-k}^{\alpha,u,v}
 (e^{-(\ell_1\otimes \ell_2)(k)},e^{-\ell_1(k)})}
 {\mathscr{K}_{M}^{\alpha,u,v} }.
 \label{eq:decomposition LG}
\ee
\end{proposition}
\begin{proof} 
Let $\mathcal F_k=\sigma(L_1(i),L_2(i):0\leq i\leq k)$.
By Markovianity of the random walk,  the paths $(L_1^{(k)},L_2^{(k)})$ are independent
of $\mathcal F_k$ and have distribution
$\PP_{\LG,M-k}^{\alpha+v,\alpha+u}$. Applying
\eqref{eq:cutting LG} with $ a=0,\xi= 1 $ and $N=M$, we have 
\begin{multline*}
      \EE_{\LG,M}^{\alpha+v,\alpha+u}
\left[e^{(u+v)L_2 (M )}
\bigl(\mathscr A_{M } 
(L_2  ,L_1 )\bigr)^{-(u+v)}
\mid\mathcal F_k\right] \\
 =
\EE_{\LG,M}^{\alpha+v,\alpha+u}
\left[
e^{(u+v)L_2^{(k)}(M-k)}
\bigl(\mathscr A_{M-k}^{a_k,\xi_k}
(L_2^{(k)},L_1^{(k)})\bigr)^{-(u+v)}
\mid\mathcal F_k\right] =\mathscr K_{M-k}^{\alpha,u,v}
\left(e^{-(L_1\otimes L_2)(k)},e^{-L_1(k)}\right).
\end{multline*} 
Together with \eqref{eq:Stationary measure LG} proves 
\eqref{eq:decomposition LG}.
\end{proof}

The main goal of this section is the following asymptotics for $\mathscr{K}_{M}^{\alpha,u,v}(a,\xi)$. Recall the definition of function $\varrho=\varrho(\a,u,v)$ in \eqref{eq:rho and chi LG} and $\mathscr N^{u,v}(  a,\xi)$ in \eqref{eq:function N KPZ}  
\begin{theorem}
\label{thm:LG main thm}
Fix $a\geq0$ and $\xi>0$. Under \eqref{eq:parameter domain LG}, as
$M\to\infty$:
\begin{enumerate}

\item If $u,v>0$, then
\be 
\label{eq:asymp of K LG MC all}
 \mathscr{K}_{M}^{\alpha,u,v}(a,\xi)
 \sim{} 
 \frac{\Gamma(u)^2\Gamma(v)^2\xi^{-v}}
 {2\sqrt\pi\Gamma(u+v) \psi_1 (\a) ^{3/2}}
   \mathscr{N}^{u,0}(a,\xi) M^{-3/2}\varrho^M. 
\ee

\item If $u>0$ and $v=0$, then
\be 
\label{eq:asymp of K LG u axis all}
 \mathscr{K}_{M}^{\alpha,u,0}(a,\xi)
 \sim 
 \frac{\Gamma(u)}{\sqrt{\pi\psi_1(\alpha) }}
 \mathscr{N}^{u,0}(a,\xi)
 M^{-1/2}\varrho^M.
\ee

\item If $u=0$ and $v>0$, then
\be
\label{eq:asymp of K LG v axis all}
 \mathscr{K}_{M}^{\alpha,0,v}(a,\xi)
 \sim 
 \frac{\Gamma(v)}{\sqrt{\pi\psi_1(\alpha) }}
 \mathscr{N}^{v,0}(0,\xi) M^{-1/2}\varrho^M.
\ee

\item If $v<\min\{0,u\}$, then
\be
\label{eq:asymp of K HD LG all}
 \mathscr{K}_{M}^{\alpha,u,v}(a,\xi)
 \sim 
 \frac{\Gamma(u-v)}{\Gamma(-2v)}
 \mathscr{N}^{u,v}(a,\xi)\varrho^M.
\ee

\item If $u<\min\{0,v\}$, then
\be
\label{eq:asymp of K LD LG all}
 \mathscr{K}_{M}^{\alpha,u,v}(a,\xi)
 \sim 
 \frac{\Gamma(v-u)}{\Gamma(-2u)}
 \mathscr{N}^{v,u}(0,\xi)\varrho^M.
\ee

\item If $u=v<0$, then
\be
\label{eq:asymp of K CE LG all} 
 \mathscr{K}_{M}^{\alpha,u,u}(a,\xi)
 \sim{}
  \frac{\psi_0(\alpha-u)-\psi_0(\alpha+u)}
 {\Gamma(-2u)} \xi^{-2u} M\varrho^M. 
\ee

\item If $u+v=0$, in particular when $u=v=0$, then 
$  \mathscr{K}_{M}^{\alpha,u,v}(a,\xi)=1 $ for all $M\geq 1$. 
\end{enumerate}
\end{theorem}

Theorem \ref{thm:LG main thm} is proven in Section  \ref{subsec:proof of the theorems LG}.

\subsection{The boundary dominated region}

\subsubsection{The high density phase}
Following the same approach as in the GLPP case,  we prove in this section  a process level Dufresne's identity for a pair of log-gamma random walks and use it to deduce the asymptotics of the partition function in the   high density phase.

\begin{lemma}
\label{lem:LG Dufresne}
Let $0<q<p$.  Let $A,S,Y$ be independent  with
$$
 A\sim\PP_{\LG}^{q},\qquad
 S\sim\PP_{\LG}^{p},\qquad
 Y\sim\Gav(p-q).
$$
Then
\be
 \left(\mathscr{T}^{\frac{1}{Y}}(S,A),\frac{1}{Y}\right)
 \overset{\mathrm d}=
 \left((A,S),\mathscr{A}_{\infty}(S,A)^{-1}\right).
 \label{eq:joint dufresne LG}
\ee
As a consequence,
\be
 \label{eq:process Dufresne LG}
 \left(\mathscr{A}_{n}(S,A)^{-1}\right)_{n\geq1}
 \overset{\mathrm d}=
 \left(\frac{1}{Y}+\mathscr{A}_{n}(A,S)^{-1}\right)_{n\geq1},
\ee
and
\be
 \label{eq:LG Dufresne}
 \mathscr{A}_{M}(S,A)\uparrow
 \mathscr{A}_{\infty}(S,A),
 \qquad
 \mathscr{A}_{\infty}(S,A)^{-1}\sim\Ga(p-q).
\ee
\end{lemma}

\begin{proof}
\cite[Theorem 3.3]{Seppalainen} gives the identity in law between processes:
\begin{equation}\label{eq:Sep}
   \mathscr{T}  ^{\frac{1}{Y}}(S,A)
 \overset{\mathrm d}=(A,S).
\end{equation}
Taking deterministic functional $ L\mapsto (\mathscr A_n( L_2,L_1)^{-1})_{n\geq0}$ on both sides of \eqref{eq:Sep}, \eqref{eq:process Dufresne LG} follows by the property   \eqref{eq:LG capping}.
Since $q<p$,  $\mathscr{A}_{n}(A,S)\to+\infty$ by strong law of large numbers,  hence $\PP_{\LG}^{q}\times \PP_{\LG}^{p} $-almost surely,
$$ \mathscr A_\infty
 \bigl([\mathscr{T}^{\frac{1}{Y}}(S,A)]_2,[\mathscr{T}^\frac{1}{Y}(S,A)]_1\bigr)^{-1}=\frac{1}{Y}. $$
Taking deterministic functional $ L\mapsto (L, \mathscr A_{\infty}( L_2,L_1)^{-1})$ on both sides of \eqref{eq:Sep}, \eqref{eq:joint dufresne LG} follows by the property of \eqref{eq:LG capping}.
The second coordinate marginal of equality in distribution \eqref{eq:joint dufresne LG} gives 
\eqref{eq:LG Dufresne}.    
\end{proof}  

\begin{remark}
\label{rmk:relation with BW dufresne}
Following the discussions in \cite[Section 1.2.3]{BarraquandCorwin} and \cite[Remark 2.3]{BrycWesolowskiLG},     
we may define the following quantities for $k\geq1$ under the same assumptions of Lemma \ref{lem:LG Dufresne},  
$$
\zeta(k)= \frac{e^{S(k+1)-S(k)}}{e^{A(k)-A(k-1)}} , \quad 
r(k) =e^{S(1)}\prod_{j=1}^{k-1}\zeta(j)
=\mathscr{A}_{k}(S,A)-\mathscr{A}_{k-1}(S,A),\quad 
R(k) =1 +\sum_{j=1}^{k}\prod_{i=1}^{j}\zeta(i).
$$
Then the random variables $\zeta(k)$  are i.i.d. with Beta prime distribution  $\Bt'(q,p)$ (i.e. distribution of $G_p/G_q$ for independent $G_p\sim\Ga(p)$ and $G_q\sim\Ga(q)$). The process $(r(k))_{k\geq1}$ is a  $\Bt(p,q)$ multiplicative random walk starting from an independent $e^{S(1)}\sim\Gav(p)$, and  $(R(k))_{k\geq1}$ is a Markov process known as perpetuities and satisfies what's called Kesten's recursion \cite{Kesten}. 
The process $R(k)$ is related with our exponential functional by
$$
\mathscr{A}_k(S,A)=e^{S(1)}R(k-1),
$$
hence in view of  Lukacs' theorem, the convergence in the Dufresne's relation \eqref{eq:LG Dufresne} is   equivalent to the perpetuities convergence   $R(k)\uparrow
 R(\infty)\sim 1+\Bt'(p-q,q) $  considered in \cite[Section 2]{BrycWesolowskiLG}.
\end{remark}
  
\begin{proposition}
\label{prop:HD LG}
Assume \eqref{eq:parameter domain LG} and $v<\min\{0,u\}$. Then for all
  $ a\geq0$ and $\xi>0$, we have 
\be
 \label{eq:asymp of K HD LG}
\mathscr{K}_{M}^{\alpha,u,v}(a,\xi)
 \sim 
 \frac{\Gamma(u-v)}{\Gamma(-2v)}
 \mathscr{N}^{u,v}(a,\xi)\left(\frac{\Gamma(\a-v)}{\Gamma(\a+u)}\right)^{M}.
\ee
\end{proposition}
\begin{proof}
    Since $\a-v>0$, we can change the reference measure of $L_2$ from $\PP_{\LG,M}^{\a+u}$ to $\PP_{\LG,M}^{\a-v}$ and obtain  
    $$
    \mathscr{K}_{M}^{\alpha,u,v}(a,\xi)
 =\left(\frac{\Gamma(\a-v)}{\Gamma(\a+u)}\right)^{M}\EE_{\LG,M}^{\alpha+v,\alpha-v}
 \left[ 
 \big(a+\xi\mathscr{A}_{M}(L_2,L_1)\big)^{-(u+v)}\right].
    $$
    Apply Lemma \ref{lem:LG Dufresne} with $A=L_1, S=L_2, p=\a-v$ and $ q=\a+v$, we have 
    $$\mathscr{A}_{M}(L_2,L_1)\uparrow
   \mathscr{A}_{\infty}(L_2,L_1) \sim\Gav(-2v).$$
   Using monotone convergence when $u+v<0$ and dominated convergence when $u+v\geq 0$, we obtain \eqref{eq:asymp of K HD LG} by noticing for $v<\min\{0,u\}$ and $G\sim\Gammainv(-2v)$, 
$$
   \EE\left[ 
 \big(a+\xi G\big)^{-(u+v)}\right]=\frac{\Gamma(u-v)}{\Gamma(-2v)}
\mathscr{N}^{u,v}(a,\xi).
 $$
\end{proof}

\subsubsection{The low density phase}
As for the GLPP case, we prove in this section  a time reversal lemma  and use it to deduce the asymptotics of 
$\mathscr{K}_{M}^{\alpha,u,v}(a,\xi)$ in the the low density phase.
\begin{lemma}
\label{lem:time reversal LG}
Consider two independent random walks $A$ and $S$ in $\mathscr{P}$. Then  for all $M\geq 0$, 
\be\label{eq:time reversal LG}
 \bigl(\mathscr{A}_{M}(A,S),A(M)-S(M)\bigr)
 \overset{\mathrm d}=
 \bigl(e^{A(M)-S(M)}\mathscr{A}_{M}(S,A),
       A(M)-S(M)\bigr).
\ee
\end{lemma}

\begin{proof}
Let $M\geq 1$. For $0\leq i\leq M$, define
$$\overleftarrow A(i)=A(M)-A(M-i),\quad \overleftarrow S(i)=S(M)-S(M-i).$$ 
Then 
$$
 \mathscr{A}_M(\overleftarrow A,\overleftarrow S)
 =e^{A(M)-S(M)}\mathscr{A}_M(S,A).
$$
Since $A$ and $S$ are random walks in $\mathcal{P}$, they have  i.i.d. increments and start from $0$, hence 
$$\left(  A(i), S(i)\right)_{0\leq i\leq M}
\overset{\mathrm d}{=}
\left( \overleftarrow A(i), \overleftarrow S(i)\right)_{0\leq i\leq M}.
$$
In particular, 
$$\left(\mathscr{A}_M(A,S),A(M)-S(M)\right)
\overset{\mathrm d}{=}
\left(\mathscr{A}_M( \overleftarrow A,\overleftarrow S),\overleftarrow A(M)-\overleftarrow S(M)\right),
$$
which is \eqref{eq:time reversal LG}. 
\end{proof}


\begin{proposition}
\label{prop:LD LG}
Assume \eqref{eq:parameter domain LG} and $u<\min\{0,v\} $. Then for all
  $ a\geq0$ and $\xi>0$, we have
\be
 \label{eq:asymp of K LD LG}
 \mathscr{K}_{M}^{\alpha,u,v}(a,\xi)
 \sim 
 \frac{\Gamma(v-u)}{\Gamma(-2u)}
 \mathscr{N}^{v,u}(0,\xi)\left(\frac{\Gamma(\a-u)}{\Gamma(\a+v)}\right)^M.
\ee
\end{proposition}
\begin{proof}
    Since $u<0$, in particular $a-u>0$,  
\begin{align*}
    \mathscr{K}_{M}^{\a,u,v}(a,\xi)
  &=
 \left(\frac{\Gamma(\a-u)}{\Gamma(\a+v)}\right)^M
 \EE_{\Geo,M}^{a-u,a+u}
 \left[e^{(u+v)\big(L_1(M)-L_2(M)\big)}\Big(
  a+\xi\mathscr{A}_{M}(L_2,L_1)\Big)^{-(u+v)}
 \right]\\
 &=\left(\frac{\Gamma(\a-u)}{\Gamma(\a+v)}\right)^M
 \EE_{\Geo,M}^{a-u,a+u}
 \left[\Big(
  ae^{L_1(M)-L_2(M)}+\xi\mathscr{A}_{M}(L_1,L_2)\Big)^{-(u+v)}
 \right],
\end{align*}
where we change the reference measure of $L_1$ from $\PP_{\LG,M}^{\a+v}$ to $\PP_{\LG,M}^{\a-u}$ in the first equality and apply  Lemma \ref{lem:time reversal LG}   with
$A=L_2\sim\PP_{\LG,M}^{\a+u} $, $S=L_1\sim\PP_{\LG,M}^{\a-u}$ in the second equality.

Since $u<0$, $L_1(M)-L_2(M)\rightarrow-\infty$ almost surely. By Dufresne's identity \eqref{eq:LG Dufresne}, we have 
\be\label{eq:a convergence to inverse gamma}
ae^{L_1(M)-L_2(M)}+\xi\mathscr{A}_{M}(L_1,L_2) \xrightarrow[M\to\infty]{a.s.}\xi\mathscr{A}_{\infty}(L_1,L_2)\sim\xi\Gav(-2u).
\ee
Observe that for all $M\geq 1$, we have 
$$\xi\leq ae^{L_1(M)-L_2(M)}+\xi\mathscr{A}_{M}(L_1,L_2)\leq (a+\xi)\mathscr{A}_{\infty}(L_1,L_2).$$
Using the left inequality for $u+v>0$ and   the right inequality for  $u+v\leq 0$,  
\eqref{eq:asymp of K LD LG} follows from the dominated convergence theorem and the fact that for $G\sim\Gammainv(-2u)$, 
$$\xi^{-(u+v)}\EE[G^{-(u+v)}]= \mathscr{N}^{v,u}(0,\xi)\frac{\Gamma(v-u)}{\Gamma(-2u)}.
$$ 
\end{proof}

\subsubsection{The coexistence line}  
We take the same approach as in  the GLPP case and start with the following lemma.
\begin{lemma}
    Let $\alpha>0$, $\alpha\pm u>0$, and $\alpha\pm v>0$. For $a,\xi>0$, define
$$
\mathscr{M}^{\alpha,u,v}(a,\xi)
:=\EE\left[(\xi+\frac{a}{Y})^{-(u+v)}
-(\frac{a}{Y})^{-(u+v)}\right],
$$
where $Y\sim \Gav(\a+u)$.
Then, for  all  $M\geq1$,  $a\geq0$ and $\xi>0$, we have 
\begin{multline}\label{eq:LG K recursion} 
    \mathscr K_{M+1}^{\alpha,u,v}(a,\xi)
-\left(\frac{\Gamma(\alpha-v)}{\Gamma(\alpha+u)}\right)\mathscr K_M^{\alpha,u,v}(a,\xi) \\
= 
\left(\frac{\Gamma(\alpha-u)}{\Gamma(\alpha+v)}\right)^M
\EE_{\LG,M}^{\alpha-u,\alpha+u}
\left[\mathscr{M}^{\alpha,u,v}
\left(a e^{L_1(M)-L_2(M)}
+\xi\mathscr A_M(L_1,L_2),\xi\right)\right].
\end{multline}
\end{lemma}

\begin{proof}
Let 
$\mathcal F_M= \sigma\Big( L_1(i),L_2(i): 0\leq i\leq M \Big)$.   
Apply  \eqref{eq:cutting LG} with  $N=M+1$ and $k=M$, we have 
\begin{equation*}
L_2(M+1)-\log\mathscr A_{M+1}^{a,\xi}(L_2,L_1)
=-\log\left(
a_M e^{-(L_2(M+1)-L_2(M))}+\xi_M\right).
\end{equation*}
where 
$$
a_M=a e^{-L_2(M)}
+\xi e^{-(L_1\otimes L_2)(M)},
\qquad \xi_M=\xi e^{-L_1(M)}.
$$
are measurable with respect to $\mathcal{F}_M$. 
By Markovianity of the random walk,  $e^{(L_2(M+1)-L_2(M))}\sim\Gav(\alpha+u)$ and  is independent
of $\mathcal F_M$. Hence   
$$\EE\left[
e^{(u+v)L_2(M+1)}
\left(\mathscr A_{M+1}^{a,\xi}(L_2,L_1)\right)^{-(u+v)}
\mathrel{\big|}\mathcal F_M\right]
=\EE\left[\left(\frac{a_M}{Y}+\xi_M\right)^{-(u+v)}\right].
$$
for $Y\sim\Gav(\a+u).$
Taking expectation with respect to $\EE_{\LG,M}^{\alpha+v,\alpha+u}$,
we obtain
\begin{align*}
\mathscr K_{M+1}^{\alpha,u,v}(a,\xi)
&= \EE_{\LG,M}^{\alpha+v,\alpha+u}\left[ \mathscr{M}^{\alpha,u,v}(a_M,\xi_M)\right]
+\frac{\Gamma(\a-v)}{\Gamma(\a+u)}\EE_{\LG,M}^{\alpha+v,\alpha+u}\left[ (a_M)^{-(u+v)} \right]\\
&=\EE_{\LG,M}^{\alpha+v,\alpha+u}
\left[e^{(u+v)L_1(M)}
\mathscr M^{\alpha,u,v}
\left(e^{L_1(M)}a_M,\xi\right)
\right]
+\frac{\Gamma(\a-v)}{\Gamma(\a+u)} \mathscr K_{M }^{\alpha,u,v}(a,\xi)\\
&=\left(\frac{\Gamma(\alpha-u)}{\Gamma(\alpha+v)}\right)^M\EE_{\LG,M}^{\alpha-u,\alpha+u}\left[ \mathscr{M}^{\alpha,u,v}\left(e^{L_1(M)}a_M,\xi\right)\right]
+\frac{\Gamma(\a-v)}{\Gamma(\a+u)} \mathscr K_{M }^{\alpha,u,v}(a,\xi),
\end{align*}
where we evaluate the moment of inverse gamma random variable in the first equality, we 
use  $\mathscr{M}^{\a,u,v}(ca,c\xi)=c^{-(u+v)}\mathscr{M}^{\a,u,v}(a,\xi)$ 
in the second equality,  
and change the reference measure of $L_1$ from
$\PP_{\LG,M}^{\alpha+v}$ to $\PP_{\LG,M}^{\alpha-u}$ in the last equality. Finally, \eqref{eq:LG K recursion} follows from the time reversal \eqref{eq:time reversal LG}:
$$
e^{L_1(M)}a_M=e^{L_1(M)-L_2(M)}
\left(a+\xi\mathscr A_M(L_2,L_1)\right)\overset{\mathrm d}=
a e^{L_1(M)-L_2(M)}
+\xi\mathscr A_M(L_1,L_2).
$$ 
\end{proof}

\begin{proposition}
\label{prop:CE LG}
Let $u=v<0$ and $\alpha+u>0$.  For fixed
$a\geq0$ and $\xi>0$,
\be\label{eq:critical line asymptotics LG}
\mathscr{K}_{M}^{\alpha,u,u}(a,\xi)
 \sim{}
  \frac{\psi_0(\alpha-u)-\psi_0(\alpha+u)}
 {\Gamma(-2u)} \xi^{-2u} M\varrho^M.
\ee
\end{proposition}
\begin{proof}
 Recall that on the coexistence line,
$
\varrho=\frac{\Gamma(\alpha-u)}{\Gamma(\alpha+u)}
$. 
Apply \eqref{eq:LG K recursion} with $v=u$ and divide both sides, and use the convergence \eqref{eq:a convergence to inverse gamma} and the same dominance in the proof of Proposition \ref{prop:LD LG}, we have for $G\sim\Gammainv(-2u)$,
\begin{multline*}
\frac{1}{\varrho^{M+1}} 
\left(\mathscr{K}_{M+1}^{\alpha,u,u}(a,\xi)-\varrho\mathscr{K}_{M }^{\alpha,u,u}(a,\xi)
\right)\\
=\frac{1}{\varrho} 
\EE_{\LG,M}^{\alpha-u,\alpha+u}
\left[\mathscr{M}^{\alpha,u,u}
\left(a e^{L_1(M)-L_2(M)}
+\xi\mathscr A_M(L_1,L_2),\xi\right)\right]\\
\xrightarrow []{M\to\infty}
\frac{1}{\varrho} \EE
\left[\mathscr{M}^{\alpha,u,u}
\left(\xi G,\xi\right)\right]= \xi^{-2u}\frac{\psi_0(\alpha-u)-\psi_0(\alpha+u)}
 {\Gamma(-2u)}, 
\end{multline*}
where $\psi_0$ is the digamma function. For the last equality, see e.g.   \cite[\href{https://dlmf.nist.gov/5.9.E16}{(5.9.16)}]{NIST:DLMF}.
Asymptotics \eqref{eq:critical line asymptotics LG} follows by the Stolz-Cesaro theorem.    
\end{proof}

\subsection{The fan region}

\subsubsection{Preliminaries about Whittaker functions} 
\label{subsubsec:whittaker preliminary} 
We follow the notations in \cite{CorwinOConnellSeppalainenZygouras} and  \cite[Section 3.1]{Barraquand}. For $\alpha\in\C$ , $x,y\in\RR^n$ and $z\in\RR^{n-1}$, define the kernels
$$
 \Psi_\alpha^{(n)}(x/y)
=\exp\left\{-\alpha\sum_{i=1}^{n}(x_i-y_i)
-\sum_{i=1}^{n}e^{-(x_i-y_i)}
-\sum_{i=1}^{n-1}e^{-(y_i-x_{i+1})}\right\}.
$$ 
and 
$$
\Psi_\alpha^{(n,n-1)}(x/z)
=\exp\bigg\{-\alpha\left(\sum_{i=1}^{n}x_i-
\sum_{i=1}^{n-1}z_i\right)
-\sum_{i=1}^{n-1} e^{-(x_i-z_i)}
-\sum_{i=1}^{n-1}
e^{-(z_i-x_{i+1})} \bigg\}.
$$

For $n,k\geq1$, the   Whittaker function and the skew Whittaker function can defined through the branching rule respectively as: 
$$
\Psi_{\alpha_1,\ldots,\alpha_n}(x)
=\int_{\RR^{n(n-1)/2}}
\prod_{i=1}^{n}\Psi_{\alpha_i}^{(i,i-1)}(x^i/x^{i-1})
\dd x^1\cdots\dd x^{n-1}, 
$$
and
\begin{equation}
\label{eq:skew Whittaker branching rule}
\Psi_{\alpha_1,\ldots,\alpha_k}^{(n)}(x/y)
=\int_{\RR^{n(k-1)}}
\prod_{i=1}^{k}\Psi_{\alpha_i}^{(n)}(x^i/x^{i-1})
\dd x^1\cdots\dd x^{k-1},
\end{equation}
where  $x^k=x$ and $x^0=\varnothing$ (resp. $x^0=y$).

As the Schur case, the (skew) Whittaker functions satisfy a Cauchy and skew-Cauchy identity. For our purpose, we  focus only on the case $n=2$, $k=1$, $\a_1=-\a_2=z$ and we omit the superscript $(n)$ in the skew Whittaker function for simplicity. For a more complete background on   Whittaker functions, we refer to \cite[Section 3.1 and Lemma 3.4]{Barraquand} and references therein. We have  
 \begin{itemize}
     \item Cauchy identity
     \be
 \label{eq:Whittaker Cauchy}
 \int_\RR e^{-\theta y}\Psi_{(z,-z)}(y,0)\dd y
 =\Gamma(\theta\pm z),
 \qquad  \theta> |\mathfrak{R} z|,
\ee
     \item Skew Cauchy identity
     \be \label{eq:Whittaker skew Cauchy}
 \int_{\RR^2}\Psi_\alpha(\mu/\lambda)
 \Psi_{(z,-z)}(\mu_1, \mu_2) \dd\mu
 =\Gamma(\alpha\pm z) 
 \Psi_{(z,-z)}(\lambda_1 ,\lambda_2) ,
 \qquad  \a> |\mathfrak{R} z|.
\ee
 \end{itemize}
 The function $\Psi_{(z,-z)}(y,0)$ is related to the K-Bessel function by $\Psi_{(z,-z)}(y,0)
 =2K_{2z}(2e^{-y/2})$.  
Using a representation for the Gamma function held for $\ta, S>0$:
\be\label{eq:gamma function integral representation}
 \Gamma(\theta)=S^{ \theta}\int_\RR e^{-\theta l-Se^{-l}}\dd l,
\ee
the function $\mathscr N^{u,v}(  a,\xi) $ defined in \eqref{eq:function N KPZ} admit the following representation for $a\geq 0, \xi>0$ and $v\in(-u, 0]$ :
\be\label{eq:function N LG as integral of Whittaker}
   \mathscr N^{u,v}(  a,\xi) = \frac{\xi^{-(u+v)}}{\Gamma(u\pm v) }
\int_{\RR}e^{-ul-( a/\xi)e^{-l}}\Psi_{(v,-v)}(l,0)\dd l.
\ee
    
\subsubsection{The two-layer Whittaker process}
In the fan region $u+v>0 $, the stationary measure  
$\PP_{\LG,M}^{\alpha,u,v}$  \eqref{eq:Stationary measure LG} has a Markovian description, which can be studied through the integral formulas developed in \cite{Barraquand}.

\begin{definition}[Two-layer Whittaker process] 
Let $N\in\ZZ_{\geq1}$.  For parameters $\a,u,v$ satisfying \eqref{eq:parameter domain LG} with $u+v>0$, 
define the probability measure  $\PP_{\Whit,N}^{\alpha,u,v}$ on
$\bl =(\lambda^0,\ldots,\lambda^N)\in(\RR^2)^{N+1}$ with density 
\be
 \PP_{\Whit,N}^{\alpha,u,v}(\dd\lambda)
 =\frac1{\mathcal Z_{\Whit,N}^{\alpha,u,v}}
 e^{-u(\lambda_1^0-\lambda_2^0)}
 e^{-v(\lambda_1^N-\lambda_2^N)}
 \prod_{j=1}^N\Psi_\alpha(\lambda^j/\lambda^{j-1})\delta_0(\dd\lambda_2^0)
 \dd\lambda_1^0\prod_{j=1}^N\dd\lambda^j,
\ee
where $\mathcal Z_{\Whit,N}^{\alpha,u,v}$ is the normalizing constant.
\end{definition}

\begin{lemma}\cite[Proposition 3.9]{Barraquand}
Under $\PP_{\Whit,N}^{\alpha,u,v}$, the process
$(\lambda^x)_{0\leq x\leq N}$ is an inhomogeneous Markov chain
on $\RR^2$
with initial distribution and transition kernel given by
\be
\begin{aligned}
 \mathsf P_N^{u,v}(\dd\lambda)
 &:=e^{-u(\lambda_1-\lambda_2)}
  \delta_0(\dd\lambda_2)
  H_{0,N}^{u,v}(\lambda_1-\lambda_2)
  \dd\lambda_1,\\
 \mathsf U_{x,N}^{u,v}(\lambda,\dd\mu)
 &:=\frac{\Psi_\alpha(\mu/\lambda)}{\Gamma(\alpha)^2}
 \frac{H_{x+1,N}^{u,v}(\mu_1-\mu_2)}
 {H_{x,N}^{u,v}(\lambda_1-\lambda_2)}\dd\mu,
 \qquad 0\leq x<N.
\end{aligned}
\label{eq:initial distribution and transition kernel u,v>0}
\ee
Here the functions
$(H_{x,N}^{u,v})_{0\leq x\leq N}:\RR\to\RR_+$ are defined by
\be 
 H_{x,N}^{u,v}(l)
  =\frac{Q_{N-x}^{v}(l)}
 {\int_{\RR}\d \ell e^{-u\ell}Q_N^{v}(\ell) }, \quad \text{with} \quad 
 Q_N^{v}(\ell)
 =\frac1{\Gamma(\alpha)^{2N}}
 \int_{\RR^2}\dd\la\Psi_{(\alpha)^N}(\la/(l,0))
 e^{-v(\la_1-\la_2)}.
 \label{eq:def of function H and Q(l)}
\ee
\end{lemma}
Let us prove a time-reversal lemma for later use. 
\begin{lemma} 
\label{lem:time reversal partition function LG}
Let $\a,u,v$ be parameters satisfying \eqref{eq:parameter domain LG}  with $u+v>0$. 
    Let $\bl \sim \PP_{\Whit,M}^{\alpha,u,v}$. 
    Then 
    $\overleftarrow \bl \sim\PP_{\Whit,M}^{\alpha,v,u}$, where for $0\leq i\leq M$, 
   \be\label{eq:time reversed two layer schur process LG}\overleftarrow \bl^i =\left( \la^M_1-\la_2^{M-i},   \la^M_1-\la_1^{M-i}\right).
   \ee
In particular,
\be\label{eq:boundary parameter symmetry LG}
 \mathcal Z_{\Whit,M}^{\alpha,u,v} =\mathcal Z_{\Whit,M}^{\alpha,v,u}.
\ee
\end{lemma}
\begin{proof}
    Recall the definition of the Whittaker function, for $0\leq i\leq M$ we have 
$$
 \Psi_\alpha(\overleftarrow\lambda^j/
                    \overleftarrow\lambda^{j-1})
 =\Psi_\alpha(\lambda^{M-j+1}/\lambda^{M-j}).
$$
 Hence 
 $\PP_{\Whit,M}^{\alpha,u,v}(\bl)=\frac{Z_{\Whit,M}^{\alpha,u,v}}{\mathcal Z_{\Whit,M}^{\alpha,v,u}}\PP_{\Whit,M}^{\alpha,v,u}$.  
The symmetry of  $\mathcal Z_{\Whit,M}^{\alpha,u,v} $ with respect to the boundary parameters follows from the observation that $\bl\mapsto\overleftarrow\bl$ is an involution on the space 
$\{\bl\in(\RR^2)^{M+1} |  \la^0_2=0\}$, 
and in particular a bijection.
\end{proof}

When $u+v>0$, the two-layer Whittaker process is related with  the stationary measure $\PP_{\LG,M}^{\alpha,u,v}$    by \cite[Proposition 3.14]{BarraquandCorwinYang}:
\be\label{eq:relation between reweighted rw and two layer Whittaker process}
 (\lambda_1^j-\lambda_1^0,
   \lambda_2^j-\lambda_2^0)_{0\leq j\leq M}
 \overset{\mathrm d}=(L_1(j),L_2(j))_{0\leq j\leq M},  
\ee
where $\bl \sim \PP_{\Whit,M}^{\alpha,u,v} $ and $\bL\sim \PP_{\LG,M}^{\alpha,u,v}$. 
Let us prove a relationship between the function  $Q^{v}_{N}(l)$ in \eqref{eq:def of function H and Q(l)}  and our   partition function  $ \mathscr{K}_{M}^{\alpha,u,v}(a,\xi)$ \eqref{eq:generalized partition function LG}.  

\begin{lemma} 
\label{lem:Relation K Q}
Assume \eqref{eq:parameter domain LG} and $u+v>0$,  we have 
\be 
 \mathscr{K}_{M}^{\alpha,u,v}(a,\xi)
 = \frac{\Gamma(\alpha)^{2M}\xi^{-(u+v)}}
 {\Gamma(u+v)
  \big(\Gamma(\alpha+u)\Gamma(\alpha+v)\big)^{M}} 
 \int_\RR
 e^{-ul-(a/\xi)e^{-l}}
 Q_M^{v}(l)\dd l. 
\label{eq:relation K q LG}
\ee
\end{lemma}

 \begin{proof}
Recall the definition of $Q_M^v(l)$
\eqref{eq:def of function H and Q(l)}.
Develop $\Psi_{(\alpha)^M}(\lambda/(l,0))$ using the branching rule
\eqref{eq:skew Whittaker branching rule} and consider the change of variables
$
\lambda^i=(l+L_1(i),L_2(i))$ for $0\leq i\leq M.$
We have
\be 
Q_M^v(l)= 
\left(\frac{\Gamma(\alpha+u)\Gamma(\alpha+v)}
{\Gamma(\alpha)^2}
\right)^M e^{-vl} \EE_{\LG,M}^{\alpha+v,\alpha+u}
\left[
e^{(u+v)L_2(M)
-e^{-l}\mathscr{A}_{M}(L_2,L_1)}
\right]. 
\ee
Hence  by Tonelli's theorem,
\begin{align*}
&\int_\RR
e^{-ul-(a/\xi)e^{-l}}Q_M^v(l)\dd l\\
&=
\left(
\frac{\Gamma(\alpha+u)\Gamma(\alpha+v)}
{\Gamma(\alpha)^2}
\right)^M
\EE_{\LG,M}^{\alpha+v,\alpha+u}
\left[
e^{(u+v)L_2(M)}
\int_\RR
e^{-(u+v)l-e^{-l}
(a/\xi+\mathscr{A}_{M}(L_2,L_1))}
\dd l
\right]\\
&=
\Gamma(u+v)\xi^{u+v}
\left(
\frac{\Gamma(\alpha+u)\Gamma(\alpha+v)}
{\Gamma(\alpha)^2}
\right)^M
\EE_{\LG,M}^{\alpha+v,\alpha+u}
\left[
e^{(u+v)L_2(M)}
\big( a+\xi\mathscr{A}_{M}(L_2,L_1)\big)^{-(u+v)}
\right]\\
&=
\Gamma(u+v)\xi^{u+v}
\left(
\frac{\Gamma(\alpha+u)\Gamma(\alpha+v)}
{\Gamma(\alpha)^2}
\right)^M
\mathscr{K}_{M}^{\alpha,u,v}( a,\xi),
\end{align*}
where we integrate out $l$ in the second equality using the identity \eqref{eq:gamma function integral representation} 
with $\ta=u+v$ and $S=  a/\xi+\mathscr{A}_{M}(L_2,L_1)$. 
\end{proof} 

To arrive at the asymptotics analysis of the   partition function, our  starting point is the following integral representation  for the function $Q_M^{v}(y)$ developed in \cite{Barraquand}.

\begin{proposition}\cite[Proposition 3.11]{Barraquand}\label{prop:prop:barraquands integral formula for Q} 
Let  $\alpha>0$, $M\geq1$, and $l\in\RR$. When $v>0$, there holds the integral representation:
\be
 Q_M^{v}(l)
 =\int_{i\RR}\frac{\dd z}{2\pi i}\,
 \mathcal F_{M,l}^{v}(z), 
 \label{eq:integral representation of Q_M(l)}
\ee
where 
\be\label{eq:F^v_M(z) LG}
\mathcal F_{M,l}^{v}(z)
 =\Psi_{(z,-z)}(l,0)\frac{\Gamma(v\pm z) }
 {2\Gamma(\pm 2z) }\left(\frac{\Gamma(\alpha\pm z) }{\Gamma(\alpha)^2}\right)^M.
\ee
When $v\leq 0$, $Q_M^{v}(l)$ is given by the analytic continuation of \eqref{eq:integral representation of Q_M(l)}.
\end{proposition}

Let us specify the analytic continuation of function $Q_M^{v}(l)$. 

\subsubsection{Analytic continuation of   $Q_M^v(y)$}
Compared with  the GLPP case in Section \ref{subsubsec:analytic continuation of q}, the arguments in this section may appear more technical due to the existence of multiple poles, which requires repeated uses of the the identity theorem for analytic functions. We emphasize that the proof strategies for both models remain completely in parallel.

Recall the definition of function $Q_M^v(y)$ in
\eqref{eq:def of function H and Q(l)} and the function $\mathcal F_{M,l}^{v}(z)$ \eqref{eq:F^v_M(z) LG}. For $\a>0$ and  integers $0\leq k<\a$, let us denote
\be
R_{M,l}^v(k)
=\frac{v+k}
 {k! \ \Gamma(1-2v-k)}
 \Psi_{(v+k,-v-k)}(l,0)
 \left(\frac{\Gamma(\alpha\pm(v+k))}{\Gamma(\alpha)^2}\right)^M. 
\ee 

\begin{lemma}\label{lem:analytic continuation of Q_M(l)}
For $\alpha>0$, $M\geq1$ and $l\in\RR$. Let $n\geq0$ be an integer
with $n<\alpha$, and fix  constant $c$ such that 
\be\label{eq:condition on c}
0<c<\min\{1/4,(\alpha-n)/2\}.
\ee
Assume that for all $v\in (-n,-n+\frac{c}{2})$, we have 
\be\label{eq:Q before translation}
Q_M^v(l)=\int_{i\RR}\frac{\d z}{2\pi i}\mathcal F_{M,l}^{v}(z)
-2\sum_{k=0}^{n-1}R_{M,l}^v(k).
\ee
Then, for all $v\in (-n-\frac{c}{2},-n+\frac{c}{2})$, we have 
\be\label{eq:local analytic continuation of Q}
Q_M^v(l)=\int_{c+i\RR}\frac{\d z}{2\pi i}\mathcal F_{M,l}^{v}(z)
-R_{M,l}^v(n)
-2\sum_{k=0}^{n-1}R_{M,l}^v(k).
\ee
\end{lemma}
\begin{proof}
 Recall the contour integral representation of the function $ Q_M^{v}(l)$ \eqref{eq:integral representation of Q_M(l)}. Since $1/\Gamma$ and $\Psi$ are entire on $\C$, for $|\mathfrak R z|<\a$, the only possible poles 
 $ z= \pm(k+v), \ k\in\ZZ_{\geq 0}$
 come from the factors $\Gamma(v\pm z)$. By the choice of $c$ \eqref{eq:condition on c}, for $v\in(-n,-n+\frac{c}{2})$ and $0\leq k< n$, we have
$$
\mathfrak R (v+n)\in (0,c),\quad \mathfrak R(-v-n)<0,\quad 
\mathfrak R (v+k)<-1+\frac{c}{2}<0, \quad \mathfrak R (-v-k)>1-\frac{c}{2}>c. 
$$ 
Recall the hypothesis \eqref{eq:Q before translation},  \eqref{eq:local analytic continuation of Q} holds on $(-n,-n+\frac{c}{2})$ by   residue theorem with the fact that 
$ R_{M,l}^v(k)=\Res_{z=v+k}\mathcal F_{M,l}^{v}(z)
$,  
see Figure \ref{fig:contour shift LG} for illustration. 

Since $Q_M^v(l)$ is analytic on $(-\a,\infty)$ by Proposition \ref{prop:prop:barraquands integral formula for Q}, it remains to  justify the two terms on the right hand side of \eqref{eq:local analytic continuation of Q} are analytic for $v\in (-n- \frac{c}{2},-n+\frac{c}{2})$. 
The residue term is analytic since there is no pole coming from the Gamma function on the numerator:  for $0\leq k\leq n$, we have 
$|\mathfrak R (v+k)|\leq n-k+|\mathfrak R (v+n)|< n+\frac{c}{2}<\a$.
To argue the real analicity of function  $v\mapsto  \int_{c+i\RR}\frac{\d z}{2\pi i}\mathcal F_{M,l}^{v}(z)$ 
for   $v\in (-n-\frac{c}{2},-n+\frac{c}{2})$, it is convenient to regard it as a comlex function and  prove it is holomorphic in the open strip  
$$
\mathbb S_{n,c}
=
\left\{
v\in\mathbb C: \mathfrak R v \in  \left(-n-\frac{c}{2},-n+\frac{c}{2}\right)
\right\}\subset\C.
$$
Observe that   for each fixed $z\in c+i\RR$ , the function $v\mapsto  \int_{c+i\RR}\frac{\d z}{2\pi i}\mathcal F_{M,l}^{v}(z)$ is holomorphic in $v\in\mathbb S_{n,c}$: indeed,  for all $z\in c+i\RR$, by the choice of $c$ \eqref{eq:condition on c}, we have 
   $$\Re(v+z)\in\left(-n+\frac{c}{2},-n+\frac{3c}{2}\right),\quad \Re(v-z)\in\left(-n-\frac{3c}{2},-n-\frac{c}{2}\right).$$
   Neither of the intervals contains an integer. Hence, by holomorphic parameter dependence of contour integrals, see e.g.\cite[Lemma 1.1, Section XV]{Lang}, $v\mapsto  \int_{c+i\RR}\frac{\d z}{2\pi i}\mathcal F_{M,l}^{v}(z)$  is holomorphic on $\mathbb S_{n,c}$, and in particular real analytic on the real interval $ (-n-\frac{c}{2},-n+\frac{c}{2})$.  
\end{proof}

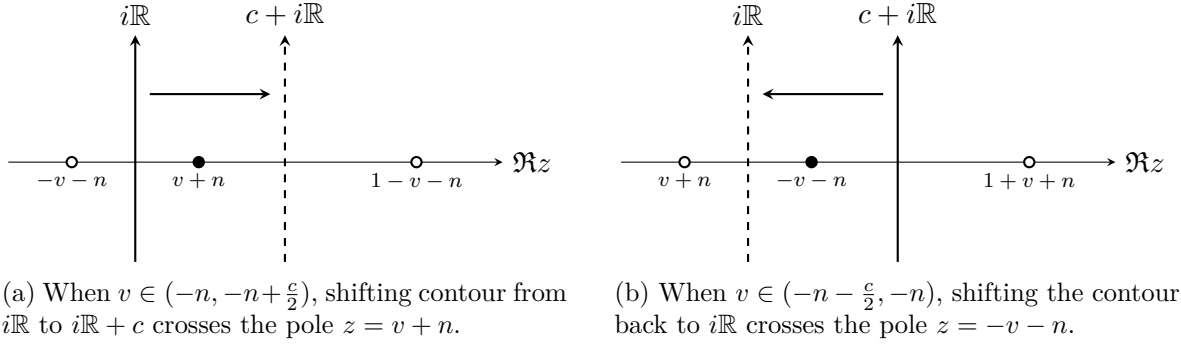
\begin{figure}
\centering
\begin{subfigure}[b]{0.48\textwidth}
    \centering
    \begin{tikzpicture}[>=stealth,scale=1.2]
       \draw[->] (-1.4,0)--(4.05,0) node[right] {$\mathfrak R z$};
\draw[->,thick] (0,-1.1)--(0,1.4);
\draw[->,thick,dashed] (1.65,-1.1)--(1.65,1.4);
\node[above] at (0,1.4) {$i\RR$};
\node[above] at (1.65,1.4) {$c+i\RR$};
\draw[thick,fill=white] (-0.7,0) circle (1.6pt) 
 node[below,font=\scriptsize]  {$-v-n$}; 
\draw[thick,fill=black] (0.7,0) circle (1.6pt)
node[below,font=\scriptsize] at (0.7,0) {$v+n$};
\draw[thick,fill=white] (3.1,0) circle (1.6pt) 
  node[below,font=\scriptsize]   {$1-v-n$};
\draw[->,thick] (0.16,0.75)--(1.49,0.75);
    \end{tikzpicture} 
    \caption{When $v\in(-n,-n+\frac{c}{2})$, shifting contour from $i\RR $ to $i\RR+c$ crosses the pole $z=v+n$.}
\label{fig:contour shift LG}
\end{subfigure}
\hfill
\begin{subfigure}[b]{0.48\textwidth}
    \centering
    \begin{tikzpicture}[>=stealth,scale=1.2]
        \draw[->] (-1.4,0)--(4.05,0) node[right] {$\mathfrak R z$};
\draw[->,thick,dashed] (0,-1.1)--(0,1.4);
\draw[->,thick] (1.65,-1.1)--(1.65,1.4);
\node[above] at (0,1.4) {$i\RR$};
\node[above] at (1.65,1.4) {$c+i\RR$};
\draw[thick,fill=white] (-0.7,0) circle (1.6pt) 
 node[below,font=\scriptsize] {$v+n$};
\draw[thick,fill=black] (0.7,0) circle (1.6pt)
node[below,font=\scriptsize] {$-v-n$};
\draw[thick,fill=white] (3.1,0) circle (1.6pt)  
node[below,font=\scriptsize]  {$1+v+n$};
\draw[->,thick] (1.49,0.75)--(0.16,0.75);
    \end{tikzpicture} 
    \caption{When $v\in(-n-\frac{c}{2},-n)$, shifting the contour  back to $i\RR$ crosses the pole $z=-v-n$.}
\label{fig:contour shift back LG}
\end{subfigure}
\caption{Poles crossed when shifting the  contour.} 
\end{figure}

\begin{proposition}
\label{prop:general formula for Q_M^s for all phases LG}
Fix $\alpha>0$, $M\geq1$, and $l\in\RR$. Let $v>-\alpha$. Then 
\be  \label{eq:Q_M^v for all v}
 Q_M^v(l)
 =\int_{i\RR}\frac{\dd z}{2\pi i}\,
 \mathcal F_{M,l}^{v}(z)
 +\sum_{k\geq0:\,v+k<0}
 \frac{-2(v+k)}
 {k!  \Gamma(1-2v-k)}
 \Psi_{(v+k,-v-k)}(l,0)
 \left(\frac{\Gamma(\alpha\pm(v+k))}{\Gamma(\alpha)^2}\right)^M.
\ee 
\end{proposition}
\begin{proof}
Let us prove by induction. For $n\geq 0$, let $H_n$ be the statement  
$$\eqref{eq:Q_M^v for all v}
    \text{ holds for all } v>\max(-\a,-n).
 $$
For $n= 0$, the sum of residues is empty and $H_0$ follows by Proposition  \ref{prop:prop:barraquands integral formula for Q}. 
    
Assume that $H_{n}$ holds. Let us prove that \eqref{eq:Q_M^v for all v} holds for $ \max(-\a,-n-1)<v\leq- n$. 
Let $c$ be chosen according to \eqref{eq:condition on c}. 
Then  \eqref{eq:local analytic continuation of Q} holds for 
$v\in (-n-\frac{c}{2},-n+\frac{c}{2})$ 
by induction hypothesis and Lemma \ref{lem:analytic continuation of Q_M(l)}, since 
 $ (-n,-n+\frac{c}{2})\subset (\max(-\a,-n),\infty)$.

When $v\in(-n-\frac{c}{2},-n)\subset(-n-\frac{c}{2},-n+\frac{c}{2})$, shifting the contour from $c+i\RR$ back to $i\RR$ crosses the pole at $z=-v-n$, see Figure \ref{fig:contour shift back LG} for illustration. Hence \eqref{eq:Q_M^v for all v} follows by the residue theorem, with the fact that $\Res_{z=-v-n}\mathcal F_{M,l}^{v}(z)=-R_{M,l}^v(n)$.  
    To extend \eqref{eq:Q_M^v for all v} to the desired interval $ (\max(-\a,-n-1),-n)$,  we again observe that both the integral term and residue term in \eqref{eq:Q_M^v for all v} are holomorphic on the same interval. 

It remains to prove  \eqref{eq:Q_M^v for all v} for     $v=-n$. By \eqref{eq:local analytic continuation of Q}, it suffices to prove that 
    $$
    \int_{c+i\RR}\frac{\d z}{2\pi i}\mathcal F_{M,l}^{ -n }(z)
-R_{M,l}^{-n}(n)
-2\sum_{k=0}^{n-1}R_{M,l}^{-n}(k)=\int_{i\RR}\frac{\dd z}{2\pi i}\,
 \mathcal F_{M,l}^{-n}(z)
-2\sum_{k=0}^{n}R_{M,l}^{-n}(k).
    $$
Recall the function $ \mathcal F_{M,l}^{v}(z)$ \eqref{eq:F^v_M(z) LG}, when $v=-n$ we have
$$
\frac{\Gamma(-n\pm z)}{2\Gamma(\pm2z)}
=\frac{2\cos(\pi z)}{\prod_{j=1}^n(j^2-z^2)}, 
$$
whose poles are $\{\pm1,\ldots, \pm n\}$. This implies that  $z\mapsto \mathcal F_{M,l}^{-n}(z)$ is holomorphic in a domain slightly larger than $\{ z:0\leq \mathfrak Rz\leq c \}$. Hence we can shift the contour back to $i\RR$ without changing the value of the integral. We conclude by noticing that $R_{M,l}^{-n}(n)=0$.   The induction is complete.   
\end{proof}

As a side result, we explain how the analytic continuation of the function $Q^v_M(l)$ can be used to   deduce that of   $\mathscr{K}_{M}^{\alpha,u,v}$ in the fan region.
\begin{corollary} 
\label{cor:analytic continuation for partition function LG}
For all parameters $\a,u,v$ satisfying  \eqref{eq:parameter domain LG} and $u+v>0$,  we have for all $M\geq 1$:
\begin{multline*}
    \mathscr{K}_{M}^{\alpha,u,v}=\frac{1}{\Gamma(u+v)\left(\Gamma(\a+u)\Gamma(\a+v)\right)^M}
    \Biggl(
\int_{i\RR}\frac{\dd z}{2\pi i}\frac{\Gamma(u\pm z)\Gamma(v\pm z)}{2\Gamma( \pm 2z)}\Gamma(\a\pm z)^{M}\\ +\mathbb{D}_M(\a,u,v)+\mathbb{D}_M(\a,v,u) 
    \Biggl), 
\end{multline*} 
where 
$$\mathbb{D}_M(\a,u,v)=\sum_{k\geq0:\,v+k<0}\frac{-2(v+k)}{k!  \Gamma(1-2v-k)}
 \Gamma(u\pm (v+k))\Gamma\left(\alpha\pm(v+k)\right)^M.$$
\end{corollary}
\begin{proof}
When    $u+v>0$ with $u>0$, $\mathbb{D}_M(\a,v,u)=0$. Corollary \ref{cor:analytic continuation for partition function LG} follows from Proposition \ref{prop:general formula for Q_M^s for all phases LG} by applying Lemma \ref{lem:Relation K Q} with $a=0,\xi=1$ and a Fubini’s Theorem, where we integrate out the $l$ variable using \eqref{eq:function N LG as integral of Whittaker}, which holds only when $u>0$.

When   $u+v>0$ but $u\leq 0$, we have $v>0$. Corollary \ref{cor:analytic continuation for partition function LG} follows from the time reversal  Lemma \ref{lem:time reversal partition function LG} and the previous case.
\end{proof}

\begin{remark}
\label{rmk:connection with BW analytic continuation}
\cite[Section 5.2]{BrycWesolowskiLG} performs an analytic continuation for the partition  function $\mathscr{K}_{M}^{\alpha,u,v}$ (up to a factor $\left(\Gamma(\a+u)\Gamma(\a+v)\right)^M$ due to different conventions used) for all $u,v$ satisfying \eqref{eq:parameter domain LG} using a quite different method. 
We instead perform an analytic continuation for the function $Q^v_M$, following the indications in \cite{Barraquand}. Corollary \ref{cor:analytic continuation for partition function LG} coincides with the result \cite[Proposition 4.6]{BrycWesolowskiLG} in the fan region. 
It would require additional work to arrive at the analytic continuation in the shock region, which we do not pursue as we do not use it. 
\end{remark}

\subsubsection{Asymptotics of $Q_M^v(l)$ and the partition function} 
Let us first deduce the asymptotics of  $Q_M^v(l)$ using representations in Proposition \ref{prop:general formula for Q_M^s for all phases LG}. 

Fix $\alpha>0$, $l\in\RR$, and $v>-\alpha$. For $M\geq1$ and
$x\in\RR$, define
\be
 \mathcal G_{M,l}^{v}(x)
 :=\frac1{2\pi\sqrt M}
 \Psi_{(ix/\sqrt M,-ix/\sqrt M)}(l,0)
 \frac{\Gamma(v\pm ix/\sqrt M)}
 {2\Gamma(\pm2ix/\sqrt M)}
 \left(
 \frac{\Gamma(\alpha\pm ix/\sqrt M)}{\Gamma(\alpha)^2}
 \right)^M. 
\ee 

\begin{lemma} 
Let $\a>0$. There exists constants $c=c(\alpha)>0$ and $C=C(\alpha)>0$ such that the  for all $M\geq 5$, $x,l\in\RR$ and $v>-\a$.
\be\label{eq:upper bound for function G v not 0}
 \left|M^{3/2}\mathcal G_{M,l}^{v}(x)\right|
 \leq C\Gamma(v)^2\Psi_{(0,0)}(l,0)
e^{-c|x|}
 \quad  v \notin\ZZ_{\leq0},
\ee
and 
\be\label{eq:upper bound for function G v=0}
 \left|M^{1/2}\mathcal G_{M,l}^{-n}(x)\right|
 \leq \frac{C}{(n!)^2}\Psi_{(0,0)}(l,0)
e^{-c|x|}
\quad 
 v=-n\in\ZZ_{\leq0}.
\ee
\end{lemma}

\begin{proof}
For  $\a>0, y\in\RR $, we have 
\begin{multline*}
    \log\Big(\frac{\Gamma(\a)^2}{\Gamma(\a\pm iy)}\Big)
    =\sum_{k\geq 0}\log \Big(1+\frac{y^2}{(\a+k)^2}\Big)\geq \sum_{k\geq 0}\frac{y^2}{(\a+k)^2+y^2}\geq \\
    \int_{\RR_+}\d s \frac{y^2}{(\a+s)^2+y^2}=|y|\arctan(\frac{|y|}{\a})\geq \max\left( \frac{\pi}{2}|y|-\a,  \frac{y^2}{\a+|y|}
 \right),
\end{multline*}
for the first equality, see e.g.   
\cite[\href{http://dlmf.nist.gov/5.8.E3}{(5.8.3)}]{NIST:DLMF}, and we use respectively $\log (1+u)\geq u/(1+u)$ and $u/(1+u)\leq \arctan u\leq u$  in the second and  last inequality. Using these two lower bounds, we have for  $y=\frac{x}{\sqrt{M}}\in \RR$
\be\label{eq:some bounds}
e^{2\pi|y|}
 \left(\frac{\Gamma(\alpha\pm iy)}{\Gamma(\alpha)^2}\right)^M
 \leq e^{4\alpha}e^{-(M-4)\frac{y^2}{\alpha+|y|}} 
 \leq e^{4\alpha}e^{- \frac{x^2}{5(\alpha+|x|)}} 
 \leq e^{21\alpha/5}e^{-|x|/5},
\ee
where we used $M-4\geq M/5$, $|y|\leq|x|$ in the second inequality and  
$\frac{x^2}{\alpha+|x|} \geq|x|-\alpha$ in the last inequality.

For all $x,y\in\RR $ and $ \la,\nu\in\RR^2$, we have 
$
|\Gamma(x+iy)|\leq |\Gamma(x)|, 
|\Psi_{\la+i\nu}(x,y)|\leq |\Psi_{\la}(x,y)|,
$
 and 
 $$ 
\frac{1}{2\Gamma(\pm 2iy)} =\frac{|y|\sinh(2\pi  |y|)}{\pi}  \leq 2y^2e^{2\pi|y|}, \quad  
\frac{\Gamma(-n\pm iy)}{2\Gamma(\pm2iy)}
  =\frac{2\cosh(\pi y)}{\prod_{j=1}^n(j^2+y^2)}
 \leq\frac{2e^{2\pi|y|}}{(n!)^2},
 $$
 These bounds together with inequality  \eqref{eq:some bounds} prove desired bounds \eqref{eq:upper bound for function G v not 0} and \eqref{eq:upper bound for function G v=0}. 
\end{proof}
  
\begin{proposition} 
Let $\alpha>0$. Then   
\begin{itemize}
\item For $v>0$ and    all $l\in\RR$,  there exists $C=C(\alpha,v)>0$ such that
\be
 |Q_M^{v}(l)|
 \leq\frac{C \Psi_{(0,0)}(l,0)}{M^{3/2}}, 
 \quad \forall M\geq 5.  
 \label{eq:upper bound of Q for v>0}
\ee
Moreover, as $M\to\infty$, 
\be
 Q_M^{v}(l)
 \sim\frac{\Gamma(v)^2\Psi_{(0,0)}(l,0)}
 {2\sqrt\pi(\psi_1(\alpha)M)^{3/2}}.
 \label{eq:asymptotics of Q^s_M(l) for v>0}
\ee

\item For $v=0$ and   $l\in\RR$, there exists $C=C(\alpha)>0$  such that, 
\be
 |Q_M^{0}(l)|
 \leq\frac{C \Psi_{(0,0)}(l,0)}{M^{1/2}},
  \quad \forall M\geq 5.
 \label{eq:upper bound of Q for v=0}
\ee
Moreover, as $M\to\infty$,  
\be \label{eq:asymptotics of Q^s_M(l) for v=0}
 Q_M^{0}(l)
 \sim\frac{\Psi_{(0,0)}(l,0)}{\sqrt{\pi\psi_1(\alpha)M}}.
\ee

\item For $-\alpha<v<0$, 
there exists constant $C =C (\alpha,v)>0$ such that for all $l\in\RR$,
\be\label{eq:upper bound of Q for v<0}
 |Q_M^{v}(l)|
 \leq C  \Psi_{(v,-v)}(l,0)
 \left(\frac{\Gamma(\alpha\pm v)}{\Gamma(\alpha)^2}\right)^M,
  \quad \forall M\geq 5. 
\ee
Moreover, as $M\to\infty$,   
\be \label{eq:asymptotics of Q^s_M(l) for v<0}
 Q_M^{v}(l)
 \sim\frac{\Psi_{(v,-v)}(l,0)}{\Gamma(-2v)}
 \left(\frac{\Gamma(\alpha\pm v)}{\Gamma(\alpha)^2}\right)^M.
\ee
\end{itemize}
\end{proposition}
\begin{proof}
    For $v>-\a$, using the change of variables $z=ix/\sqrt M$  we have 
\be
 \int_{i\RR}\frac{\dd z}{2\pi i}
 \mathcal F_{M,l}^{v}(z)  
 =\int_\RR\mathcal G_{M,l}^{v}(x)\dd x. 
\ee 

When $v>0$, the only contribution in \eqref{eq:Q_M^v for all v} is the integral term. Hence \eqref{eq:upper bound of Q for v>0} follows directly by inequality \eqref{eq:upper bound for function G v not 0}. Observe that when $v\notin\ZZ_{\leq 0}$, we have pointwise convergence for all $x\in\RR $
$$
M^{3/2}\mathcal G_{M,l}^{v}(x)\xrightarrow[M\to\infty]{}
\frac{\Gamma(v)^2}{\pi}\Psi_{(0,0)}(l,0)x^2e^{-\psi_1(\a)x^2}.
$$
Applying the dominated convergence theorem using bound \eqref{eq:upper bound for function G v not 0}, the asymptotics   \eqref{eq:asymptotics of Q^s_M(l) for v>0} follows by a Gaussian integral.

When $v=0$, the only contribution in \eqref{eq:Q_M^v for all v} is the integral term since the residue term is $0$. Hence \eqref{eq:upper bound of Q for v=0} follows directly by inequality \eqref{eq:upper bound for function G v=0}. Observe that when $v\in\ZZ_{\leq 0}$, we have pointwise convergence for all $x\in\RR $
$$
M^{1/2}\mathcal G_{M,l}^{-n}(x)\xrightarrow[M\to\infty]{}
\frac{1}{\pi(n!)^2}\Psi_{(0,0)}(l,0)e^{-\psi_1(\a)x^2}.
$$
Applying the dominated convergence theorem using the bound \eqref{eq:upper bound for function G v=0}, the asymptotics   \eqref{eq:asymptotics of Q^s_M(l) for v=0} follows by a Gaussian integral.

When $v<0$, the integral term is $O(M^{-3/2})$ or $O(M^{-1/2})$ using bounds \eqref{eq:upper bound for function G v not 0} or \eqref{eq:upper bound for function G v=0},   hence  the contribution in  \eqref{eq:Q_M^v for all v} is dominated by the largest residue term which grows exponentially due to the strict convexity of the function $\log\Gamma$ .  
\end{proof}
 
We are ready to deduce the asymptotics of
the partition function throughout the fan region.

\begin{proposition}[Fan region asymptotics] 
\label{prop:fan region asymptotics LG}
Assume \eqref{eq:parameter domain LG} and  $u+v>0$. 
Fix $\mathfrak a\geq0$ and $\xi>0$.
If $u,v>0$, then \eqref{eq:asymp of K LG MC all} holds.
If $u>0$ and $v=0$, then \eqref{eq:asymp of K LG u axis all} holds.
If $u=0$ and $v>0$, then \eqref{eq:asymp of K LG v axis all} holds.
If $v<0<u+v$, then \eqref{eq:asymp of K HD LG all} holds.
If $u<0<u+v$, then \eqref{eq:asymp of K LD LG all} holds.
\end{proposition}

\begin{proof}
Recall the relation between $\mathscr{K}_{M}^{\alpha,u,v}$ and
$Q_M^{v}$ in \eqref{eq:relation K q LG}.   When $u+v>0 $ and $u>0$, thanks to the upper bound \eqref{eq:upper bound of Q for v>0} \eqref{eq:upper bound of Q for v=0} and \eqref{eq:upper bound of Q for v<0}, by applying the dominated convergence theorem, the asymptocis of $\mathcal{K}_{M}^{a,r,s}(\alpha,\psi)$ from that of function  $Q_M^v(l)$ \eqref{eq:asymptotics of Q^s_M(l) for v>0}  \eqref{eq:asymptotics of Q^s_M(l) for v=0} and \eqref{eq:asymptotics of Q^s_M(l) for v<0}, where we use \eqref{eq:function N LG as integral of Whittaker} to evaluate the integral in $l$.

When $u+v>0$ and $u\leq0$,  using  the time reversal \eqref{eq:boundary parameter symmetry LG},  we have 
\begin{multline*}
    \mathscr K_M^{\alpha,u,v}(  a,\xi)
 =
 \frac{\xi^{-(u+v)}}{\Gamma(u+v)}
 \left(
 \frac{\Gamma(\alpha)^2}
 {\Gamma(\alpha+u)\Gamma(\alpha+v)}
 \right)^M \\ 
 \left(
 \int_\RR e^{-vl}Q_M^u(l)\dd l
 -
 \int_\RR e^{-ul}
 \left(1-e^{-(  a/\xi)e^{-l}}\right)
 Q_M^v(l)\dd l
 \right). 
\end{multline*}
Thanks to the upper bounds
\eqref{eq:upper bound of Q for v=0} and
\eqref{eq:upper bound of Q for v<0},  by applying the dominated convergence theorem, the asymptotics of the infinite sum follows from that of function $Q_M^v(l)$ \eqref{eq:asymptotics of Q^s_M(l) for v=0} and
\eqref{eq:asymptotics of Q^s_M(l) for v<0}.  The finite substraction is    $ O(M^{-3/2}) $  due to bound \eqref{eq:upper bound of Q for v>0}, hence negligeable compared to the  infinite sum, which is of order $M^{-1/2}$ for $u=0$ and exponentially large when $u<0$. 
\end{proof}

\subsubsection{Infinite two-layer Whittaker process} 
In this section, we show that  as $N\to\infty$, the two-layer Whittaker process  have a well defined limit, described by a family of simple stochastic processes 
$\Big(\PP^{u,v}_{\Whit,\infty}\Big)_{v\leq 0, u+v>0}$,
which we refer to as the infinite two-layer Whittaker processes. 
\begin{definition}[Infinite two-layer Whittaker processes]
Fix $\alpha>0$. For $u,v\in\RR$ such that $-\alpha<v\leq0$ and
$u+v>0$, define $\PP_{\Whit,\infty}^{\alpha,u,v}$ as the distribution of the
homogeneous Markov chain on $\RR^2$ with initial distribution and
transition kernel
\be\label{eq:ic and transition kernel of infinite two layer Whittaker process} 
\P^{u,v}(\dd\lambda)
 =e^{-u(\lambda_1-\lambda_2)}  
 \delta_0( \lambda_2)
H^{u,v}(\lambda_1-\lambda_2)\dd\lambda, \quad 
\U^{u,v}(\lambda,\dd\mu)
 =\frac{\Psi_{\a}(\mu/\lambda)}
{\Gamma(\alpha\pm v )}
\frac{H^{u,v}(\mu_1-\mu_2)}
{H^{u,v}(\lambda_1-\lambda_2)}\dd\mu, 
\ee
where 
$$
H^{u,v}(l)=\frac{\Psi_{(v,-v)}(l,0)}{\Gamma(u\pm v)}.
$$
The Cauchy and skew-Cauchy identities \eqref{eq:Whittaker Cauchy}  \eqref{eq:Whittaker skew Cauchy}
imply that $\P^{u,v}$ is a probability distribution and $\U^{u,v}$ is a
stochastic kernel.  
\end{definition}

\begin{proposition}\label{prop:two layer Whittaker converges to infinite two layer Whittaker}
Let $\a,u>0$. We have finite-dimensional distributions convergence  
For $v\geq 0$   
    \be
    \PP_{\Whit,N}^{\a,u,v}\xRightarrow[N\to\infty]{}
    \PP_{\Whit,\infty}^{\a,u,0}.
    \ee  
    For $v\in(-\min\{\a,u\},0)$,  
    \be
    \PP_{\Whit,N}^{\a,u,v}\xRightarrow[N\to\infty]{}
    \PP_{\Whit,\infty}^{\a,u,v}.
    \ee  
\end{proposition}
\begin{proof}
    Since $\PP_{\Whit,N}^{\alpha,u,v}$ \eqref{eq:initial distribution and transition kernel u,v>0} and $\PP_{\Whit,\infty}^{\alpha,u,v}$ \eqref{eq:ic and transition kernel of infinite two layer Whittaker process} are distributions of Markov chains, it suffices to prove that the initial distribution $\P^{u,v}_{N} $ and the  transition kernel $\U_{x,N}^{u,v} $ converges to the corresponding limit for each  $x\geq0$.

When $v>0$ or $v=0$, the asymptotics of $Q_N^v(l)$ are given by
\eqref{eq:asymptotics of Q^s_M(l) for v>0} and
\eqref{eq:asymptotics of Q^s_M(l) for v=0}. Recall the definition of function  $H_{x,N}^{u,v}(l)$ in
\eqref{eq:def of function H and Q(l)}. In both cases  we have  
\be
H_{x,N}^{u,v}(l)
=\frac{Q_{N-x}^v(l)}
{\displaystyle\int_\RR e^{-u\ell}Q_N^v(\ell)\dd\ell}
\xrightarrow[N\to\infty]{}
\frac{\Psi_{(0,0)}(l,0)}{\Gamma(u)^2}=H^{u,0}(l),
\ee
As a consequence,   the initial and transition densities in
\eqref{eq:initial distribution and transition kernel u,v>0} converge to those of
$\P^{u,0}$ and $\U^0$  as $N\to\infty$. For all $\la,\mu\in\RR^2$ and all $x\geq 0$,
 
$$
\P^{u,v}_{N}(\la) 
\xrightarrow[N\to\infty]{} 
\P^{u,0}(\la),\qquad
\U_{x,N}^{u,v}(\la,\mu)
 \xrightarrow[N\to\infty]{} 
\U^{u,0}(\la,\mu),
$$
as desired.

When $v\in(-\min\{\a,u\},0)$,  the asymptotics of $Q_M^v(l)$
are given by
\eqref{eq:asymptotics of Q^s_M(l) for v<0}.  We have  
$$
H_{x,N}^{u,v}(l)
=
 \frac{Q_{N-x}^v(l)}
 {\displaystyle\int_\RR e^{-u\ell}Q_N^v(\ell)\dd\ell}
 \xrightarrow[N\to\infty]{}
 H^{u,v}(l)
 \left(
 \frac{\Gamma(\alpha\pm v) }
 {\Gamma(\alpha)^2}
 \right)^{-x}.
$$
As a consequence,  for all $\la,\mu\in\RR^2$ and all $x\geq 0$,
$$
\P^{u,v}_{N}(\la) 
\xrightarrow[N\to\infty]{} 
\P^{u,v}(\la),\qquad
\U_{x,N}^{u,v}(\la,\mu)
 \xrightarrow[N\to\infty]{} 
\U^{u,v}(\la,\mu),
$$
as desired.
\end{proof}

Recall the relation between the infinite two-layer Schur process and the Hariya--Yor process in Corollary \ref{cor:HY equals the increments of infinite two layer Schur process}, we have an analogous 
relation between the infinite two-layer Whittaker process and corresponding Hariya--Yor process. We omit its proof as it is completely parallel to the arguments in Section \ref{subsubsec:two layer Schur and HY relation} and all  proof ingredients are known in  the literature.

\begin{lemma}\cite{BarraquandCorwin}\label{lem:HY equals the increments of infinite two layer Whittaker process}
    When $u+v>0$. Let $(\la^x)_{x\geq0} \sim \PP^{\a,u,v}_{\Whit,\infty}$, we have 
    $$
\big(\la^x-\la^0\big)_{x\geq0}=\Big(\mathbb{HY}^{\a,u,v}_{\LG}(k)\Big)_{k\geq 0}.
    $$
\end{lemma}
\begin{proof}
    Follow the exactly parallel argument in Section \ref{subsubsec:two layer Schur and HY relation}.
\end{proof}

\subsection{Proof of the theorems}\label{subsec:proof of the theorems LG}

\begin{proof}[Proof of Proposition \ref{prop:Girsanov LG}]
    Assume first that $v<0$ and $u>v$.
Let $Y_*\sim\Ga(-2v)$ be independent of $(A,S)$. We have
$$
\begin{aligned}
\EE\left[F\left(\mathscr{T}^{1/Y}(S,A)|_{\llbracket 0,N\rrbracket}\right)\right]
&=\frac{\Gamma(-2v)}{\Gamma(u-v)}
\EE\left[F\left(\mathscr{T}^{Y_*}(S,A)|_{\llbracket 0,N\rrbracket}\right)
Y_*^{u+v}\right]\\
&=\frac{\Gamma(-2v)}{\Gamma(u-v)}
\EE\left[F\left((A,S)|_{\llbracket 0,N\rrbracket}\right)
\mathscr{A}_{\infty}(S,A)^{-(u+v)}\right],
\end{aligned}
$$
where we change the distribution of the parameter of the $T$ transform from
that of $1/Y$ to that of  $Y_*$ in the first equality and use the Dufresne's
relation \eqref{eq:joint dufresne LG} in the second equality. 

Using \eqref{eq:cutting LG}, we obtain
$$
\mathscr{A}_{\infty}(S,A)
=\mathscr{A}_{N}(S,A)+e^{S(N)-A(N)}
\mathscr{A}_{\infty}(S^{(N)},A^{(N)}).
$$
By Markovianity of the random walks, $(A^{(N)},S^{(N)})$ is
independent of $\mathcal{F}_N=\sigma(A(i),S(i):0\leq i\leq N)$
and has the same distribution as $(A,S)$. By the Dufresne's relation
\eqref{eq:LG Dufresne},
$$
\mathscr{A}_{\infty}(S^{(N)},A^{(N)}) \sim\Gav(-2v).
$$
Recall the definition of function $\mathscr{N}$
\eqref{eq:function N for two paths LG}. We have
$$ 
\frac{\Gamma(-2v)}{\Gamma(u-v)}
\EE\left[\mathscr{A}_{\infty}(S,A)^{-(u+v)}
\mathrel{\big|}\mathcal F_N\right] 
=\mathscr{N}_N^{u,v}(A,S), 
$$
which gives \eqref{eq:LG Girsanov}.

When $u+v>0$, including in particular $v=0$, we have 
\begin{multline*}
    \EE 
\left[
 F\left(\mathscr{T}^Y(S,A)|_{\llbracket 0, N \rrbracket}\right)
 \right] 
\\
=\int_{(\RR^2)^{N+1}}
F\left((\lambda^k-\lambda^0)_{0\leq k\leq N}\right)
e^{-u(\lambda^0_1-\lambda^0_2)}  
 \delta_0( \lambda^0_2)
H^{u,v}(\lambda^0_1-\lambda^0_2)\dd\lambda^0\\
\prod_{i=1}^{n}
\frac{\Psi_{\a}(\lambda^{i}/\lambda^{i-1})}
{\Gamma(\alpha\pm v )}
\frac{H^{u,v}(\lambda^{i}_1-\lambda^{i}_2)}
{H^{u,v}(\lambda^{i-1}_1-\lambda^{i-1}_2)}\dd\lambda^{i}
 \\ 
=\int_{A,S\in\RR^{N}}
F((A,S))
\PP_{\LG,N}^{\a+v}(A) \PP_{\LG,N}^{\a-v}(S) e^{-v(S(N)-A(N))}  \\
\int_{l\in\RR}
e^{-ul-e^{-l\mathscr{A}_N(S,A)}}
H^{r,s}(l+A(N)-S(N)) \\ 
=\EE \left[
 F(A,S)\mathscr{N}_{N}^{u,v}(A,S)
 \right].
\end{multline*}
The the first equality follows from Lemma 
\ref{lem:HY equals the increments of infinite two layer Whittaker process} and the definition of infinite two-layer Schur process    $\PP^{u,v}_{\Whit,\infty}$ \eqref{eq:ic and transition kernel of infinite two layer Whittaker process}. 
In the second equality, we expand Schur function and use the change of variables $\la^i_1=A(i)+l,\la^i_2=S(i)$. We evaluate the last integral using  \eqref{eq:function N LG as integral of Whittaker}.  
\end{proof}

\begin{proof}[Proof of Theorem \ref{thm:LG main thm}]
The high density case \eqref{eq:asymp of K HD LG all} and the low density case \eqref{eq:asymp of K LD LG all} are proven in Proposition \ref{prop:HD LG} and \ref{prop:LD LG}.

The coexistence line case \eqref{eq:asymp of K CE LG all} is proven in Proposition \ref{prop:CE LG}.

The asympotics in the fan region is proven in Proposition \ref{prop:fan region asymptotics LG}, which contains in particular the maximal current case \eqref{eq:asymp of K LG MC all} and the two axis cases \eqref{eq:asymp of K LG u axis all} and \eqref{eq:asymp of K LG v axis all}. 
 
Finally, when $u+v=0$, which includes the case $u=v=0$, we have 
  $\mathscr K_M=1$   by definition \eqref{eq:generalized partition function LG}.  
\end{proof}

\begin{proof}[Proof of Corollary \ref{cor:HY theorem LG}]

Let us first prove the convergence for   $u>0, u+v>0$, in which case $v_{\star}=\min\{0,v\}$.   
For all $k\geq 0$ and all $M\geq k$, let 
$$\big(L_1^x(M),L_2^x(M)\big)_{0\leq x\leq M} \sim\PP_{\LG,M}^{\alpha,u,v}, \quad   (\la^x_1(M),\la^x_2(M) \big)_{0\leq x\leq M}\sim \PP_{\Whit,M}^{\alpha,u,v} \quad  (\la^x)_{x\geq0}\sim \PP_{\Whit,\infty}^{\alpha,u,v}.$$ 
Then we have 
\begin{multline*}
 \big(L_1^x(M),L_2^x(M)\big)_{0\leq x\leq k}\overset{d}{=}\big(\la^x_1(M)-\la^0_1(M),\la^x_2(M)-\la^0_2(M)\big)_{0\leq x\leq k}\\
    \xRightarrow[M\to\infty]{}
    \big(\la^x_1 -\la^0_1 ,\la^x_2 -\la^0_2 \big)_{0\leq x\leq k}
    \overset{d}{=}
    \Big(\mathbb{HY}^{r,s}_{\LG}(x)\Big)_{0\leq x\leq k}. 
\end{multline*} 
The first equality in distribution  follows from the relation \eqref{eq:relation between reweighted rw and two layer Whittaker process}, the second equality in distribution follows from Lemma \ref{lem:HY equals the increments of infinite two layer Whittaker process} and the convergence follows from  Proposition \ref{prop:two layer Whittaker converges to infinite two layer Whittaker}. 

Let us then  prove the convergence for the cases $u\leq \min\{0,v\} $ and $v<\min\{u,0\}$. We have  
\begin{multline*}
     \PP_{\LG,M}^{\a,u,v}
\Big((L_1,L_2)|_{[0,k]}=(\ell_1,\ell_2)\Big) 
 =\PP_{\LG,k}^{\a+v}(\ell_1)\PP_{\LG,k}^{\a+u}(\ell_2)
 \frac{\mathscr{K}_{M-k}^{\a,u,v}
   \big(e^{-(L_1\otimes L_2)(k)},e^{-L_1(k)}\big)}
  {\mathscr{K}_{M}^{\a,u,v} }\\
  \xrightarrow[]{M\to\infty}\  \PP_{\LG,k}^{\a+v}(\ell_1)\PP_{\LG,k}^{\a+u}(\ell_2)\varrho^{-k}
  \begin{cases}
      e^{(u+v)\ell_1(k)},   &u\leq \min\{0,v\},\\
      \mathcal{N}^{r,s}\left(e^{-(L_1\otimes L_2)(k)},e^{-L_1(k)}\right), &v<\min\{u,0\},
  \end{cases}\\ 
  =\begin{cases}
      \mathbb{HY}_{\LG}^{a,r,r} \quad  &u\leq \min\{0,v\},\\
      \mathbb{HY}_{\LG}^{a,r,s} &v<\min\{u,0\}. 
  \end{cases}
\end{multline*} 
The first equality is Proposition \ref{prop:decomposition LG}. 
The second convergences follows from asymptotics  \eqref{eq:asymp of K LG v axis all}\eqref{eq:asymp of K LG u axis all}\eqref{eq:asymp of K CE LG all} and \eqref{eq:asymp of K HD LG all} in Theorem \ref{thm:LG main thm} respectively, together with the fact that 
$\mathscr{N}^{u,v}(0,1)=\mathscr{N}^{u,v}(0,1)=1$ and 
the evaluation in \eqref{eq:special case function N LG}. The marginal distribution of the Hariya Yor process is given by Proposition \ref{prop:Girsanov LG}. And  for the first case in the third equality, we change the reference measure of $L_1$ from $\PP_{\LG,k}^{\a+v}$ to $\PP_{\LG,k}^{\a-v}$, and use the fact that  
$\mathbb{HY}_{\LG}^{\a,v,v}=\PP_{\LG,k}^{\a-v}\times \PP_{\LG,k}^{\a+v}$. For the second case in the third equality, we change the reference measure of $L_2$ from $\PP_{\LG,k}^{\a+u}$ to $\PP_{\LG,k}^{\a-v}$, and use the fact that 
$$
\mathcal{N}^{r,s}
\left(e^{-(L_1\otimes L_2)(k)},e^{-L_1(k)}\right)= e^{(u+v)\ell_2(k)} 
\mathcal N^{r,s}\Bigl(\mathcal A_k(l_2,l_1),e^{l_2(k)-l_1(k)}\Bigr).
$$
\end{proof}

\renewcommand{\emph}[1]{\textit{#1}}
\bibliography{hariyayor}
\bibliographystyle{goodbibtexstyle.bst}
\end{document}